\documentclass[11pt,reqno]{amsart}
\usepackage{amsmath,amsthm,amssymb,bm}
\usepackage[colorlinks=true,pdfstartview=FitV,linkcolor=blue,citecolor=blue,urlcolor=blue]{hyperref}
\usepackage[all]{xy}
\usepackage{tikz}
\usetikzlibrary{cd}
\usepackage{comment}

\makeatletter

\@addtoreset{equation}{subsection}
\makeatother

\newtheorem{Thm}{Theorem}[subsection]
\newtheorem{Lem}[Thm]{Lemma}

\newtheorem{Prop}[Thm]{Proposition}

\newtheorem{Def}[Thm]{Definition}

\newtheorem*{Thm*}{Theorem 1}
\newtheorem{LemA}{Lemma}[section]
\newtheorem{PropA}[LemA]{Proposition}

\newcommand{\Z}{\mathbb{Z}}
\newcommand{\Q}{\mathbb{Q}}

\newcommand{\Hom}{\mathrm{Hom} \,}

\newcommand{\id}{\mathrm{id}}

\newcommand{\cl}{\mathrm{cl}}

\newcommand{\fg}{{\mathfrak g}}

\newcommand{\fn}{{\mathfrak n}}

\newcommand{\fS}{{\mathfrak S}}

\newcommand{\ga}{\alpha}
\newcommand{\gb}{\beta}

\newcommand{\gl}{\lambda}
\newcommand{\gL}{\Lambda}
\newcommand{\gd}{\delta}
\newcommand{\gD}{\Delta}

\renewcommand{\ggg}{\gamma}
\newcommand{\gs}{\sigma}

\newcommand{\gee}{\varepsilon}

\newcommand{\ol}{\overline}

\newcommand{\wti}{\widetilde}

\newcommand{\wt}{\mathrm{wt}}

\newcommand{\bB}{\mathbf{B}}

\newcommand{\normsq}[1]{\lVert #1 \rVert^2}

\renewcommand{\(}{(\!(}
\renewcommand{\)}{)\!)}

\title[Existence of KR crystals for near adjoint nodes]
{Existence of Kirillov--Reshetikhin crystals\\ for near adjoint nodes in exceptional types}

\author[K.~Naoi]{Katsuyuki Naoi}
\address[K.~Naoi]{%
Institute of Engineering \\
Tokyo University of Agriculture and Technology\\
2-24-16 Naka-cho, Koganei-shi, Tokyo 184-8588, JAPAN}
\email{naoik@cc.tuat.ac.jp}

\author[T.~Scrimshaw]{Travis Scrimshaw}
\address[T. Scrimshaw]{
School of Mathematics and Physics \\
The University of Queensland \\
St.\ Lucia, QLD 4072, AUSTRALIA}
\email{tcscrims@gmail.com}
\urladdr{https://people.smp.uq.edu.au/TravisScrimshaw/}

\keywords{affine quantum group, Kirillov--Reshetikhin crystal, crystal pseudobase}
\subjclass[2010]{17B37, 17B65, 17B10}

\begin{document}

\begin{abstract}
 We prove that, in types $E_{6,7,8}^{(1)}$, $F_4^{(1)}$ and $E_6^{(2)}$, every Kirillov--Reshetikhin module associated with the node adjacent to
 the adjoint one (near adjoint node) has a crystal pseudobase,
 by applying the criterion introduced by Kang \textit{et al.}
 In order to apply the criterion, we need to prove some statements concerning values of a bilinear form.
 We achieve this by using the global bases of extremal weight modules.
\end{abstract}

\maketitle

\section{Introduction}

 Let $\fg$ be an affine Kac--Moody Lie algebra, and denote by $U_q'(\fg)$ the associated quantum affine algebra without the degree operator.
 Kirillov--Reshetikhin (KR for short) modules are a distinguished family of finite-dimensional simple $U_q'(\fg)$-modules (see, for example,~\cite{CPbook}).
 In this article KR modules are denoted by $W^{r,\ell}$, where $r$ is a node of the Dynkin diagram of $\fg$ except 
 the node $0$ prescribed in~\cite{MR1104219} and $\ell$ is a positive integer.
 KR modules are known to have several good properties, such as their $q$-characters satisfy the $T$ ($Q$, $Y$)-system relations, fermionic formulas for their graded characters, and so on (see~\cite{MR1745263,MR1993360,MR2254805,Her10},
 for example, and references therein).

 Another important (conjectural) property of a KR module is the existence of a crystal base in the sense of Kashiwara, which was presented in~\cite{MR1745263,MR1903978}.
 In this article, we mainly consider a slightly weaker version of the conjecture, the existence of a crystal pseudobase (crystal base modulo signs, see 
 Subsection~\ref{Subsection:Crystal (pseudo)bases and global bases}).

 If a given KR module $W^{r,\ell}$ is multiplicity free as a $U_q(\fg_0)$-module, it is known to have a crystal pseudobase,
 where $\fg_0$ is the subalgebra of $\fg$ whose Dynkin diagram is obtained from that of $\fg$ by removing $0$.
 In nonexceptional types, in which all $W^{r,\ell}$ are multiplicity free, this was shown by Okado and Schilling~\cite{MR2403558}.
 Recently this was also proved for all multiplicity free $W^{r,\ell}$ of exceptional types by Biswal and the second author~\cite{biswal2019existence}
 in a similar fashion.

On the other hand, if $W^{r,\ell}$ is not multiplicity free, then the conjecture has been solved in only a few cases so far.
 Kashiwara showed for all affine types that all fundamental modules $W^{r,1}$ have crystal bases~\cite{MR1890649}, and in types $G_2^{(1)}$ and $D_4^{(3)}$, the first author  verified the existence of a crystal pseudobase for all $W^{r,\ell}$~\cite{naoi2018existence}.

 We say a node $r$ is \textit{near adjoint} if the distance from $0$ is precisely $2$.
 The goal of this paper is to show the conjecture for all KR modules associated with near adjoint nodes in exceptional types.
 This has already been done in~\cite{naoi2018existence} for types $G_2^{(1)}$ and $D_4^{(3)}$,
 and our main theorem below covers all remaining types.
 

 \begin{Thm*}\label{Thm:Main}
  Assume that $\fg$ is either of type $E_n^{(1)}$ $(n=6,7,8)$, $F_4^{(1)}$, or $E_6^{(2)}$, and $r$ is the near adjoint node.
  Then for every $\ell \in \Z_{>0}$, the KR module $W^{r,\ell}$ has a crystal pseudobase.
 \end{Thm*}

 In particular, since a KR module $W^{r,\ell}$ in type $E_6^{(1)}$ is multiplicity free if $r$ is not the near adjoint node,
Theorem~\ref{Thm:Main} solves the conjecture for all KR modules of this type.

 As with previous works~\cite{MR2403558,naoi2018existence,biswal2019existence}, Theorem~\ref{Thm:Main} is proved 
 by applying the criterion for the existence of a crystal pseudobase introduced in~\cite{MR1194953}.
 In our cases, however, this is much more involved and we need a new idea, which we will explain below.

 By the criterion, the existence of a crystal pseudobase is reduced to showing that certain vectors
 are almost orthonormal with respect to a prepolarization (bilinear form having some properties) and satisfy additional conditions concerning the values of the prepolarization.
 In the previous works these statements were proved by directly calculating the values of the prepolarization (although in~\cite{naoi2018existence} 
 the amount of calculations was reduced using an induction argument on $\ell$).
 However, this appears to be quite difficult to do in our cases.
 Hence we apply a more sophisticated method using the global basis of an extremal weight module introduced by Kashiwara~\cite{MR1262212}.
 For example, it is previously known that a global basis is almost orthonormal~\cite{MR2074599}, and therefore 
 the required almost orthonormality of given vectors is deduced by connecting them with a global basis.
 The other conditions are also proved in a similar spirit.

 Besides the KR modules treated in this paper, 
 there are several families of $W^{r,\ell}$ for which the existence of crystal pseudobases remain open:
 $r=3,5$ in type $E_7^{(1)}$, $3\leq r \leq 7$ in type $E_8^{(1)}$, and $r=3$ in types $F_4^{(1)}$ and $E_6^{(2)}$, 
 where the labeling of nodes are given in Figure~\ref{figure} in Subsection~\ref{Subsection:Three}. 
 We hope to study these in our future work.

 The paper is organized as follows. 
 In Section~\ref{Section:pre}, we recall the basic notions needed in the proof of the main theorem.
 In Subsection~\ref{Subsection:Three}, we reduce the main theorem to three statements (C1)--(C3),
 and these are proved in Subsections~\ref{Subsection:C1}--\ref{Subsection:C3}.
 In Subsection~\ref{Subsection:C3}, we use a certain relation~\eqref{eq:the_relation} in $W^{r,\ell}$,
 whose proof is postponed to Appendix~\ref{Appendix} since, while straightforward, it is slightly lengthy and technical.

\section*{Acknowledgments}

The authors would like to thank Rekha Biswal for helpful discussions.
This work benefited from computations using \textsc{SageMath}~\cite{sage}.
The first author was supported by JSPS Grant-in-Aid for Young Scientists (B) No.\ 16K17563.
The second author was partially supported by the Australian Research Council DP170102648. \\

\section*{Index of notation}

We provide for the reader's convenience a brief index of the notation which is used repeatedly in this paper:\\
 
\noindent Subsection~\ref{Subsection:QAA}: $\fg$, $I$, $C=(c_{ij})_{i,j \in I}$, $\ga_i$, $h_i$, $\gL_i$, $\gd$, $P$, $P^+$, $Q$, $Q^+$, $W$, $s_i$, $I_0$,
$\varpi_i$, $P^*$, $d$,\\ 
\phantom{Subsection} $P_\cl$, $q_i$, $D$, $q_s$, $U_q(\fg)$, $e_i$, $f_i$, $q^h$, $U_q'(\fg)$, $U_q(\fn_{\pm})$, $e_i^{(n)}$, $\wt_P$, $U_q(\fg_J)$, $t_i$, $\gD$,
$\wt_{P_\cl}$.\\
\noindent Subsection~\ref{Subsection:Crystal (pseudo)bases and global bases}: $\tilde{e}_i$, $\tilde{f}_i$, $A$, $\ol{ \phantom{a} }$.\\
\noindent Subsection~\ref{polarization}: $\normsq{u}$.\\
\noindent Subsection~\ref{Subsection:extremal weight modules}: $V(\gL)$, $v_\gL$, $L(\gL)$, $B(\gL)$, $\bB(\gL)$, $V(\gL)_\Z$, $\bB(\gL^1,-\gL^2)$.\\
\noindent Subsection~\ref{Subsection:fusion}: $M_a$, $\iota_a$, $\iota$, $W^{r,\ell}$, $z_r$, $L(W^{r,1})$, $w_\ell$, $\iota_k$.\\
\noindent Subsection~\ref{Subsection:Criterion}: $\fg_0$, $P_0$, $P_0^+$, $V_0(\gl)$.\\
\noindent Subsection~\ref{Subsection:Three}: $W^\ell$, $I_{01}$, $J$, $R$, $R^+$, $R_L^+$, $R_1$, $\theta_1$, $\theta_J$, $e_{\bm{r}}^{(p)}$, $E_{\bm{r}}^{(p)}$, $c_\fg$, $\bm{i}$, $\bm{j}$, $\bm{i}[k_2,k_1]$, $\bm{j}[k_2,k_1]$,\\
\phantom{Subsection 2.1} $s_{\bm{r}}$, $\gL_i^\vee$, $E^{\bm{p}}$ for $\bm{p} \in \Z^{6}$, $\wt$, $S_\ell$.\\
\noindent Subsection~\ref{Subsection:C2}: $\bm{\gee}_i$, $E^{\bm{p}}$ for $\bm{p} \in \Z^5$, $\ol{S}_\ell$, $m(\bm{p}_1,\ldots,\bm{p}_n\colon \bm{\gl})$.\\
\noindent Subsection~\ref{Subsection:C3}: $\bm{a}$, $m(\bm{p}_1,\bm{p}_2)$.

\section{Preliminaries}\label{Section:pre}

\subsection{Quantum affine algebra}\label{Subsection:QAA}

Let $\fg$ be an affine Kac--Moody Lie algebra not of type $A_{2n}^{(2)}$ over $\Q$ with index set $I=\{0,1,\cdots,n\}$ and Cartan matrix $C=(c_{ij})_{i,j\in I}$.
We assume that the index $0$ coincides with the one prescribed in~\cite{MR1104219} 
(we do not assume this for the other indices,
and in fact later we use another labeling, see Figure~\ref{figure} in Subsection~\ref{Subsection:Three}).
Let $\ga_i$ and $h_i$ ($i\in I$) be the simple roots and simple coroots respectively,
$\gL_i$ ($i \in I$) the fundamental weights, $\gd$ the generator of null roots,
$P = \bigoplus_i \Z \gL_i \oplus \Z\gd$ the weight lattice, $P^+ = \bigoplus_{i \in I} \Z_{\ge 0} \gL_i \oplus \Z \gd$ the set of dominant weights,
$Q = \bigoplus_{i\in I} \Z \ga_i$ the root lattice, $Q^+ = \sum_{i \in I} \Z_{\geq 0} \ga_i \subseteq Q$,
$W$ the Weyl group with reflections $s_i$ ($i \in I$),
and $(\ , \ )$ a nondegenerate $W$-invariant bilinear form on $P$ satisfying $(\ga_0,\ga_0) =2$.
Set $I_0 = I\setminus \{0\}$, and 
\[ \varpi_i = \gL_i - \langle K, \gL_i\rangle \gL_0 \ \ \ \text{for $i \in I_0$},
\]
where $K \in P^* = \Hom(P,\Z)$ is the canonical central element. Let $d \in P^*$ be the element satisfying 
$\langle d, \gL_i\rangle = 0$ ($i \in I$) and $\langle d, \gd\rangle =1$.
Set $P_{\cl} = P/\Z \gd$,
and let $\cl\colon P \twoheadrightarrow P_\cl$ be the canonical projection.
For simplicity of notation, we will write $\ga_i$, $\varpi_i$ for $\cl(\ga_i)$, $\cl(\varpi_i)$
when there should be no confusion.

Let $q$ be an indeterminate.
Set $q_i=q^{(\ga_i,\ga_i)/2}$,
\[ [m]_i = \frac{q_i^m-q_i^{-m}}{q_i-q_i^{-1}}, \ \ \ [n]_i! = [n]_i[n-1]_i\cdots[1]_i, \ \ \text{and} \ \ \begin{bmatrix} m \\ n \end{bmatrix}_i
= \frac{[m]_i[m-1]_i\cdots[m-n+1]_i}{[n]_i!}
\]
for $i \in I$, $m \in \Z$, $n \in \Z_{\ge 0}$.
Choose a positive integer $D$ such that $(\ga_i,\ga_i)/2 \in \Z D^{-1}$ for all $i \in I$, and set $q_s = q^{1/D}$.
Let $U_q(\fg)$ be the quantum affine algebra, which is an associative $\Q(q_s)$-algebra generated by 
$e_i$, $f_i$ ($i \in I$), $q^{h}$ ($h \in D^{-1}P^*$) with certain defining relations (see, for example,~\cite{MR1890649}).
Denote by $U_q'(\fg)$ the quantum affine algebra without the degree operator,
that is, the subalgebra of $U_q(\fg)$ generated by $e_i$, $f_i$ ($i\in I$) and $q^h$ ($h \in D^{-1}P_{\cl}^*$).
Let $U_q(\fn_+)$ (resp.\ $U_q(\fn_-)$) be the subalgebra generated by $e_i$ (resp.\ $f_i$) ($i \in I$).
For $i \in I$ and $n \in \Z$, set $e_i^{(n)} = e_i^n/[n]_i!$ if $n \geq 0$, and $e_i^{(n)} = 0$ otherwise.
Define $f_i^{(n)}$ analogously.
We define a $Q$-grading $U_q(\fg) = \bigoplus_{\ga \in Q} U_q(\fg)_\ga$ by
\[ U_q(\fg)_\ga = \{ X \in U_q(\fg) \mid q^hXq^{-h} = q^{\langle h, \ga\rangle}X \ \text{for} \ h \in D^{-1}P^*\}.
\]
If $0\neq X \in U_q(\fg)_\ga$, we write $\wt_P(X) = \ga$.
For a proper subset $J \subset I$, denote by $\fg_J$ the corresponding simple Lie subalgebra,
and by $U_q(\fg_J)$ (resp.\ $U_q(\fn_{+,J})$, $U_q(\fn_{-,J})$) the $\Q(q_s)$-subalgebra of $U_q(\fg)$ generated by $e_i,f_i,q^{\pm D^{-1}h_i}$ 
(resp.\ $e_i$, $f_i$) with $i \in J$.


Set $t_i = q^{(\ga_i,\ga_i)h_i/2}$ for $i \in I$, and denote by $\gD$ the coproduct of $U_q(\fg)$ defined by
\begin{align*}\label{eq:coproduct}
 \gD(q^h) = q^h &\otimes q^h,\ \ \ 
 \gD(e_i^{(m)}) = \sum_{k=0}^m q_i^{k(m-k)}e_i^{(k)} \otimes t_i^{-k}e_i^{(m-k)},\\ 
 &\gD(f_i^{(m)})= \sum_{k=0}^m q_i^{k(m-k)}t_i^{m-k}f_i^{(k)}\otimes f_i^{(m-k)}
\end{align*}
for $h \in D^{-1}P^*$, $i\in I$, $m \in \Z_{> 0}$.

For a $U_q(\fg)$-module (resp.\ $U_q'(\fg)$-module) $M$ and $\gl \in P$ (resp.\ $\gl \in P_\cl$), write
\[ M_\gl = \{ v \in M \mid q^h v = q^{\langle h, \gl\rangle}v \ \text{for $h \in D^{-1}P^*$ (resp.\ $h \in D^{-1}P^*_\cl$)}\},
\]
and if $v \in M_\gl$ with $v\neq 0$, we write $\wt_P(v) = \gl$ (resp.\ $\wt_{P_\cl}(v) = \gl$).
We will omit the subscript $P$ or $P_\cl$ when no confusion is likely.
We say a $U_q(\fg)$-module (or $U_q'(\fg)$-module) $M$ is \textit{integrable} if $M = \bigoplus_{\gl} M_\gl$
and the actions of $e_i$ and $f_i$ ($i\in I$) are locally nilpotent.

Throughout the paper we will repeatedly use the following assertions.
For $i,j \in I$ such that $i\neq j$ and $r,s \in \Z_{\geq 0}$, it follows from the Serre relations that
\begin{equation}\label{eq:commutation3}
 \begin{split}
  e_i^{(r)}e_j^{(s)} &\in U_q(\fn_+)_{s(\ga_j-c_{ij}\ga_i)} e_i^{(r+c_{ij}s)} \ \ \  \text{if $r+c_{ij}s > 0$},\\
  e_j^{(s)}e_i^{(r)} &\in e_i^{(r+c_{ij}s)}U_q(\fn_+)_{s(\ga_j-c_{ij}\ga_i)} \ \ \ \text{if $r+c_{ij}s > 0$},
 \end{split}
\end{equation}
where $U_q(\fn_+)_\ga = U_q(\fn_+) \cap U_q(\fg)_\ga$.
For $i,j \in I$ such that $c_{ij}=c_{ji}=-1$ and $r,s,t \in \Z_{\geq 0}$, we have
\begin{align}\label{eq:commutation2}
 e_i^{(r)}e_j^{(s)}e_i^{(t)}= \sum_{m=0}^{r-s+t} \begin{bmatrix} r-s+t \\ m \end{bmatrix}_i e_j^{(t-m)}e_i^{(r+t)}e_j^{(s-t+m)}
 \ \ \ \text{if} \ r+t\geq s,
\end{align}
see \cite[Lemma 42.1.2]{MR1227098}.
Given a $U_q(\fg)$-module $M$, $v \in M_\gl$ and $r,s \in \Z_{\ge 0}$, we have
\begin{subequations}\label{eq:commutation4}
 \begin{align}
 e_i^{(r)}f_i^{(s)}v&=\sum_{k =0}^{\min(r,s)} \begin{bmatrix} r-s+\langle h_i,\gl\rangle \\ k\end{bmatrix}_i
 f_i^{(s-k)}e_i^{(r-k)}v, \\
 f_i^{(r)}e_i^{(s)}v&=\sum_{k =0}^{\min(r,s)} \begin{bmatrix} r-s-\langle h_i,\gl\rangle \\ k\end{bmatrix}_i
 e_i^{(s-k)}f_i^{(r-k)}v
 \end{align}
\end{subequations}
for $i \in I$, and $e_i^{(r)}f_j^{(s)}=f_j^{(s)}e_i^{(r)}$ for $i,j\in I$ such that $i\neq j$,
see [\textit{loc.\ cit.}, Corollary~3.1.9].

\subsection{Crystal (pseudo)bases and global bases}\label{Subsection:Crystal (pseudo)bases and global bases}

Let $M$ be an integrable $U_q(\fg)$-module (or $U_q'(\fg)$-module).
For $i \in I$, we have
\[ M = \bigoplus_{\gl; \langle h_i,\gl\rangle \geq 0} \bigoplus_{n=0}^{\langle h_i,\gl\rangle} f_i^{(n)}(\ker e_i \cap M_\gl).
\]
Endomorphisms $\tilde{e}_i,\tilde{f}_i$ ($i\in I$) on $M$ called the \emph{Kashiwara operators}
are defined by
\[ \tilde{f}_i(f_i^{(n)}u)=f_i^{(n+1)}u, \ \ \ \tilde{e}_i(f_i^{(n)}u)=f_i^{(n-1)}u
\]
for $u \in \ker e_i \cap M_\gl$ with $0\leq n \leq \langle h_i, \gl \rangle$.
These operators also satisfy that
\[ \tilde{e}_i(e_i^{(n)}v) = e_i^{(n+1)}v, \ \ \ \tilde{f}_i(e_i^{(n)}v) = e_i^{(n-1)}v
\]
for $v \in \ker f_i \cap M_\mu$ with $0 \leq n \leq -\langle h_i,\mu\rangle$.
Let $A$ be the subring of $\Q(q_s)$ consisting of rational functions without poles at $q_s=0$.
A free $A$-submodule $L$ of $M$ is called a \textit{crystal lattice} of $M$ if
$M \cong \Q(q_s)\otimes_A L$, $L = \bigoplus_{\gl} L_\gl$ where $L_\gl = L \cap M_\gl$, and $\tilde{e}_i$, $\tilde{f}_i$ ($i \in I$)
preserve $L$.
\begin{Def}[\cite{MR1115118,MR1194953}]\normalfont \ \\[2pt]
  (1) A pair $(L,B)$ is called a \textit{crystal base} of $M$ if\\[2pt]
  (i) $L$ is a crystal lattice of $M$, \ \ \ \ (ii) $B$ is a $\Q$-basis of $L/q_sL$,\\[2pt]
  (iii) $B=\bigsqcup_{\gl} B_\gl$ where $B_\gl = B\cap \big(L_\gl/q_sL_\gl)$, \ \ \ \ (iv) $\tilde{e}_iB \subseteq B \cup \{0\}$,
  $\tilde{f}_iB \subseteq B \cup \{0\}$,\\[2pt] (v) for $b,b'\in B$ and $i \in I$, $\tilde{f}_ib = b'$ if and only if $\tilde{e}_ib' = b$.\\[3pt]
 (2) $(L,B)$ is called a \textit{crystal pseudobase} of $M$ if they satisfy the conditions (i), (iii)--(v), and
 (ii') $B= B' \sqcup (-B')$ with $B'$ a $\Q$-basis of $L/q_sL$.
\end{Def}

Recall that, if $M_1$ and $M_2$ are integrable $U_q(\fg)$-modules and $(L_i,B_i)$ is a crystal base of $M_i$ ($i=1,2$),
then $(L_1\otimes_A L_2, B_1\otimes B_2)$ is a crystal base of $M_1\otimes M_2$, where $B_1 \otimes B_2 = \{b_1 \otimes b_2 
\mid b_i \in B_i\}\subseteq (L_1\otimes_A L_2)/q_s(L_1\otimes_A L_2)$.

Let $\ol{ \phantom{a} }$ denote the automorphism of $\Q(q_s)$ sending $q_s$ to $q_s^{-1}$, and set $\ol{A} = \{ \ol{a}\mid a \in A\}$.
We also denote by $\ol{ \phantom{a}}$ the involutive $\Q$-algebra automorphism of $U_q(\fg)$ defined by
\[ \ol{e_i}=e_i, \ \ \ \ol{f_i}=f_i, \ \ \ \ol{q^h}=q^{-h}, \ \ \ \ol{a(q_s)x} = a(q_s^{-1})\ol{x}
\]
for $i \in I$, $h\in D^{-1}P^*$, $a(q_s) \in \Q(q_s)$ and $x \in U_q(\fg)$.
Let $U_q(\fg)_\Q$ be the $\Q[q_s,q_s^{-1}]$-subalgebra of $U_q(\fg)$ generated by $e_i^{(n)}$, $f_i^{(n)}$, $q^h$ for
$i \in I$, $n \in \Z_{>0}$, $h \in D^{-1}P^*$.

\begin{Def}[\cite{MR1115118}]\normalfont \ \\[2pt]
 (1) Let $V$ be a vector space over $\Q(q_s)$, $L_0$ a free $A$-submodule, $L_\infty$ a free $\ol{A}$-submodule,
  and $V_\Q$ a free $\Q[q_s,q_s^{-1}]$-submodule.
  We say that $(L_0,L_\infty,V_\Q)$ is \textit{balanced} if each of $L_0$, $L_\infty$, and $V_\Q$ generates $V$ as a 
  $\Q(q_s)$-vector space, and the canonical map 
  \[ L_0 \cap L_\infty \cap V_\Q \to L_0/q_sL_0
  \]
  is an isomorphism.\\
 (2)  Let $M$ be an integrable $U_q(\fg)$-module with a crystal base $(L,B)$,
  $\ol{\phantom{a}}$ be an involution of $M$ (called a \textit{bar involution}) 
  satisfying $\ol{xu}=\ol{x}\,\ol{u}$ for $x \in U_q(\fg)$ and $u \in M$,
  and $M_\Q$ a $U_q(\fg)_\Q$-submodule of $M$ such that 
  \[ \ol{M_\Q} = M_{\Q}, \ \ \ u-\ol{u}\in (q_s-1)M_\Q \ \ \text{for } u\in M_\Q.
  \] 
  Assume that $(L, \ol{L}, M_\Q)$ is balanced, where $\ol{L} = \{ \ol{u} \mid u\in L\}$.
  Then, letting $G$ be the inverse of $L \cap \ol{L} \cap M_\Q \stackrel{\sim}{\to} L/q_sL$,
  the set 
  \[ \mathbf{B}=\{G(b) \mid b \in B\}
  \]
  forms a basis of $M$
  called a \textit{global basis} of $M$ (with respect to the bar involution $\ol{\phantom{a}}$\,).
\end{Def}

Note that the global basis $\bB$ is an $A$-basis of $L$.

\subsection{Polarization}\label{polarization}

A $\Q(q_s)$-bilinear pairing $(\ ,\ )$ between $U_q(\fg)$-modules (resp.\ $U_q'(\fg)$-modules) $M$ and $N$ 
is said to be \textit{admissible} if it satisfies
\begin{equation}\label{eq:admissible}
 \begin{split}
 (q^hu,v) &= (u,q^hv), \ \ \ (e_i^{(m)}u,v)=(u,q_i^{-m^2}t_i^{-m}f_i^{(m)}v),\\
 &(f_i^{(m)}u,v) = (u,q_i^{-m^2}t_i^{m}e_i^{(m)}v)
 \end{split}
\end{equation}
for $h \in D^{-1}P^*$ (resp.\ $h\in D^{-1}P^*_\cl$), $i \in I$, $m \in \Z_{> 0}$, $u \in M$, $v \in N$.
A bilinear form $(\ , \ )$ on $M$ is called a \textit{prepolarization} if it is symmetric and satisfies~\eqref{eq:admissible} for $u,v \in M$.
A prepolarization is called a \textit{polarization} if it is positive definite with respect to the following total order on $\Q(q_s)$:
\[ f > g \text{ if and only if } f-g \in \bigsqcup_{n\in \Z} \left\{ q_s^n(c+ q_sA)\mid c \in \Q_{>0}\right\},
\]
and $f \geq g $ if $f=g$ or $f>g$.
Throughout the paper, we use the notation $\normsq{u} = (u,u)$ for $u \in M$.

\subsection{Extremal weight modules}\label{Subsection:extremal weight modules}

For an arbitrary $\gL \in P$, let $V(\gL)$ be the \textit{extremal weight module}~\cite{MR1262212}
with generator $v_\gL$, which is an integrable 
$U_q(\fg)$-module generated by $v_\gL$ of weight $\gL$ with certain defining relations.
If $\gL$ belongs to the $W$-orbit of a dominant (resp.\ antidominant) weight, say $\gL^\circ$, then $V(\gL)$ is a simple highest (resp.\ lowest) weight module 
with highest (resp. lowest) weight $\gL^\circ$.
In [\textit{loc. cit.}], it was shown for any $\gL\in P$ that $V(\gL)$ has a crystal base $\big( L(\gL), B(\gL)\big)$
and $\big(L(\gL), \ol{L(\gL)}, V(\gL)_\Q\big)$ is balanced,
where the bar involution is defined by $\ol{xv_\gL}=\ol{x}v_\gL$ 
for $x \in U_q(\fg)$, and $V(\gL)_\Q = U_q(\fg)_\Q v_\gL$.
We denote by 
\[ \bB(\gL) = \{G(b)\mid b \in B(\gL)\} \subseteq V(\gL)
\] 
the associated global basis.
Let $U_q(\fg)_\Z$ denote the $\Z[q_s,q_s^{-1}]$-subalgebra of $U_q(\fg)$ generated by $e_i^{(n)},f_i^{(n)}$ ($i\in I$, $n \in \Z_{>0}$) 
and $q^h$ ($h\in D^{-1}P^*$),
and set $V(\gL)_\Z = U_q(\fg)_\Z v_\gL \subseteq V(\gL)$.
The following proposition is due to~\cite{MR1115118} for highest and lowest weight cases, and~\cite{MR2074599}
for level zero cases.

\begin{Prop}\label{Prop:Nakajima_prepolarization}
 Let $\gL \in P$.\\[2pt]
 {\normalfont(1)} There exists a polarization $(\ ,\ )$ on $V(\gL)$ such that 
 $\normsq{v_\gL} =1$.\\[2pt]
 {\normalfont(2)} We have $\big(L(\gL),L(\gL)\big) \subseteq A$, and $(\tilde{e}_iu,v) \equiv (u,\tilde{f}_iv)$ mod $q_sA$ for $u,v \in L(\gL)$ and $i \in I$.\\[2pt]
 {\normalfont(3)} $\bB(\gL)$ is an almost orthonormal basis with respect to $(\ ,\ )$, that is,
 \[ (v,v') \in \gd_{vv'}+q_sA \ \ \ \text{for} \ v,v' \in \mathbf{B}(\gL).
 \]
 {\normalfont(4)} We have
 \[ L(\gL) = \big\{v \in V(\gL)\bigm| \normsq{v} \in A\big\}, \ \ \pm\bB(\gL) = \big\{v \in V(\gL)_\Z\bigm| \ol{v}=v,\ \normsq{v} \in 1+q_sA\big\}.
 \]
\end{Prop}

Let $\gL^1,\gL^2 \in P^+$. 
By~\cite{MR1180036} (see also~\cite{MR1262212}), the triple
\[ \big(L(\gL^1)\otimes_A L(-\gL^2),\ol{L(\gL^1)\otimes_A L(-\gL^2)}, V(\gL^1)_\Q\otimes_{\Q[q_s,q_s^{-1}]} V(-\gL^2)_\Q\big)
\]
in the tensor product $V(\gL^1) \otimes V(-\gL^2)$ is balanced.
Here the bar involution is defined by
\[ \ol{x(v_{\gL_1}\otimes v_{-\gL_2})} = \ol{x}(v_{\gL_1}\otimes v_{-\gL_2}) \ \ \ \text{for} \ x \in U_q(\fg).
\]
Denote the associated global basis by
\[ \bB(\gL^1,-\gL^2)=\{G(b) \mid b \in B(\gL^1)\otimes B(-\gL^2)\} \subseteq V(\gL^1)\otimes V(-\gL^2).
\]
It is easily checked from the definition that
\begin{equation}\label{eq:gl.base_of_hwlw}
 v_{\gL^1}\otimes \bB(-\gL^2) \subseteq \bB(\gL^1,-\gL^2).
\end{equation}

By the construction of the global basis of an extremal weight module in \cite[Subsection 8.2]{MR1262212}, the following lemma is obvious.

\begin{Lem}\label{Lem:hw-lw-lev0}
 Let $\gL \in P$, and suppose that $\gL^1, \gL^2 \in P^+$ satisfy $\gL^1 - \gL^2 = \gL$.
 There exists a unique surjective $U_q(\fg)$-module homomorphism $\Psi$ from $V(\gL^1)\otimes V(-\gL^2)$ to $V(\gL)$ mapping 
 $v_{\gL^1} \otimes v_{-\gL^2}$ to $v_{\gL}$, and 
 $\Psi$ maps the subset $\{X \in \bB(\gL^1,-\gL^2)\mid \Psi(X) \neq 0\}$ bijectively to $\bB(\gL)$.
\end{Lem}

\subsection{Kirillov--Reshetikhin modules}\label{Subsection:fusion}

Given a $U_q'(\fg)$-module $M$, we define a $U_q(\fg)$-module $M_{\mathrm{aff}}=\Q(q_s)[z,z^{-1}]\otimes M$
by letting $e_i$ and $f_i$ ($i\in I$) act by $z^{\gd_{0i}}\otimes e_i$ and $z^{-\gd_{0i}}\otimes f_i$ respectively,
and $q^{D^{-1}d}$ on $z^k \otimes M$ by the scalar multiplication by $q_s^{k}$.
Set $M_a = M_{\mathrm{aff}}/(z-a)M_{\mathrm{aff}}$ for nonzero $a \in \Q(q_s)$, which is again a $U_q'(\fg)$-module.
We denote by $\iota_a\colon M \stackrel{\sim}{\to} M_a$ the $\Q(q_s)$-linear (not $U_q'(\fg)$-linear) isomorphism defined by
$\iota_a(v) = p_a(1\otimes v)$,
where $p_a\colon M_{\mathrm{aff}} \to M_a$ is the projection.
If no confusion is likely, we will write $\iota$ for $\iota_a$ sometimes.

Let $r \in I_0$. 
In~\cite{MR1890649}, a $U_q'(\fg)$-module automorphism $z_r$ of weight $\gd$ 
is constructed on the level-zero fundamental extremal weight module $V(\varpi_r)$,
which preserves the global basis $\bB(\varpi_r)$.
Set 
\[ 
W^{r,1} = V(\varpi_r)/(z_r-1)V(\varpi_r),
\]
which is a finite-dimensional simple integrable $U_q'(\fg)$-module called a \textit{fundamental module}.
Note that $W^{r,1}_{\mathrm{aff}} \cong V(\varpi_r)$.
Let $p\colon V(\varpi_r) \to W^{r,1}$ be the canonical projection, and define a bilinear form $(\ ,\ )$ on $W^{r,1}$ by
\begin{equation}\label{eq:invariance_of_zk}
 \big(p(u),p(v)\big) = \sum_{k \in \Z} (z^k_ru,v) \ \ \ \text{for } u,v \in V(\varpi_r).
\end{equation}
Since $(u,v)= (z_ru,z_rv)$ holds for $u,v \in V(\varpi_r)$ by \cite[Lemma 4.7]{MR2074599}, this is a well-defined polarization on $W^{r,1}$.
Let $L(W^{r,1}) = p\big(L(\varpi_r)\big)$. 
It follows from Proposition~\ref{Prop:Nakajima_prepolarization} that
\begin{align}\label{eq:LWr1}
 L(W^{r,1}) = \{u \in W^{r,1}\mid \normsq{u} \in A\},
\quad \text{ and } \quad
(u,v) \in A \text{ for any } u,v \in L(W^{r,1}).
\end{align}

Fix $r \in I_0$ and $\ell \in \Z_{>0}$.
Let $w_1 \in W^{r,1}$ denote a vector such that $\wt_{P_\cl}(w_1)=\varpi_r$ and 
$\normsq{w_1} = 1$. 
Hereafter we write $\iota_k$ for $\iota_{q^k}$ ($k\in D^{-1}\Z$).
Set
\[ m= \begin{cases} (\ga_r,\ga_r)/2 & \fg\colon \text{nontwisted affine type},\\ 1 & \fg\colon \text{twisted affine type}.\end{cases}
\]
Let 
\[ \widetilde{W} = W^{r,1}_{q^{m(1-\ell)}} \otimes W^{r,1}_{q^{m(3-\ell)}} \otimes \cdots \otimes W^{r,1}_{q^{m(\ell-3)}} \otimes W^{r,1}_{q^{m(\ell-1)}},
\]
and denote by $w_\ell$ a vector of $\widetilde{W}$ defined by
\[ w_\ell = \iota_{m(1-\ell)}(w_1) \otimes \iota_{m(3-\ell)}(w_1)\otimes \cdots \otimes \iota_{m(\ell-3)}(w_1)\otimes \iota_{m(\ell-1)}(w_1).
\]
The $U_q'(\fg)$-submodule $W^{r,\ell} = U_q'(\fg)w_\ell \subseteq \wti{W}$
 is called the \textit{Kirillov--Reshetikhin module} (KR module for short) associated with $r,\ell$.

\begin{Prop}\label{Prop:properties_of_KR} Let $r\in I_0$, $\ell \in \Z_{>0}$.\\
 {\normalfont(1)} $W^{r,\ell}$ is a finite-dimensional simple integrable $U_q'(\fg)$-module.\\
 {\normalfont(2)} The weight space $W^{r,\ell}_{\ell\varpi_r}$ is $1$-dimensional and spanned by $w_\ell$.\\
 {\normalfont(3)} The weight set $\{\gl \in P_\cl \mid W^{r,\ell}_{\gl} \neq 0\}$ coincides with the intersection of 
  $\ell\varpi_r-\sum_{i \in I_0}\Z_{\geq 0}\ga_i$ and the convex hull of the $W$-orbit of $\ell\varpi_r$.\\
 {\normalfont(4)} The vector $w_\ell \in W^{r,\ell}$ satisfies
  \[ e_iw_\ell = 0 \ \ \text{if $i \in I_0$} \ \ \ \text{and} \ \ \ f_iw_\ell = 0 \ \ \text{if $i \in I\setminus\{r\}$}.
  \]
\end{Prop}

\begin{proof}
 The assertion (1) is proved in \cite[Proposition 3.6]{MR2403558}. The assertions (2) and (3) follow from \cite[Theorem 5.17]{MR1890649},
 and (4) is proved from (3).
\end{proof}

Next we shall recall how to define a prepolarization on $W^{r,\ell}$.
There exists a unique $U_q'(\fg)$-module homomorphism
\[
R\colon W^{r,1}_{q^{m(\ell-1})} \otimes W^{r,1}_{q^{m(\ell-3)}} \otimes \cdots \otimes W^{r,1}_{q^{m(1-\ell)}} \to  
   W^{r,1}_{q^{m(1-\ell)}}  \otimes \cdots  \otimes W^{r,1}_{q^{m(\ell-3)}}\otimes W^{r,1}_{q^{m(\ell-1)}}
\]
mapping $\iota_{m(\ell-1)}(w_1) \otimes \cdots \otimes \iota_{m(1-\ell)}(w_1)$ to $w_\ell$, and its image is 
$W^{r,\ell}$ (see~\cite{MR2403558}).
The following lemma is proved straightforwardly.

\begin{Lem}\label{Lem:tensor_admissible}
 Assume that $\ell \in \Z_{>0}$, $M_k$, $N_k$ $(1\leq k \leq \ell)$ are $U_q'(\fg)$-modules, and $(\ , \ )_k\colon M_k \times N_k
 \to \Q(q_s)$ $(1\leq k \leq \ell)$ are admissible pairings.
 Then the $\Q(q_s)$-bilinear pairing $(\ ,\ )\colon (M_1\otimes \cdots \otimes M_\ell) \times (N_1\otimes \cdots \otimes N_\ell) \to \Q(q_s)$ defined by
 \[ (u_1\otimes u_2\otimes \cdots \otimes u_\ell, v_1\otimes v_2\otimes \cdots \otimes v_\ell) = (u_1,v_1)_1 (u_2,v_2)_2
    \cdots (u_\ell,v_\ell)_\ell
 \]
 is admissible.
\end{Lem}

The lemma gives an admissible pairing $(\ , \ )_0$ between $W^{r,1}_{q^{m(\ell-1)}} \otimes \cdots \otimes W^{r,1}_{q^{m(1-\ell)}}$
and $W^{r,1}_{q^{m(1-\ell)}}  \otimes \cdots  \otimes W^{r,1}_{q^{m(\ell-1})}$, which defines a bilinear form $(\ ,\ )$ on $W^{r,\ell}$
by
\begin{equation}\label{eq:construction_of_prepolarization} 
 \big(R(u),R(v)\big) = \big(u,R(v)\big)_0 \ \ \ \text{for} \ u,v \in W^{r,1}_{q^{m(\ell-1)}} \otimes \cdots \otimes W^{r,1}_{q^{m(1-\ell)}}.
\end{equation}
By \cite[Proposition 3.4.3]{MR1194953}, $(\ ,\ )$ is a nondegenerate prepolarization on $W^{r,\ell}$, and $\normsq{w_\ell} = 1$ holds.
We will use the following lemma later,
whose proof is similar to that of \cite[Lemma 3.6]{naoi2018existence}

\begin{Lem}\label{Lem:realization_of_KR}
 Let $r \in I_0$ and $\ell \in \Z_{>0}$, and set
 \[ W_1 = W^{r,1}_{q^{m(\ell-1)}}\otimes W^{r,\ell-1}_{q^{-m}} \ \ \text{and} \ \ W_2 = W^{r,1}_{q^{m(1-\ell)}} \otimes W^{r,\ell-1}_{q^m}.
 \] 
 There are unique $U_q'(\fg)$-module homomorphisms $R_1\colon W_1 \to W^{r,\ell}$ and $R_2\colon W^{r,\ell} \to W_2$ satisfying
 \[ R_1\big(\iota(w_1) \otimes \iota(w_{\ell-1})\big)=w_\ell \ \ \text{and} \ \ R_2(w_\ell) = \iota(w_1) \otimes \iota(w_{\ell-1})
 \]
 respectively, and for any $u,v \in W_1$ we have 
 \[ \big(R_1(u),R_1(v)\big) = \big(u,R_2\circ R_1(v)\big)_1,
 \]
 where $(\ ,\ )_1$ is the admissible pairing between $W_1$ and $W_2$ obtained from Lemma~\ref{Lem:tensor_admissible}.
\end{Lem}

\subsection{Criterion for the existence of a crystal pseudobase}\label{Subsection:Criterion}

Following the previous works~\cite{MR2403558,naoi2018existence,biswal2019existence}, we will prove Theorem~\ref{Thm:Main} by applying a criterion 
for the existence of a crystal pseudobase introduced in~\cite{MR1194953}.

We write $\fg_0=\fg_{I_0}$ for short.
We identify the weight lattice $P_0$ of $\fg_0$ with the subgroup $\bigoplus_{i \in I_0} \Z \varpi_i$ of $P_\cl$,
and set $P_0^+ = \sum_{i \in I_0} \Z_{\ge 0} \varpi_i$.
For $\gl \in P_0^+$, denote by $V_0(\gl)$ the simple integrable $U_q(\fg_0)$-module with highest weight $\gl$.

Let $A_\Z$ and $K_\Z$ be the subalgebras of $\Q(q_s)$ defined respectively by
\[ A_\Z = \{f(q_s)/g(q_s) \mid f(q_s),g(q_s) \in \Z[q_s], g(0)=1\}, \ \ \ K_\Z=A_\Z[q_s^{-1}].
\]
Let $U_q'(\fg)_{K_\Z}$ denote the $K_\Z$-subalgebra of $U_q'(\fg)$ generated by $e_i,f_i,q^h$ ($i \in I, h \in D^{-1}P_\cl^*$).


\begin{Prop}[{\cite[Propositions 2.6.1 and 2.6.2]{MR1194953}}]\label{Prop:criterion}
 Assume that $M$ is a finite-dimensional integrable $U_q'(\fg)$-module
 having a prepolarization $(\ , \ )$ and a $U_q'(\fg)_{K_\Z}$-submodule $M_{K_\Z}$ such that $(M_{K_\Z},M_{K_\Z})\subseteq
 K_\Z$.
 We further assume that there exist weight vectors $u_k \in M_{K_\Z}$ $(1\leq k \leq m)$ satisfying the following conditions:
 \begin{enumerate}
  \setlength{\parskip}{1pt} 
  \setlength{\itemsep}{1pt} 
  \item[(i)] $\wt (u_k) \in P_0^+$ for $1\leq k \leq m$ and $M \cong \bigoplus_{k=1}^m V_0\big(\wt(u_k)\big)$ as $U_q(\fg_0)$-modules,
  \item[(ii)] $(u_k,u_l) \in \gd_{kl} +q_sA$ for $1 \leq k,l \leq m$,
  \item[(iii)] $\normsq{e_iu_k}\in q_i^{-2\langle h_i,\wt(u_k)\rangle-2}q_sA$ for all $i \in I_0$ and $1\leq k \leq m$.
 \end{enumerate}
 Then $(\ , \ )$ is a polarization, and the pair $(L,B)$ with
 \[ L = \{u \in M \mid \normsq{u} \in A\} \text{ and } B= \{b \in (M_{K_\Z} \cap  L)/(M_{K_\Z} \cap q_sL)\mid (b,b)_0 = 1\},
 \]
 where $(\ , \ )_0$ is the $\Q$-valued bilinear form on $L/q_s L$ induced by $(\ , \ )$, is a crystal pseudobase of $M$.
\end{Prop}

From~\cite{MR1194953}, we know the $U_q'(\fg)_{K_\Z}$-submodule $W^{r,\ell}_{K_\Z} = U_q'(\fg)_{K_\Z}w_\ell \subseteq W^{r,\ell}$ satisfies 
$\big(W^{r,\ell}_{K_\Z},W^{r,\ell}_{K_\Z}\big) \subseteq K_\Z$.
Hence if we show for $M= W^{r,\ell}$ the existence of weight vectors $u_1,\ldots,u_m$ satisfying (i)--(iii),
Theorem~\ref{Thm:Main} follows from Proposition~\ref{Prop:criterion}.
We will show this in the next section with an explicit construction of the vectors $u_1, \dotsc, u_m$.

\section{Proof of Theorem~\ref{Thm:Main}}


\subsection{Set of vectors}\label{Subsection:Three}

In the rest of this paper, assume that $\fg$ is either of type $E_{n}^{(1)}$ ($n=6,7,8$), $F_4^{(1)}$ or $E_6^{(2)}$
and the nodes of the Dynkin diagram is labeled as in Figure~\ref{figure}.
We have
\[ q_i = q^{1/2} \ \ (\fg\colon F_4^{(1)}, i=3,4), \ \ \ \ q_i=q^2 \ \ (\fg\colon E_6^{(2)}, i=3,4), \ \ \ \ q_i=q \ \ (\text{otherwise}).
\]
From now on, for $i \in I$ such that $q_i=q$ we write $[m]$ for $[m]_i$, $[n]!$ for $[n]_i!$, and $\begin{bmatrix} m \\ n \end{bmatrix}$ 
for $\begin{bmatrix} m \\ n \end{bmatrix}_i$.
Note that in all types $r=2$ is the unique near adjoint node.
In the sequel, we will consider $W^{2,\ell}$ only and, hence, write $W^\ell$ for $W^{2,\ell}$.

Let us prepare several notation.
Define two subsets $I_{01}$ and $J$ of $I$ by $I_{01} = I_0 \setminus \{1\}$, and 
\[ J = \begin{cases} \{2,3,4\} & (\fg\colon E_6^{(1)}),\\ \{2,3,4,5\} & (\fg\colon E_7^{(1)}), \\ \{2,3,4,5,6,7\} & (\fg\colon E_8^{(1)}),\\
   \{2,3\} & (\fg \colon F_4^{(1)}, E_6^{(2)}).\end{cases}
\]
Let $R\subseteq Q$ denote the root system of $\fg_0$, and $R^+ = R\cap Q^+$ the set of positive roots.
For a subset $L \subset I_0$ denote by $R_L$ the root subsystem of $R$ generated by the simple roots corresponding 
to the elements of $L$, and let $R_L^+ = R_L\cap R^+$.
We write $R_{1} = R_{I_{01}}$.
Let $\theta_{1}$ be the highest short root of $R_{1}$ if $\fg$ is of type $E_6^{(2)}$, and the highest root of $R_{1}$ otherwise.
Define $\theta_J \in R_J$ similarly (see Table~\ref{table}). 
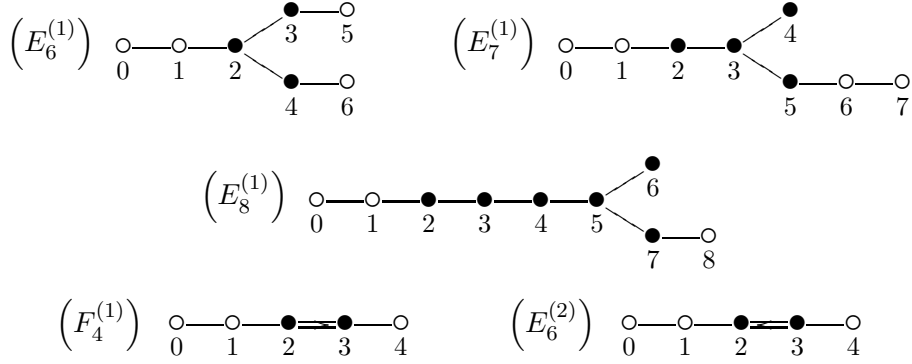
\begin{figure}[t]
\begin{align*} \raisebox{-12pt}{$\left(E_6^{(1)}\right)$} \ \ {\Large
\xymatrix@C=14pt@R=-3pt@M=0pt{
 & & & \underset{3}{\bullet} \ar@{-}[r] & \underset{5}{\circ} \\
\underset{0}{\circ} \ar@{-}[r] & \underset{1}{\circ} \ar@{-}[r] & \underset{2}{\bullet} \ar@{-}[ru] \ar@{-}[rd] & &  \\
 & & & \underset{4}{\bullet} \ar@{-}[r] & \underset{6}{\circ}}
} \hspace{35pt}
 \raisebox{-12pt}{$\left(E_7^{(1)}\right)$} \ \ {\Large
\xymatrix@C=14pt@R=-3pt@M=0pt{
 & & & & \underset{4}{\bullet} & & \\
\underset{0}{\circ} \ar@{-}[r] & \underset{1}{\circ} \ar@{-}[r] & \underset{2}{\bullet} \ar@{-}[r] & \underset{3}{\bullet} \ar@{-}[ru] \ar@{-}[rd]& & & \\
 & & & & \underset{5}{\bullet} \ar@{-}[r] & \underset{6}{\circ} \ar@{-}[r]& \underset{7}{\circ}}
}\end{align*}
\begin{align*}
\raisebox{-12pt}{$\left(E_8^{(1)}\right)$} \ \ {\Large
\xymatrix@C=14pt@R=-3pt@M=0pt{
 & & & & & & \underset{6}{\bullet} & \\
\underset{0}{\circ} \ar@{-}[r] & \underset{1}{\circ} \ar@{-}[r] & \underset{2}{\bullet} \ar@{-}[r] &\underset{3}{\bullet} \ar@{-}[r] & \underset{4}{\bullet} \ar@{-}[r] &
\underset{5}{\bullet}\ar@{-}[ru] \ar@{-}[rd] & &  \\
 & & & & & & \underset{7}{\bullet} \ar@{-}[r] & \underset{8}{\circ}}
}
\end{align*}
\begin{align*} 
\left(F_4^{(1)}\right) \ \ {\Large
\xymatrix@C=14pt@M=0pt{
 \underset{0}{\circ} \ar@{-}[r] & \underset{1}{\circ} \ar@{-}[r] & \underset{2}{\bullet} \ar@{=}[r]|(.7)@{>} & \underset{3}{\bullet}  \ar@{-}[r] & \underset{4}{\circ}}
} \hspace{35pt}
\left(E_6^{(2)}\right) \ \ {\Large
\xymatrix@C=14pt@M=0pt{
 \underset{0}{\circ} \ar@{-}[r] & \underset{1}{\circ} \ar@{-}[r] & \underset{2}{\bullet} \ar@{=}[r]|(.3)@{<} & \underset{3}{\bullet}  \ar@{-}[r] & \underset{4}{\circ}}
}
\end{align*}
\caption{Dynkin diagrams of types $E_{6,7,8}^{(1)}$, $F_4^{(1)}$, and $E_6^{(2)}$ ($\bullet$: nodes belonging to $J$)}\label{figure}
\end{figure}

\begin{table}[t]
\[
\begin{array}{|c|l|l|} \hline
\fg & \multicolumn{1}{|c|}{\theta_{1}}  & \multicolumn{1}{|c|}{\theta_J} \\ \hline 
E_6^{(1)} &  \ga_2+\ga_3+\ga_4+\ga_5+\ga_6  & \ga_2+\ga_3+\ga_4 \\ 
            &  =-\varpi_1+\varpi_5+\varpi_6  &  =-\varpi_1+\varpi_3+\varpi_4-\varpi_5-\varpi_6 \\ \hline 
E_7^{(1)} &  \ga_2+2\ga_3+\ga_4+2\ga_5+2\ga_6+\ga_7 &  \ga_2+2\ga_3+\ga_4+\ga_5 \\ 
            &   =-\varpi_1+\varpi_6                    & =-\varpi_1+\varpi_3-\varpi_6\\ \hline
E_8^{(1)} & \text{\scriptsize $\ga_2+2\ga_3+3\ga_4+4\ga_5+2\ga_6+3\ga_7+2\ga_8$}  & \ga_2+2\ga_3+2\ga_4+2\ga_5+\ga_6 +\ga_7 \\ 
            &   =-\varpi_1 +\varpi_8                                          &  =-\varpi_1+\varpi_3-\varpi_8\\ \hline
F_4^{(1)} &  \ga_2+2\ga_3+2\ga_4=-\varpi_1+2\varpi_4 &  \ga_2+2\ga_3=-\varpi_1+2\varpi_3-2\varpi_4 \\ \hline
E_6^{(2)} & \ga_2+\ga_3+\ga_4=-\varpi_1+\varpi_4 &  \ga_2+\ga_3=-\varpi_1+\varpi_3-\varpi_4 \\ \hline
\end{array}
\]
\caption{Explicit forms of $\theta_1$ and $\theta_J$}\label{table}
\end{table}

For $i \in I$ and $k \in \Z$, set
\[ E_i^{(k)}= \begin{cases} e_i^{(2k)} & \text{if $\fg$ is of type $F_4^{(1)}$ and $\ga_i$ is short},\\
                       e_i^{(k)} & \text{otherwise}.\end{cases}
\]
For $p \in \Z$ and a sequence $\bm{r}=(r_kr_{k-1}\cdots r_1)$ of elements of $I$ (in this paper we always read such sequences from right to left),
we use the abbreviations
\begin{equation}\label{eq:epEp}
 e^{(p)}_{\bm{r}} = e^{(p)}_{r_k} e^{(p)}_{r_{k-1}}\cdots e^{(p)}_{r_1} \ \ \text{and} \ \ E^{(p)}_{\bm{r}} = E^{(p)}_{r_k}\cdots E^{(p)}_{r_1}.
\end{equation}
Set
\[ c_{\fg} = \begin{cases} 2 & (\fg\colon F_4^{(1)}),\\ 1 & (\text{otherwise}),\end{cases}
\]
and choose a sequence $\bm{i}=(i_L i_{L-1} \cdots i_1i_{0})$ of elements of $I_{01}$ satisfying
\begin{equation}\label{equation:condition_for_i}
 i_0 = 2, \ \ s_{i_L}\cdots s_{i_1}(\ga_2) = \theta_{1}, \ \ \text{and} \ \ \langle h_{i_k}, s_{i_{k-1}}\cdots s_{i_1}(\ga_2)\rangle = -c_\fg \ \
  \text{for} \ 1\leq k \leq L.
\end{equation} 
Similarly, choose a sequence $\bm{j}=(j_{L'} j_{L'-1}\cdots j_1j_0)$ of elements of $J$ satisfying
\begin{equation*}
 \begin{split}
  j_0 = 2, \ \ s_{j_{L'}}\cdots s_{j_1}(\ga_2) = \theta_J, \ \ \text{and} \ \ 
  \langle h_{j_k}, s_{j_{k-1}}\cdots s_{j_1}(\ga_2)\rangle = -c_\fg  \ \  \text{for} \ 1\leq k \leq L'.
 \end{split}
\end{equation*} 
In the rest of this paper, we fix $\bm{i}=(i_L  \cdots i_{0})$ and $\bm{j}=(j_{L'} \cdots j_0)$
satisfying these conditions.
For $0\leq k_1 \leq k_2 \leq L$, denote by $\bm{i}[k_2,k_1]$ the subsequence $(i_{k_2}i_{k_2-1}\cdots i_{k_1})$ of $\bm{i}$,
and let $\bm{i}[k_2,k_1]$ be the empty set if $k_2<k_1$.
We define $\bm{j}[k_2,k_1]$ similarly.
For a sequence $\bm{r} = (r_\ell r_{\ell-1}\cdots r_1)$ of elements of $I$, set $s_{\bm{r}} = s_{r_\ell}\cdots s_{r_1} \in W$,
and let $s_{\bm{r}}$ be the identity element of $W$ if $\bm{r}$ is the empty set.
Let $h_\ga \in P^*$ ($\ga \in R$) denote the coroots,
and $\gL_i^\vee \in P^*\otimes_\Z \Q$ ($i \in I$) elements satisfying $\langle \gL_i^\vee, \ga_j\rangle =\gd_{ij}$ for $i,j \in I$.

\begin{Lem}\label{Lem:fundamental_properties} \ \\[2pt]
 {\normalfont (1)} Neither of the subsequences $\bm{i}[L,1]$ and $\bm{j}[L',1]$ contains $2$.\\
 {\normalfont (2)} We have $\langle h_i,\theta_1\rangle = 0$ for all $i \in J$.\\
 {\normalfont (3)} For any $p \in \Z_{\geq 0}$, we have 
  \[
  \wt_P\big(E_{\bm{i}[k,0]}^{(p)}\big)= ps_{\bm{i}[k,1]}(\ga_2) \ \ (0\leq k \leq L)
  \ \ \ \text{and} \ \ \ 
  \wt_P\big(E_{\bm{j}[k,0]}^{(p)}\big) =ps_{\bm{j}[k,1]}(\ga_2) \ \ (0\leq k \leq L').
  \]
  In particular, $\wt_P\big(E_{\bm{i}}^{(p)}\big)= p\theta_{1}$ and $\wt_P\big(E_{\bm{j}}^{(p)}\big)=p\theta_J$ hold.\\
 {\normalfont (4)} Both $s_{\bm{i}}$ and $s_{\bm{j}}$ are reduced expressions.\\
 {\normalfont (5)} If $\ga \in R^+$ satisfies $s_{\bm{i}[L,1]}^{-1}(\ga) \in -R^+$ 
  (resp.\ $s_{\bm{j}[L',1]}^{-1}(\ga) \in -R^+$), 
  then we have $\langle h_\ga,\theta_1\rangle >0$ (resp.\ $\langle h_\ga, \theta_J\rangle >0$).\\
 {\normalfont(6)} For any $p \in \Z_{\geq 0}$, 
  $E_{\bm{i}}^{(p)}$ (resp.\ $E_{\bm{j}}^{(p)}$) does not depend on the choice of $\bm{i}$ (resp.\ $\bm{j}$).
\end{Lem}

\begin{proof}
 The assertion (1) is obvious since $\langle \gL_2^\vee, \theta_1 \rangle = \langle \gL_2^\vee,\theta_J \rangle = 1$ (see Table~\ref{table}),
 and (2) is checked directly. The assertion (3) is easily seen from the conditions on $\bm{i}$ and $\bm{j}$.
 We will show the assertion (4) for $s_{\bm{i}}$ (the proof for $s_{\bm{j}}$ is similar).
 By the condition on $\bm{i}$, we have for any $0\leq k \leq L$ that
 \[ \langle s_{\bm{i}[L,k+1]}(h_{i_k}),\theta_1\rangle =\langle h_{i_k},s_{\bm{i}[k,1]}(\ga_{2})\rangle >0.
 \]
 Since $\langle h_i,\theta_1\rangle \geq 0$ for all $i\in I_{01}$ and $s_{\bm{i}[L,k+1]}(\ga_{i_k}) \in R_1$, this implies that
 $s_{\bm{i}[L,k+1]}(\ga_{i_k})$ is a positive root for any $k$, which implies that $s_{\bm{i}}$ is reduced.
 Let us show the assertion (5) for $s_{\bm{i}[L,1]}$ (the proof for $s_{\bm{j}[L',1]}$ is similar).
 There exists $1\leq k \leq L$ such that $\ga = s_{\bm{i}[L,k+1]}(\ga_{i_k})$, and we have
 \[ \langle h_\ga, \theta_1\rangle = \langle s_{\bm{i}[L,k+1]}(h_{i_k}),\theta_1\rangle = \langle h_{i_k},s_{\bm{i}[k,1]}(\ga_2)\rangle>0,
 \]
 as required. 
 Finally, let us show the assertion (6) for $E_{\bm{i}}^{(p)}$ (the proof for $E_{\bm{j}}^{(p)}$ is similar).
 If $\fg$ is either of type $F_4^{(1)}$ or $E_6^{(2)}$, $\bm{i}=(43)$ is the unique choice.
 Hence we may assume that $\fg$ is of type $E_n^{(1)}$ ($n=6,7,8$).
 Assume that $\bm{i}'=(i'_{L_0},\ldots,i_0')$ is another choice.
 Since $\sum_{k=0}^{L} \ga_{i_k} = \sum_{k=0}^{L_0}\ga_{i'_k}=\theta_{1}$, we have $L_0=L$.
 Let $r$ be the smallest number such that $i_r \neq i_r'$, and let $s$ be the smallest number such that $r<s$ and $i_r = i_s'$.
 Then since 
 \[ \langle h_{i_r},\sum_{k=0}^{r-1} \ga_{i_k}\rangle =-1 = \langle h_{i_s'},\sum_{k=0}^{s-1} \ga_{i_k'}\rangle
 \]
 and $i_k' \neq i_r$ for $r\leq k < s$,
 we have $\langle h_{i_r}, \ga_{i_k'}\rangle = 0$ for $r\leq k <s$.
 Hence setting 
 \[ \bm{i}''=(i_{L}'\cdots i_{s+1}'i_{s-1}'\cdots i_r'i_r\cdots i_0),
 \]
 we have $E_{\bm{i}'}^{(p)}=E_{\bm{i}''}^{(p)}$.
 By repeating this argument we can show that $E_{\bm{i}'}^{(p)} = E_{\bm{i}}^{(p)}$, and hence the assertion (6) is proved. 
\end{proof}
For $\bm{p} = (p_1,p_2,\ldots,p_6) \in \Z^{6}$, we write 
 \[ E^{\bm{p}} = e_0^{(p_6)}e_1^{(p_5)}e_2^{(p_4)}E_{\bm{j}}^{(p_3)}E_{\bm{i}}^{(p_2)}e_{10}^{(p_1)} \in U_q(\fn_+),
 \]
and define a map $\wt\colon \Z^{6} \to P_0$ by 
\begin{align*}
 &\wt(\bm{p})\\
&=(p_1-p_2-p_3-p_4+2p_5-p_6)\varpi_1 +(-p_1+2p_4-p_5)\varpi_2+(p_2-p_3)\ggg_1+(p_3-p_4)\ggg_2,
\end{align*} 
where we set 
\begin{equation}\label{eq:ggg}
 \ggg_1 = \varpi_1+\cl(\theta_1) \in P_0^+ \ \ \text{and} \ \  \ggg_2 =\varpi_1+\ggg_1+\cl(\theta_J) \in P_0^+.
\end{equation}
 For $\ell \in \Z_{>0}$, define a finite subset $S_\ell \subseteq \Z_{\geq 0}^{6}$ by
 \begin{align*}
  S_\ell= \big\{(p_1,\ldots,p_6)\in\Z_{\geq 0}^{6} \bigm| p_6\leq p_5\leq p_4\leq p_3\leq p_2,\ p_2+p_3+p_4-p_5\leq p_1 \leq p_4 +\ell\}.
 \end{align*}
Note that if $\bm{p}\in S_\ell$, then $\wt_{P_\cl}(E^{\bm{p}}w_\ell) = \wt(\bm{p})+\ell\varpi_2 \in P_0^+$.
As stated in the final part of the previous section, Theorem~\ref{Thm:Main} is proved once we show the following.

\begin{Prop}\label{Prop:critical_3}
  For any $\ell \in \Z_{>0}$, the vectors $\{E^{\bm{p}}w_\ell \mid \bm{p} \in S_\ell\} \subseteq W^{\ell}$ satisfy the following 
  conditions:
 \begin{enumerate}
  \setlength{\parskip}{1pt} 
  \setlength{\itemsep}{1pt} 
  \item[(C1)] $W^{\ell} \cong \bigoplus_{\bm{p} \in S_\ell} V_0\big(\wt(\bm{p}) + \ell\varpi_2\big)$ as $U_q(\fg_0)$-modules,
  \item[(C2)] $(E^{\bm{p}}w_\ell,E^{\bm{p}'}w_\ell) \in \gd_{\bm{p},\bm{p}'} + q_sA$ for $\bm{p},\bm{p}' \in S_\ell$,
  \item[(C3)] $\normsq{e_iE^{\bm{p}}w_\ell}\in q_i^{-2\langle h_i,\wt(\bm{p})\rangle-2\ell\gd_{i2}-2}q_sA$ for $i \in I_0$ and $\bm{p} \in S_\ell$.
 \end{enumerate}
\end{Prop}


\subsection{Proof of (C1) in Proposition~\ref{Prop:critical_3}}\label{Subsection:C1}

By~\cite{MR1993360,MR2254805,MR2428305}, the multiplicities of a KR module are known to 
coincide with the cardinalities of \textit{highest weight rigged configurations}.
In our cases, explicit formulas for the number of them have been obtained using the Kleber algorithm~\cite{kleber1998finite},
and hence we have the following.

\begin{Prop}[{\cite[Section 9]{scrimshaw2017uniform}}]
\label{prop:hwrc}
 Let $\ell \in \Z_{>0}$. Define a subset $T_\ell \subseteq \Z_{\geq 0}^5$ by
 \[ T_\ell=\{\bm{r}=(r_1,r_2,\ldots,r_5) \in \Z_{\geq 0}^5 \mid r_1+r_2+r_3+r_4 \leq \ell, \ r_4+2r_5\leq r_2\},
 \]
 and a map $\wt_{T}\colon \Z_{\geq 0}^5 \to P_0$ by
 \[ \wt_{T}(\bm{r}) = (r_2-r_4-2r_5)\varpi_1+(-r_1-r_2-r_3-r_4+r_5)\varpi_2+r_3\ggg_1+r_4\ggg_2,
 \]
 where $\ggg_1,\ggg_2$ are given in~\eqref{eq:ggg}.
 Then we have 
 \begin{align*}
   W^\ell\cong \bigoplus_{\bm{r} \in T_\ell} V_0\big(\wt_{T}(\bm{r})+\ell\varpi_2\big)^{\oplus (1+r_2-r_4-2r_5)}
 \end{align*}
 as $U_q(\fg_0)$-modules.
\end{Prop}

Now (C1) is easily deduced from Proposition~\ref{prop:hwrc}.
Indeed, the map $\phi\colon \Z^{6} \to \Z^{5}$ defined by
\[ \phi(p_1,\ldots,p_6) \mapsto (p_6,p_1-p_2-p_6,p_2-p_3,p_3-p_4,p_4-p_5)
\]
sends $S_\ell$ to $T_\ell$, $\wt_T\circ \phi = \wt$ holds, and for any $\bm{r}\in T_\ell$,
\begin{align*}
 \phi^{-1}(\bm{r}) \cap S_\ell = \{\bm{r}_0+k(1,1,1,1,1,0)\mid r_1\leq k \leq r_1+r_2-r_4-2r_5\},
\end{align*} 
where
\[ \bm{r}_0=(r_1+r_2+r_3+r_4+r_5,r_3+r_4+r_5,r_4+r_5,r_5,0,r_1),
\]
and hence Proposition~\ref{prop:hwrc} is equivalent to (C1).

\subsection{Proof of (C2) in Proposition~\ref{Prop:critical_3}}\label{Subsection:C2}

In this and next subsections, we need to consider prepolarizations on several types of modules 
(extremal weight modules, KR modules, or tensor products of them) simultaneously.
Therefore, when we would like to indicate what prepolarization we are considering, we will occasionally write $(\ ,\ )_M$ and $\normsq{\ \ }_M$ for $(\ ,\ )$ 
and $\normsq{\ \ }$ on a module $M$.

We begin with the following lemma.

\begin{Lem}\label{Lem:orthogonality}
 Let $M$ be a $U_q'(\fg)$-module with a prepolarization $(\ , \ )$, and $u \in M_\gl$ for some $\gl \in P_\cl$.
 Assume that $f_0u=e_1u=f_1u=0$.
 Then for any $\bm{p}, \bm{p}' \in \Z_{\geq 0}^{6}$ with $\bm{p} \neq \bm{p}'$, $(E^{\bm{p}}u,E^{\bm{p}'}u)=0$ holds.
\end{Lem}

\begin{proof}
 Set $\bm{p}=(p_1,\ldots,p_6)$ and $\bm{p}'=(p_1',\ldots,p_6')$.
 We may assume that $p_6 \geq p_6'$.
 By the admissibility, we have
 \[ (E^{\bm{p}}u,E^{\bm{p}'}u) = q^c(E^{\bm{p}-p_6{\bm{\gee}}_6}u,f_0^{(p_6)}E^{\bm{p}'}u),
 \]
 where $c$ is a certain integer and $\bm{\gee}_i = ( \underbrace{0,\ldots,0,1}_i,0,\ldots,0)$ ($1\leq i \leq 6$) is the standard basis of $\Z^6$.
 Since $e_1^{(a)}e_0^{(b)}u=0$ if $a>b$ by~\eqref{eq:commutation3}, it follows from~\eqref{eq:commutation4} that
 \[ f_0^{(p_6)}E^{\bm{p}'}u = \gd_{p_6p_6'} q^{c'}E^{\bm{p}'-p_6\bm{\gee}_6}u
 \] 
 with $c' \in \Z$, and hence we may (and do) assume that $p_6=p_6'=0$.
 If we further assume that $p_5=p_5'$, then $\bm{p} \neq \bm{p}'$ implies $\wt_{P_\cl}(E^{\bm{p}}u) \neq \wt_{P_\cl}(E^{\bm{p}'}u)$,
 which forces $(E^{\bm{p}}u,E^{\bm{p}'}u)=0$.

 Hence we may assume that $p_5 > p_5'$.
 In this case, we have
 \begin{equation}\label{eq:applying_addmissible}
  (E^{\bm{p}}u,E^{\bm{p}'}u)= q^{c''}(E^{\bm{p}-p_5\bm{\gee}_5}u, f_1^{(p_5)}E^{\bm{p}'}u)
 \end{equation}
 with $c''\in \Z$, and by applying~\eqref{eq:commutation3} and~\eqref{eq:commutation4}, it is easily proved that
 \[ f_1^{(p_5)}E^{\bm{p}'}u \in e_0U_q(\fg)u.
 \]
 Since $f_0E^{\bm{p}-p_5\bm{\gee}_5}u=0$,~\eqref{eq:applying_addmissible} implies $(E^{\bm{p}}u,E^{\bm{p}'}u)=0$, and the assertion is proved.
\end{proof}

Since the vector $w_\ell\in W^{\ell}$ satisfies the assumption of the lemma,
$(E^{\bm{p}}w_\ell,E^{\bm{p}'}w_\ell) = 0$ follows if $\bm{p} \neq \bm{p}'$.
In order to verify (C2) in Proposition~\ref{Prop:critical_3}, it remains to show $\normsq{E^{\bm{p}}w_\ell} \in 1+q_sA$ for $\bm{p}\in S_\ell$.

\begin{Lem}\label{Lem:reduction_of_(ii)}
 For any $\bm{p}=(p_1,\ldots,p_6) \in \Z^6_{\geq 0}$ such that $p_1-p_5+p_6\leq 3\ell$,
 we have $\normsq{E^{\bm{p}}w_\ell} \in (1+qA)\normsq{E^{\bm{p}-p_6\bm{\gee}_6}w_\ell}$. 
\end{Lem}

\begin{proof}
 We have
 \begin{align}\label{eq:e06ep}
  \normsq{E^{\bm{p}}w_\ell}= q^{p_6(3\ell-p_1+p_5-p_6)}(E^{\bm{p}-p_6\bm{\gee}_6}w_\ell,f_0^{(p_6)}E^{\bm{p}}w_\ell).
 \end{align}
 Since $f_0E^{\bm{p}-p_6\bm{\gee}_6}w_\ell=0$ holds, it follows from~\eqref{eq:commutation4}
 that 
 \[ \eqref{eq:e06ep} = q^{p_6(3\ell-p_1+p_5-p_6)}\begin{bmatrix} 3\ell-p_1+p_5\\p_6\end{bmatrix} \lVert E^{\bm{p}-p_6\bm{\gee}_6}w_\ell \rVert^2 \in (1+qA)
    \lVert E^{\bm{p}-p_6\bm{\gee}_6}w_\ell \rVert^2.
 \]
 The lemma is proved.
\end{proof}

In the sequel, we regard $\Z^5$ as a subgroup of $\Z^6$ via $\Z^5 \ni \bm{p} \hookrightarrow (\bm{p},0) \in \Z^6$.
Hence for $\bm{p} = (p_1,\ldots,p_5) \in \Z^5$, we have
\[ E^{\bm{p}}=e_1^{(p_5)}e_2^{(p_4)}E_{\bm{j}}^{(p_3)}E_{\bm{i}}^{(p_2)}e_{10}^{(p_1)}.
\]
For $\ell \in \Z_{>0}$, set 
\[ \ol{S}_\ell = S_\ell \cap \Z^5 = \{(p_1,\ldots,p_5)\mid p_5\leq p_4\leq p_3\leq p_2, \ p_2+p_3+p_4-p_5 \leq p_1 \leq p_4+\ell\}.
\]
By the lemma, the proof of the assertion $\normsq{E^{\bm{p}}w_\ell} \in 1+q_sA$ for $\bm{p} \in S_\ell$ is reduced to the case 
$\bm{p} \in \ol{S}_\ell$.
An idea for the proof of this assertion is to use the 
almost orthonormality of $\bB(\ell\varpi_2)$, the global basis of the extremal weight module $V(\ell\varpi_2)$.
To do this we need to show that $E^{\bm{p}}v_{\ell\varpi_2} \in \pm\bB(\ell\varpi_2) \cup \{0\}$ for 
$\bm{p} \in \Z_{\geq 0}^{5}$.
For this purpose, we prepare several lemmas.

\begin{Lem}\label{Lem:being_global_basis}
 Let $\gL \in P$ and $i \in I$, and assume that $u \in \pm \bB(\gL)$.\\
 {\normalfont(1)} If 
 \[ f_i^{(n)}u \in \pm \bB(\gL) \cup \{0\} \ \ \text{for all} \ n > 0,
 \]
 then we have $e_i^{(n)}u \in \pm\bB(\gL) \cup \{0\}$ for all $n >0$.\\[3pt]
 {\normalfont(2)} In particular, if $f_iu=0$ then $e_i^{(n)}u \in \pm\bB(\gL) \cup \{0\}$ for all $n >0$.

\end{Lem}

\begin{proof}
 Let us prove the assertion (1) (note that (2) is just a special case).
 Since $u \in \pm\bB(\gL)$, it follows from Proposition~\ref{Prop:Nakajima_prepolarization} (4) 
 that $e_i^{(n)}u$ is bar-invariant and $e_i^{(n)}u\in V(\gL)_\Z$
 for any $n>0$.
 Hence, again by the same proposition, it suffices to show that $\normsq{e_i^{(n)}u} \in 1+q_sA$ for $n>0$ such that
 $e_i^{(n)}u \neq 0$.
 Set
 \[ L_1 = \{ v \in V(\gL) \mid \normsq{v} \in 1+q_sA \} \subseteq L(\gL).
 \]
 Let $\gl \in P$ be the weight of $u$, and set $\gl_i = \langle h_i,\gl\rangle \in \Z$.
 Write
 \[ u = \sum_{k=\max(0,-\gl_i)}^N f_i^{(k)}u_k, \ \ \ \text{where } u_k \in \ker e_i\cap V(\gL)_{\gl+k\ga_i}.
 \]
 Here we set $N = \max\{k\in \Z_{\geq 0} \mid u_k \neq 0\}$.
 By Proposition~\ref{Prop:Nakajima_prepolarization} (2),
 it follows for every $u_k$ that
 \begin{equation}\label{eq:invariance_of_prepolarization}
  \normsq{f^{(m)}_iu_k} = \normsq{\tilde{f}^{m}_iu_k} \in (1+q_sA)\normsq{u_k}\ \ \ \text{if } 0 \leq m \leq 2k+\gl_i.
 \end{equation}
 We shall show that $u_k \in q_i^{k(\gl_i+k)}L_1$ for every $k$ by the descending induction.
 For $0\leq n\leq N+\gl_i$, we have
 \begin{equation}\label{eq:f_iu}
  0\neq f_i^{(n)}u=\sum_{k=\max(0,-\gl_i)}^N \begin{bmatrix} k+n\\ k\end{bmatrix}_i f_i^{(k+n)}u_k \in \pm \bB(\gL) \subseteq L_1
 \end{equation}
 by the assumption.
 Since $f_i^{(k+N+\gl_i)}u_k = 0$ for $k<N$,~\eqref{eq:f_iu} with $n=N+\gl_i$ implies $\begin{bmatrix} 2N+\gl_i\\ N\end{bmatrix}_if_i^{(2N+\gl_i)}u_N \in L_1$. 
 Hence we have $u_N \in q_i^{N(N+\gl_i)}L_1$ by~\eqref{eq:invariance_of_prepolarization}, and the induction begins.
 Next let $k_0$ be an integer such that $\max(0,-\gl_i)\leq k_0<N$. 
 By~\eqref{eq:f_iu} with $n=k_0+\gl_i$, we have
 \begin{equation}\label{eq:inL1}
  \sum_{k=k_0}^N \begin{bmatrix} k+k_0+\gl_i\\ k\end{bmatrix}_i f_i^{(k+k_0+\gl_i)}u_k \in L_1.
 \end{equation}
 It is easily checked from the admissibility that $f_i^{(k+k_0+\gl_i)}u_k$'s are pairwise orthogonal with respect to the polarization, 
 and then it follows from~\eqref{eq:inL1} that $f_i^{(2k_0+\gl_i)}u_{k_0} \in q^{k_0(k_0+\gl_i)}L_1$,
 since the induction hypothesis implies for $k>k_0$ that
 \[ \begin{bmatrix} k+k_0+\gl_i\\ k\end{bmatrix}_i f_i^{(k+k_0+\gl_i)}u_k \in q_i^{k(k-k_0)}L_1 \subseteq q_sL(\gL).
 \]
 Hence $u_{k_0} \in q_i^{k_0(k_0+\gl_i)}L_1$ holds by~\eqref{eq:invariance_of_prepolarization}, as required.

 Now assume that $0< n  \leq N$.
 It follows from~\eqref{eq:commutation4} that
 \begin{equation}\label{eq:eiu}
  e_i^{(n)} u = \sum_{k=\max(0,-\gl_i)}^N \begin{bmatrix} k+n+\gl_i \\ n \end{bmatrix}_i f_i^{(k-n)}u_k,
 \end{equation}
 and since we have
 \[ \begin{bmatrix} k+n+\gl_i \\ n \end{bmatrix}_if_i^{(k-n)}u_k \begin{cases} \in q_i^{(k-n)(k+\gl_i)}L_1 & (k \geq n),\\
                               = 0 & (\text{otherwise})\end{cases}
 \] 
 by the above argument,~\eqref{eq:eiu} and the pairwise orthogonality of $f_i^{(l)}u_k$'s imply $e_i^{(n)} u \in L_1$.
 Since $e_i^{(n)}u = 0$ for $n>N$, this completes the proof. 
\end{proof}

\begin{Lem}\label{Lem:1}
 Let $\bm{p}=(p_1,p_2,p_3,p_4) \in \Z_{\geq 0}^4$. In $V(-\ell\gL_0)$, we have the following: \\
 {\normalfont(1)} For any $1\leq k \leq L$, we have 
 \[  f_{i_{k}}E_{\bm{i}[k-1,0]}^{(p_3)}e_1^{(p_2)}e_0^{(p_1)}v_{-\ell\gL_0}=e_{i_k}
  E_{\bm{i}[k,0]}^{(p_3)}e_1^{(p_2)}e_0^{(p_1)}v_{-\ell\gL_0}=0.
 \]
 {\normalfont(2)} For any $i \in I_{01}$ such that $\langle h_i,\theta_1\rangle =0$, we have $f_iE_{\bm{i}}^{(p_3)}e_1^{(p_2)}e_0^{(p_1)}v_{-\ell\gL_0} = 0$.\\
 {\normalfont(3)} For any $1\leq k \leq L'$, we have
 \[  f_{j_{k}}E_{\bm{j}[k-1,0]}^{(p_4)}E_{\bm{i}}^{(p_3)}e_1^{(p_2)}e_0^{(p_1)}v_{-\ell\gL_0} =
    e_{j_{k}}E_{\bm{j}[k,0]}^{(p_4)}E_{\bm{i}}^{(p_3)}e_1^{(p_2)}e_0^{(p_1)}v_{-\ell\gL_0}=0.
 \]
 {\normalfont(4)} For any $i \in J$ such that $\langle h_i,\theta_J\rangle =0$, we have $f_iE_{\bm{j}}^{(p_4)}E_{\bm{i}}^{(p_3)}e_1^{(p_2)}e_0^{(p_1)}=0$.
\end{Lem}

\begin{proof}
 Set $v=e_1^{(p_2)}e_0^{(p_1)}v_{-\ell\gL_0}$ and $\gL = \wt_{P}(v)=-\ell\gL_0+p_1\ga_0+p_2\ga_1$.

 (1) We have
  \[ s_{\bm{i}[k-1,1]}^{-1}\wt_P(f_{i_{k}}E_{\bm{i}[k-1,0]}^{(p_3)}v) = \gL +p_3\ga_2 -s_{\bm{i}[k-1,1]}^{-1}(\ga_{i_k}),
  \]
  and since $s_{\bm{i}[k-1,1]}^{-1}(\ga_{i_k})$ is a positive root in $R_{I_0\setminus \{1,2\}}$ by Lemma~\ref{Lem:fundamental_properties} (1) and (4),
  the right-hand side does not belong to $-\ell\gL_0 + Q^+$. 
  Hence $f_{i_{k}}E_{\bm{i}[k-1,0]}^{(p_3)}v=0$ holds. Since $\langle h_{i_k}, \wt(E_{\bm{i}[k-1,0]}^{(p_3)}v)\rangle = -c_{\fg}p_3$, 
  we also have $e_{i_k}^{(c_{\fg}p_3+1)}E_{\bm{i}[k-1,0]}^{(p_3)}v=0$, and the proof of (1) is complete.
 
 (2) We have
  \begin{equation}\label{eq:orbit}
   s_{\bm{i}[L,1]}^{-1}\wt(f_iE_{\bm{i}}^{(p_3)}v)= \gL +p_3\ga_2 -s_{\bm{i}[L,1]}^{-1}(\ga_i),
  \end{equation}
  and $s_{\bm{i}[L,1]}^{-1}(\ga_i) \in R_1^+$ by Lemma~\ref{Lem:fundamental_properties} (5). 
  Moreover, we have $s_{\bm{i}[L,1]}^{-1}(\ga_i) \neq \ga_2$ since $\ga_i \neq \theta_1$, and hence the right-hand side of~\eqref{eq:orbit} does not 
  belong to $-\ell\gL_0+Q^+$, which implies (2).

 (3) Set $W=U_q(\fg_J)E_{\bm{i}}^{(p_3)}v$.
  The assertion (2), together with Lemma~\ref{Lem:fundamental_properties} (2), implies that $W_\gl = 0$ unless 
  $\gl \in \gL + p_3\theta_1 + Q^+$.
  Using this, the assertion (3) is proved by a similar argument to that of (1).
  Finally the proof of the assertion (4) is similar to that of (2).   
\end{proof}



\begin{Lem}\label{Lem:belonging_to_B}
 Let $\ell \in \Z_{>0}$.\\
 {\normalfont(1)} For any $(p_1,\ldots,p_5) \in \Z_{\geq 0}^{5}$, the vector
  \[ e_2^{(p_5)}E_{\bm{j}}^{(p_4)}E_{\bm{i}}^{(p_3)}e_1^{(p_2)}e_0^{(p_1)}v_{-\ell\gL_0}
  \]
  in $V(-\ell\gL_0)$ belongs to $\pm\bB(-\ell\gL_0) \cup \{0\}$.\\
 {\normalfont(2)} For any $\bm{p} = (p_1,\ldots,p_5) \in \Z_{\geq 0}^{5}$, $E^{\bm{p}}v_{-\ell \gL_0} \in V(-\ell\gL_0)$
  belongs to $\pm\bB(-\ell\gL_0) \cup \{0\}$. 
\end{Lem}

\begin{proof}
  Obviously,
  \[ f_0v_{-\ell\gL_0} = f_1e_0^{(p_1)}v_{-\ell\gL_0}= f_2e_1^{(p_2)}e_0^{(p_1)}v_{-\ell\gL_0}=0
  \] 
  holds.
  Then the assertion (1) is proved by applying Lemma~\ref{Lem:being_global_basis} (2) repeatedly using Lemma~\ref{Lem:1}.
  For any $n>0$, it is easily seen using~\eqref{eq:commutation4} that
  \[ f_1^{(n)}e_2^{(p_4)}E_{\bm{j}}^{(p_3)}E_{\bm{i}}^{(p_2)}e_{10}^{(p_1)}v_{-\ell\gL_0} 
     = e_2^{(p_4)}E_{\bm{j}}^{(p_3)}E_{\bm{i}}^{(p_2)}e_1^{(p_1-n)}e_{0}^{(p_1)}v_{-\ell\gL_0},
  \]
  which belongs to $\pm\bB(-\ell\gL_0) \cup \{0\}$ by (1).
  Hence it follows from Lemma~\ref{Lem:being_global_basis} that $E^{\bm{p}}v_{-\ell\gL_0} \in \pm\bB(-\ell\gL_0)\cup \{0\}$.
  The assertion (2) is proved.
\end{proof}

Now we prove the following.

\begin{Prop}\label{Prop:being_global_basis_of_level0}
 Let $\ell \in \Z_{>0}$. 
 For any $\bm{p} \in \Z_{\geq 0}^5$, the vector $E^{\bm{p}}v_{\ell\varpi_2} \in V(\ell\varpi_2)$ belongs to $\pm \bB(\ell\varpi_2) \cup \{0\}$.
\end{Prop}

\begin{proof}
 By Lemma~\ref{Lem:belonging_to_B} (2) and~\eqref{eq:gl.base_of_hwlw}, we have
 \[ v_{\ell\gL_2} \otimes E^{\bm{p}}v_{-3\ell\gL_0} \in \pm\bB(\ell\gL_2,-3\ell\gL_0) \cup \{0\},
 \]
 and then Lemma~\ref{Lem:hw-lw-lev0} implies that $E^{\bm{p}}v_{\ell\varpi_2} \in \pm\bB(\ell\varpi_2) \cup \{0\}$ as required,
 since $\varpi_2 = \gL_2-3\gL_0$.
\end{proof}

Next we will show that $\normsq{E^{\bm{p}}v_{\ell\varpi_2}}_{V(\ell\varpi_2)} = \normsq{E^{\bm{p}}w_1^{\otimes \ell}}_{(W^1)^{\otimes\ell}}$ for $\bm{p} \in \Z_{\geq 0}^5$.
Before doing that we prepare a lemma,
which is also used in the next subsection.

\begin{Lem}\label{Lem:coproduct_rule}
 Let $M_1,\ldots,M_n$ be integrable $U_q'(\fg)$-modules, $\bm{\gl}=(\gl_1,\ldots,\gl_n)$ an $n$-tuple of elements of $P_\cl$,
 and $u_k \in \big(M_k)_{\gl_k}$ $(1\leq k \leq n)$.
 Assume that each $u_k$ satisfies $e_iu_k=0$ for $i \in I_0$.
 Then for any $\bm{p} \in \Z_{\geq 0}^5$, the vector $E^{\bm{p}}(u_1\otimes \cdots \otimes u_n) 
 \in M_1 \otimes \cdots \otimes M_n$ can be written in the form
 \begin{equation}\label{eq:coproduct_rule}
  E^{\bm{p}}(u_1\otimes \cdots \otimes u_n)
   =\sum_{\begin{smallmatrix}\bm{p}_1,\ldots,\bm{p}_n \in \Z_{\geq 0}^5; \\ \bm{p}_1+\cdots+\bm{p}_n=\bm{p}\end{smallmatrix}}
    q^{m(\bm{p}_1,\ldots,\bm{p}_n:\bm{\gl})}E^{\bm{p}_1}u_1\otimes \cdots \otimes E^{\bm{p}_n}u_n,
 \end{equation}
 where $m(\bm{p}_1,\ldots,\bm{p}_n:\bm{\gl}) \in D^{-1}\Z$ are certain numbers depending only on $\bm{p}_1,\ldots,\bm{p}_n$ and $\bm{\gl}$.
\end{Lem}

\begin{proof} 
 By the definition of the coproduct, $E^{\bm{p}}(u_1\otimes \cdots \otimes u_n)$ is a sum of vectors of the form
 \begin{equation}\label{eq:vector}
  q^m\bigotimes_{k=1}^n e_1^{(s_k)}e_2^{(r_k)}e_{j_{L'}}^{(h_{kL'})}\cdots e_{j_0}^{(h_{k0})}
  e_{i_{L}}^{(g_{kL})}\cdots e_{i_0}^{(g_{k0})}e_1^{(b_k)}e_0^{(a_k)}u_k.
 \end{equation}
 Since $e_1^{(b_k)}e_0^{(a_k)}u_k = 0$ if $b_k>a_k$ by~\eqref{eq:commutation3} and $\sum_k a_k = \sum_k b_k = p_1$,
 the vector~\eqref{eq:vector} becomes $0$ unless $a_k = b_k$ for all $k$.

 Take a sufficiently large positive integer $\ell$.
 For any $k$, there is a $U_q(\fn_+)$-module homomorphism from $V(-\ell\gL_0)$ to $M_k$ mapping $v_{-\ell\gL_0}$ to $u_k$, which follows from the well-known fact that
 $V(-\ell\gL_0)$ is generated by $v_{-\ell\gL_0}$ as a $U_q(\fn_+)$-module with relations
 \[ e_0^{\ell+1}v_{-\ell\gL_0}= 0 \ \ \text{and} \ \ e_iv_{-\ell\gL_0}=0 \ \ (i \in I_0).
 \]
 Then since $\sum_k g_{kt} = 2p_2$ if $\fg$ is of type $F_4^{(1)}$ and $t\neq 0$ and $\sum_k g_{kt} =p_2$ otherwise,
 we see from Lemma~\ref{Lem:1} (1) that the vector~\eqref{eq:vector} becomes $0$ unless 
 $c_{\fg}g_{k0}=g_{k1}=\cdots =g_{kL}$ for all $k$.
 By a similar argument using Lemma~\ref{Lem:1} (3), we also see that the vector~\eqref{eq:vector} with 
 $c_{\fg}g_{k0}=g_{k1}=\cdots =g_{kL}$ 
 becomes $0$ unless $c_{\fg}h_{k0}=h_{k1}=\cdots=h_{kL'}$ for all $k$. The proof is complete.
\end{proof}

\begin{Prop}\label{Prop:equal_of_prepolarization}
 Let $\ell \in \Z_{>0}$ and $\bm{p} \in \Z_{\geq 0}^5$.\\
 {\normalfont(1)} We have $\normsq{E^{\bm{p}}v_{\ell\varpi_2}}_{V(\ell\varpi_2)} = \normsq{E^{\bm{p}}w_1^{\otimes \ell}}_{(W^1)^{\otimes \ell}}$.\\
 {\normalfont(2)} If $E^{\bm{p}}w_1^{\otimes \ell} \neq 0$, we have $\normsq{E^{\bm{p}}w_1^{\otimes \ell}} \in 1+q_sA$.
\end{Prop}

\begin{proof}
(1) First we show the following:
\begin{equation}
\label{eq:norm_equal_claim}
\normsq{E^{\bm{p}}v_{\ell\varpi_2}}_{V(\ell\varpi_2)} = \normsq{E^{\bm{p}}v_{\varpi_2}^{\otimes \ell}}_{V(\varpi_2)^{\otimes \ell}} \text{ for }
 \bm{p} \in \Z_{\geq 0}^5.
\end{equation}

 By~\cite{MR2074599}, there exists an injective $U_q(\fg)$-module homomorphism $\Phi$ from $V(\ell\varpi_2)$ to 
 $V(\varpi_2)^{\otimes \ell}$ mapping $v_{\ell\varpi_2}$ to $v_{\varpi_2}^{\otimes \ell}$.
 Although $\Phi$ does not preserve the values of the polarizations in general, 
 the relations between $(\ ,\ )_{V(\ell\varpi_2)}$ and $(\ ,\ )_{V(\varpi_2)^{\otimes \ell}}$ are explicitly described in [\textit{loc.\ cit.}], which we recall here.
 Define a $\Q(q_s)[t^{\pm 1}]$-valued bilinear form $(\!(\ ,\ )\!)_t$ on $V(\varpi_2)$ by
 \[ \(u,v\)_t = \sum_{k\in \Z} t^m (z_2^{-m}u,v)_{V(\varpi_2)},
 \] 
 where $z_2$ is the automorphism on $V(\varpi_2)$ in Subsection~\ref{Subsection:fusion}.
 Define a $\Q(q_s)[t_1^{\pm 1},\ldots,t_\ell^{\pm 1}]$-valued bilinear form $\(\ ,\ \)$ on $V(\varpi_2)^{\otimes \ell}$ by
 \[
\left(\!\!\!\left( \bigotimes_{k=1}^{\ell} u_k, \bigotimes_{k=1}^{\ell} v_k \right)\!\!\!\right) = \prod_{k=1}^\ell \(u_k,v_k\)_{t_k}.
 \]
 Then by \cite[Proposition 4.10]{MR2074599}, it holds for $u,v \in V(\ell\varpi_2)$ that
 \begin{equation}\label{eq:prepolarization_in_tensor}
  (u,v) = \frac{1}{\ell!}\left[\(\Phi(u),\Phi(v)\)\prod_{k\neq m} (1-t_kt_m^{-1})\right]_1,
 \end{equation}
 where $[f]_1$ denotes the constant term in $f$.

 For $\bm{p},\bm{p}'\in \Z_{\geq 0}^5$ such that $\bm{p}\neq \bm{p}'$, we have $(E^{\bm{p}}w_1,E^{\bm{p}'}w_1)_{W^1} =0$ by Lemma~\ref{Lem:orthogonality}.
 Then by~\eqref{eq:invariance_of_zk}, this, together with the weight consideration, implies
 \[ (z_2^{-m}E^{\bm{p}}v_{\varpi_2},E^{\bm{p}'}v_{\varpi_2})_{V(\varpi_2)} = 0 \ \ \text{unless } \bm{p} = \bm{p}' \text{ and } m=0.
 \]
 Hence in particular, it follows that 
 \begin{equation}\label{eq:equality_between_tensor}
 \(E^{\bm{p}}v_{\varpi_2}, E^{\bm{p}'}v_{\varpi_2}\)_t = (E^{\bm{p}}v_{\varpi_2}, E^{\bm{p}'}v_{\varpi_2})_{V(\varpi_2)} \ \ \ \text{for } \bm{p},\bm{p}'
  \in \Z_{\geq 0}^5, 
 \end{equation}
 which implies 
  $\(E^{\bm{p}}v_{\varpi_2}^{\otimes \ell}, E^{\bm{p}'}v_{\varpi_2}^{\otimes \ell}\) = (E^{\bm{p}}v_{\varpi_2}^{\otimes \ell},
    E^{\bm{p}'}v_{\varpi_2}^{\otimes \ell})_{V(\varpi_2)^{\otimes \ell}}$
 by Lemma~\ref{Lem:coproduct_rule}.
 Now Equation~\eqref{eq:prepolarization_in_tensor} implies for $\bm{p} \in \Z_{\geq 0}^5$ that
 \begin{align*}
  \normsq{E^{\bm{p}}v_{\ell\varpi_2}}_{V(\ell\varpi_2)} &= \frac{1}{\ell!}\left[\(E^{\bm{p}}v_{\varpi_2}^{\otimes \ell},E^{\bm{p}}v_{\varpi_2}^{\otimes \ell}\)
  \prod_{k\neq m} (1-t_kt_m^{-1})\right]_1\\
  &= \normsq{E^{\bm{p}}v_{\varpi_2}^{\otimes \ell}}_{V(\varpi_2)^{\otimes \ell}} \cdot \frac{1}{\ell!}\left[\prod_{k\neq m} (1-t_kt_m^{-1})\right]_1
   = \normsq{E^{\bm{p}}v_{\varpi_2}^{\otimes \ell}}_{V(\varpi_2)^{\otimes \ell}},
 \end{align*}
 and the claim~\eqref{eq:norm_equal_claim} is proved.

 In order to verify the assertion (1), by~\eqref{eq:norm_equal_claim} it suffices to show $\normsq{E^{\bm{p}}v_{\varpi_2}^{\otimes \ell}}_{V(\varpi_2)^{\otimes \ell}}
 =\normsq{E^{\bm{p}}w_1^{\otimes \ell}}_{(W^1)^{\otimes \ell}}$ for 
 $\bm{p} \in \Z_{\geq 0}^5$.
 We see from~\eqref{eq:invariance_of_zk} that
 \[ (p(u),p(v)) = \(u,v\)_t\big|_{t=1} \ \ \ \text{for } u,v \in V(\varpi_2).
 \] 
 Hence by~\eqref{eq:equality_between_tensor}, we have
 \[ (E^{\bm{p}}w_1,E^{\bm{p}'}w_1)_{W^1} = (E^{\bm{p}}v_{\varpi_2},E^{\bm{p}'}v_{\varpi_2})_{V(\varpi_2)} \ \ \ \text{for } \bm{p},\bm{p}'\in \Z_{\geq 0}^5,
 \]
 and then $\normsq{E^{\bm{p}}v_{\varpi_2}^{\otimes \ell}}_{V(\varpi_2)^{\otimes \ell}} = \normsq{E^{\bm{p}}w_1^{\otimes \ell}}_{(W^1)^{\otimes \ell}}$ 
 follows by Lemma~\ref{Lem:coproduct_rule}.
 The assertion (1) is proved.

 (2) Since $(\ , \ )_{(W^1)^{\otimes \ell}}$ is positive definite, $v \in (W^{1})^{\otimes \ell}$ satisfies 
 $\normsq{v} = 0$ if and only if $v=0$.
 Hence the assertion (2) follows from (1), Proposition~\ref{Prop:being_global_basis_of_level0} and Proposition~\ref{Prop:Nakajima_prepolarization}~(3).
\end{proof}

\begin{Prop}\label{Prop:being_global_basis_for_e^pv}
 Let $\ell \in \Z_{>0}$.
 If $\bm{p} \in \ol{S}_\ell$, then $E^{\bm{p}}w_1^{\otimes \ell} \neq 0$, and hence $\normsq{E^{\bm{p}}w_1^{\otimes \ell}} \in 1+q_sA$ follows 
 from Proposition~\ref{Prop:equal_of_prepolarization}.
\end{Prop}

\begin{proof}
 Let us prove the assertion by the induction on $\ell$.
 First assume that $\ell =1$. 
 In this case, we have
 \[ \ol{S}_1 = \{\bm{0}, \bm{\gee}_1, \bm{\gee}_1+\bm{\gee}_2, (2,1,1,1,1)\}.
 \]
 If $\bm{p} \in \ol{S}_1 \setminus \{\bm{\gee}_1+\bm{\gee}_2\}$, $E^{\bm{p}}v_{\varpi_2} \neq 0$ is checked from the following elementary 
 fact: for an integrable $U_q'(\fg)$-module $M$, $\gl \in P_\cl$ and $i \in I$,
 \begin{equation}\label{eq:elementary_fact}
  \text{if $u \in M_\gl\setminus \{0\}$, then $e^{(k)}_iu \neq 0$ for $0\leq k\leq -\langle h_i,\gl\rangle$}.
 \end{equation}
 On the other hand, by Proposition~\ref{Prop:properties_of_KR} we have 
 \begin{align*}
  \normsq{e_2e_1e_0w_1}&=q^{-1}(e_1e_0w_1,f_2e_2e_1e_0w_1)=q^{-1}(e_1e_0w_1,e_2e_1e_0f_2w_1)\\
  &=\normsq{e_1e_0f_2w_1} = q^{-1}(e_0f_2w_1,e_1e_0f_1f_2w_1)\\ 
  &=\normsq{e_0f_1f_2w_1} = q(f_1f_2w_1,f_0e_0f_1f_2w_1)=q[2]\normsq{f_1f_2w_1}=q[2].
 \end{align*}
 Hence we have $e_2e_1e_0w_1 \neq 0$, and then $E^{\bm{\gee}_1+\bm{\gee}_2}w_1 \neq 0$ is proved by applying~\eqref{eq:elementary_fact}.
 Thus the case $\ell=1$ is proved.

 Assume $\ell > 1$.
 By Lemma~\ref{Lem:coproduct_rule}, $E^{\bm{p}}w_1^{\otimes \ell}$ can be written in the form
 \[ E^{\bm{p}}(w_1 \otimes w_1^{\otimes (\ell-1)}) = \sum_{\bm{p}_1+\bm{p}_2=\bm{p}} 
    q^{m(\bm{p}_1,\bm{p}_2:\varpi_1,(\ell-1)\varpi_1)} E^{\bm{p}_1}w_1 \otimes E^{\bm{p}_2}w_1^{\otimes (\ell-1)},
 \]
 and for the vectors $\{E^{\bm{p}_1}w_1 \mid \bm{p}_1 \in \Z_{\geq 0}^5 \text{ such that } E^{\bm{p}_1}w_1\neq 0\}$ 
 are linearly independent by Lemma~\ref{Lem:orthogonality},
 it is enough to show the existence of $\bm{p}_1$ satisfying
 \begin{equation}\label{eq:nonzeroness_of_two}
  E^{\bm{p}_1}w_1 \neq 0 \ \ \ \text{and} \ \ \ E^{\bm{p}-\bm{p}_1} w_1^{\otimes(\ell-1)} \neq 0.
 \end{equation}    
 If $p_1< p_4 + \ell$, then $\bm{p}_1 = \bm{0}$ satisfies~\eqref{eq:nonzeroness_of_two} by the induction hypothesis since 
 $\bm{p} \in \ol{S}_{\ell-1}$.
 Assume that $p_1 =p_4+\ell$, and set $k_0 = \max\{1\leq k \leq 5\mid p_k \neq 0\}$.
 If $k_0 \neq 2$, set $\bm{p}_1=(\underbrace{2,1,\ldots,1}_{k_0},0,\ldots,0)$.
 That $E^{\bm{p}_1}w_1 \neq 0$ follows from~\eqref{eq:elementary_fact},
 and it is easily checked that $\bm{p}-\bm{p}_1 \in \ol{S}_{\ell-1}$. Therefore~\eqref{eq:nonzeroness_of_two} holds.
 Finally if $k_0 =2$, $\bm{p}_1=(1,1,0,0,0)$ satisfies~\eqref{eq:nonzeroness_of_two}.
 The proof is complete. 
\end{proof}

The following lemma connects values of the prepolarizations on $(W^{1})^{\otimes \ell}$ and $W^\ell$.

\begin{Lem}\label{Lem:equality_of_prepolarizations_on_W}
 Let $\ell \in \Z_{>0}$, and $X,Y \in U_q'(\fg)$.
 Suppose that the images of $X,Y$ under the $\ell$-iterated coproduct $\gD^{(\ell)}\colon U_q'(\fg) \to U_q'(\fg)^{\otimes \ell}$ are written in the forms
 \begin{align*}
  \gD^{(\ell)}(X) &= \sum_{k=1}^{N_1} f_k(q_s)X_{k,1}\otimes X_{k,2}\otimes \cdots \otimes X_{k,\ell}, \ \ \ \text{and}\\
  \gD^{(\ell)}(Y) &= \sum_{m=1}^{N_2} g_m(q_s)Y_{m,1}\otimes Y_{m,2}\otimes \cdots \otimes Y_{m,\ell}
 \end{align*}
 respectively, where $N_1,N_2 \in \Z_{\geq 0}$, $f_k,g_k \in \Q(q_s)$, and $X_{k,j}, Y_{m,j} \in U_q'(\fg)$ are vectors homogeneous with respect to the $Q$-grading.
 We further assume that, for any $1\leq k \leq N_1$ and $1\leq m \leq N_2$,
 \begin{equation}\label{eq:condition_for_equality}
  \text{if} \ \prod_{j=1}^\ell (X_{k,j}w_1,Y_{m,j}w_1)_{W^1} \neq 0,\ \text{then} \ \wt_P(X_{k,j}) =\wt_P(Y_{m,j})
    \ \text{for all} \ 1\leq j \leq \ell.
 \end{equation} 
 Then we have $(Xw_1^{\otimes \ell},Yw_1^{\otimes \ell})_{(W^1)^{\otimes \ell}}=(Xw_\ell,Yw_\ell)_{W^\ell}$.
\end{Lem}

\begin{proof}
 By~\eqref{eq:construction_of_prepolarization}, we have
 \begin{align}\label{eq:epwl}
  (Xw_\ell,Yw_\ell)_{W^\ell} = \Big(X\big(\iota_{\ell-1}(w_1)\otimes \cdots \otimes \iota_{1-\ell}(w_1)\big),
    Y\big(\iota_{1-\ell}(w_1)\otimes\cdots \otimes \iota_{\ell-1}(w_1)\big)\Big)_0.
 \end{align}
 For an arbitrary homogeneous vector $Z\in U_q'(\fg)_\gb$ and $k \in \Z$, we have 
 \[ Z\iota_k(w_1)=q^{k\langle d,\gb\rangle}\iota_k(Zw_1).
 \]
 Hence setting $\wt_P(X_{k,j})= \gb_{k,j}$ and $\wt_P(Y_{m,j}) = \ggg_{m,j}$, it follows that 
 \[ X\big(\iota_{\ell-1}(w_1)\otimes \cdots \otimes \iota_{1-\ell}(w_1)\big)=\sum_k f_k(q_s)q^{\sum_j(\ell+1-2j)\langle d,\gb_{k,j}\rangle}
    \iota(X_{k,1}w_1)\otimes \cdots\otimes \iota(X_{k,\ell}w_1),
 \]
 and
 \[ Y\big(\iota_{1-\ell}(w_1)\otimes \cdots \otimes \iota_{\ell-1}(w_1)\big)=\sum_m g_m(q_s)q^{\sum_j(2j-\ell-1)\langle d,\ggg_{m,j}\rangle}
    \iota(Y_{m,1}w_1)\otimes \cdots\otimes \iota(Y_{m,\ell}w_1).
 \]
 Then we have 
\begin{align*}
\eqref{eq:epwl} & =\sum_{k,m}f_k(q_s)g_m(q_s)q^{\sum_j(\ell+1-2j)\langle d,\gb_{k,j}-\ggg_{m,j}\rangle}\prod_j(X_{k,j}w_1,Y_{m,j}w_1)_{W^1}\\
                 &=\sum_{k,m}f_k(q_s)g_m(q_s)\prod_j(X_{k,j}w_1,Y_{m,j}w_1)_{W^1}=(Xw_1^{\otimes \ell},Yw_1^{\otimes \ell})_{(W^1)^{\otimes \ell}}
 \end{align*}
 by the assumption, and the assertion is proved.
\end{proof}

Now the following proposition, together with Proposition~\ref{Prop:being_global_basis_for_e^pv}, completes the proof of (C2) in Proposition~\ref{Prop:critical_3}.

\begin{Prop}\label{Prop:equality}
 Let $\ell \in \Z_{>0}$. For any $\bm{p}\in \Z_{\geq 0}^5$, we have $\normsq{E^{\bm{p}}w_1^{\otimes \ell}}_{(W^1)^{\otimes \ell}} = \normsq{E^{\bm{p}}w_\ell}_{W^\ell}$.
\end{Prop}

\begin{proof}
 It suffices to show that $X=Y=E^{\bm{p}}$ satisfy the assumptions of Lemma~\ref{Lem:equality_of_prepolarizations_on_W}.
 The vector $\gD^{(\ell)}(E^{\bm{p}})$ can be written in the form $\sum_{k} q_s^{m_k}q^{H_{k1}}E_{k1}\otimes \cdots \otimes  q^{H_{k\ell}}E_{k\ell}$,
 where $m_k \in \Z$, $E_{kj}$ are some products of $e_i^{(m)}$'s and $H_{kj} \in D^{-1}P_\cl^*$.
 By Lemma~\ref{Lem:coproduct_rule}, $q^{H_{k1}}E_{k1}w_1 \otimes \cdots \otimes q^{H_{k\ell}}E_{k\ell} w_1 = 0$ unless $E_{kj} =E^{\bm{p}_j}$ ($1\leq j \leq \ell$) 
 for some $\bm{p}_j \in \Z_{\geq 0}^5$, and then $\prod_j (q^{H_{kj}}E_{kj}w_1,q^{H_{mj}}E_{mj}w_1) \neq 0$ implies 
 $E_{kj}=E_{mj}$ for all $j$ by Lemma~\ref{Lem:orthogonality}.
 Hence~\eqref{eq:condition_for_equality} is obviously satisfied, and the proof is complete.
\end{proof}

\subsection{Proof of (C3) in Proposition~\ref{Prop:critical_3}}\label{Subsection:C3}

First we show the case $i=1$.
The proof is similar to \cite[proof of Eq.\,(3.3) with $i=1$]{naoi2018existence}.
We reproduce it here for the reader's convenience.

\begin{Lem}
 For any $\bm{p} \in S_\ell$, we have
 \[ \normsq{e_1E^{\bm{p}}w_\ell} \in q^{-2\langle h_1,\wt(\bm{p})\rangle}A.
 \]
\end{Lem}

\begin{proof}
 Set
 \[ p=\langle h_1,\wt(\bm{p})\rangle = p_1-p_2-p_3-p_4+2p_5-p_6 \geq 0.
 \]
 We have 
 \begin{align*}
  \normsq{e_1E^{\bm{p}}w_\ell}&=q^{-p-1}(E^{\bm{p}}w_\ell,f_1e_1E^{\bm{p}}w_\ell)
   =q^{-p-1}\Big([-p]\normsq{E^{\bm{p}}w_\ell}
   +(E^{\bm{p}}w_\ell,e_1f_1E^{\bm{p}}w_\ell)\Big)\\
  &\equiv q^{-2p}\normsq{f_1E^{\bm{p}}w_\ell} \ \ \ \text{mod } q^{-2p}A,
 \end{align*}
 where we have used the fact $\normsq{E^{\bm{p}}w_\ell} \in 1+q_sA$ by (C2) (which we have already proved).
 Hence it suffices to show that $\normsq{f_1E^{\bm{p}}w_\ell} \in A$.
 Set $r=3\ell -p_1+p_5$.
 It is easily checked that $f_0^{(k)}f_1E^{\bm{p}-p_6\bm{\gee}_6}w_\ell = 0$ for $k>1$, and hence we have
 \begin{align}\label{eq:calculation1}
  &\normsq{f_1E^{\bm{p}}w_\ell} = q^{p_6(r-p_6-1)}(f_1E^{\bm{p}-p_6\bm{\gee}_6}w_\ell, f^{(p_6)}_0e_0^{(p_6)}f_1E^{\bm{p}-p_6\bm{\gee}_6}w_\ell)\nonumber\\
  &=q^{p_6(r-p_6-1)}\left(\begin{bmatrix}r-1 \\ p_6\end{bmatrix}\normsq{f_1E^{\bm{p}-p_6\bm{\gee}_6}w_\ell}
   + \begin{bmatrix} r-1\\ p_6-1\end{bmatrix}(f_1E^{\bm{p}-p_6\bm{\gee}_6}w_\ell,e_0f_0f_1E^{\bm{p}-p_6\bm{\gee}_6}w_\ell)\right)\nonumber\\
  &\in \normsq{f_1E^{\bm{p}-p_6\bm{\gee}_6}w_\ell}A + q^{2(r-p_6)}\normsq{f_0f_1E^{\bm{p}-p_6\bm{\gee}_6}w_\ell}A.
 \end{align}
 It follows that 
 \begin{align*}
  \normsq{f_1E^{\bm{p}-p_6\bm{\gee}_6}&w_\ell} = q^{p+p_6-1}(E^{\bm{p}-p_6\bm{\gee}_6}w_\ell,e_1f_1E^{\bm{p}-p_6\bm{\gee}_6}w_\ell)\\
  &=q^{p+p_6-1}\Big([p+p_6]\normsq{E^{\bm{p}-p_6\bm{\gee}_6}w_\ell}+(E^{\bm{p}-p_6\bm{\gee}_6}w_\ell,f_1e_1E^{\bm{p}-p_6\bm{\gee}_6}w_\ell)\Big)\\
  &=q^{p+p_6-1}\Big([p+p_6]\normsq{E^{\bm{p}-p_6\bm{\gee}_6}w_\ell}+q^{p+p_6+1}[p_5+1]^2\normsq{E^{\bm{p}+\bm{\gee}_5-p_6\bm{\gee}_6}w_\ell}\Big) \in A.
 \end{align*}
 Moreover, it is easily checked that 
 \[ f_0f_1E^{\bm{p}-p_6\bm{\gee}_6}w_\ell = [3\ell-p_1+1]E^{\bm{p}-\bm{\gee}_1-p_6\bm{\gee}_6},
 \]
 and hence it also follows that 
 \[ q^{2(r-p_6)}\normsq{f_0f_1E^{\bm{p}-p_6\bm{\gee}_6}w_\ell}= q^{2(r-p_6)}[3\ell-p_1+1]^2\normsq{E^{\bm{p}-\bm{\gee}_1-p_6\bm{\gee}_6}w_\ell}\in q^{2(p_5-p_6)}A\subseteq A.
 \]
 Hence $\normsq{f_1E^{\bm{p}}w_\ell} \in A$ follows from~\eqref{eq:calculation1}, and the proof is complete.
\end{proof}

When we show (C3) for $i \in I_0 \setminus \{1\}$, as we did in the proof of (C2),
we may assume that $\bm{p} \in \ol{S}_\ell\big(=S_\ell \cap \Z_{\geq 0}^5\big)$ by the following lemma.

\begin{Lem}
 For any $\bm{p}=(p_1,\ldots,p_6) \in \Z_{\geq 0}^6$ such that $p_1-p_5+p_6\leq 3\ell$ and $i \in I_0\setminus\{1\}$,
 we have $\normsq{e_iE^{\bm{p}}w_\ell} \in (1+qA)\normsq{e_iE^{\bm{p}-p_6\bm{\gee}_6}w_\ell}$. 
\end{Lem}

\begin{proof}
 Since $e_iE^{\bm{p}}w_\ell=e_0^{(p_6)}e_iE^{\bm{p}-p_6\bm{\gee}_6}w_\ell$, the same proof for Lemma~\ref{Lem:reduction_of_(ii)} holds here.
\end{proof}

Our next goal is to give estimates for the values $\normsq{e_iE^{\bm{p}}w_1^{\otimes \ell}}$ ($i\in I_0\setminus \{1\}$).
For this purpose, let us prepare some lemmas.
The proof of the following lemma is almost the same with that of Lemma~\ref{Lem:being_global_basis}, with $L_1$ replaced by $L(\gL)$. 

\begin{Lem}\label{Lem:belonging_to_L0}
 Let $\gL \in P$ and $i \in I$, and assume that $u \in V(\gL)$ is a weight vector.
 If $f_i^{(n)}u \in L(\gL)$ for all $n \in \Z_{\geq 0}$, then $e_i^{(n)}u \in L(\gL)$ for all $n >0$.\\
\end{Lem}

\begin{Lem}\label{Lem:belonging_to_L}
 Let $\gL,\gl \in P$, $i \in I$, and $u \in V(\gL)_\gl$, and assume that
 \[ u \in q^aL(\gL), \ \ and \ \ f_iu\in q^bL(\gL)
 \]
 for some $a,b\in D^{-1}\Z$. Set $r_i=(\ga_i,\ga_i)/2$.\\
 {\normalfont(1)} 
 We have 
 \[ e_iu \in q^{\min(a,b-r_i\langle h_i,\gl\rangle)}L(\gL).
 \]
 {\normalfont(2)} Further assume that $\langle h_i,\gl\rangle \leq 0$ and $f_i^{(2)}u=0$.
 Then  we have
 \[ e_i^{(n)} u\in q^{\min(a,b-r_i(\langle h_i,\gl\rangle+n-1))}L(\gL)\ \ \ \text{for any } n >0.
 \]
\end{Lem}

\begin{proof}
 Set $\gl_i = \langle h_i, \gl\rangle \in \Z$, and write
 \[ u= \sum_{k=\max(0,-\gl_i)}^N f_i^{(k)}u_k, \ \ \ \text{where } u_k \in \ker e_i \cap V(\gL)_{\gl+k\ga_i}.
 \]
 We have 
 \[ f_iu=\sum_{k=\max(0,-\gl_i)}^N [k+1]_if_i^{(k+1)}u_k \in q^bL(\gL),
 \]
 and since $f_i^{(k+1)}u_k$'s are pairwise orthogonal with respect to $(\ ,\ )$, it follows from Proposition~\ref{Prop:Nakajima_prepolarization} (4) that 
 $[k+1]_if_i^{(k+1)}u_k \in q^bL(\gL)$ for every $k$.
 Then since $f_i^{(k+1)}u_k \neq 0$ for $k \geq \max(0,-\gl_i+1)$ such that $u_k\neq 0$, we have 
 \begin{equation}\label{containment_of_uk}
  u_k \in q^{b+r_ik}L(\gL) \ \ \text{for } \max(0,-\gl_i+1) \leq k\leq N
 \end{equation}
 by  Proposition~\ref{Prop:Nakajima_prepolarization} (2).
 We have 
 \[ e_iu=\sum_{k=\max(1,-\gl_i)}^N[k+\gl_i+1]_if_i^{(k-1)}u_k,
 \]
 and hence if $\gl_i \geq 0$, \eqref{containment_of_uk} implies $e_iu \in q^{b-r_i\gl_i}L(\gL)$ and the assertion (1) holds.
 When $\gl_i <0$, we need to show further that 
 \begin{equation}\label{eq:further}
  f_i^{(-\gl_i-1)}u_{-\gl_i} \in q^{\min(a,b-r_i\gl_i)}L(\gL).
 \end{equation}
 Similarly as above, we see that $u \in q^aL(\gL)$ implies $u_k \in q^aL(\gL)$ for all $k$,
 and hence \eqref{eq:further} follows. 
 The proof of (1) is complete.
 Under the assumption of (2), we may put $N = -\gl_i+1$ and we have 
 $e_i^{(n)}u=f_i^{(-\gl_i-n)}u_{-\gl_i}+[n+1]_if_i^{(-\gl_i+1-n)}u_{-\gl_i+1}$, which belongs to 
 $q^{\min(a,b-r_i(\gl_i+n-1))}L(\gL)$. 
 Hence (2) is also proved.
\end{proof}

\begin{Lem}\label{Lem:setminus_J}
 Assume that the sequence $\bm{i}$ satisfies the following condition: there exists $1\leq m \leq L$ such that
 $i_m,i_{m+1},\ldots,i_L$ are pairwise distinct, $c_{2i_m} <0$, and $c_{i_{k}i_{k+1}} = -1$ for $m \leq k \leq L-1$.\footnote{
 If $\fg$ is not of type $E_6^{(1)}$, this condition, together with the condition~\eqref{equation:condition_for_i} on $\bm{i}$,
 uniquely determines the sequence $(i_L, i_{L-1}, \dotsc, i_m)$ (see Figure~\ref{figure} and Table~\ref{table}).
 In type $E_6^{(1)}$, on the other hand, there are two possibilities; $(5, 3)$ or $(6, 4)$.}
 Let $\ell \in \Z_{>0}$ and $(p_1,p_2,p_3,p_4) \in \Z_{\geq 0}^4$, and set
 \[ v_k = f_{i_k}f_{i_{k+1}}\cdots f_{i_L}E_{\bm{i}}^{(p_3)}e_1^{(p_2)}e_{0}^{(p_1)}v_{-\ell\gL_0}\in V(-\ell\gL_0)
    \ \ \ \text{for} \ m\leq k \leq L.
 \]
 {\normalfont(1)} We have
 \[ v_k = e_{\bm{i}[L,k]}^{(c_\fg p_3-1)}E_{\bm{i}[k-1,0]}^{(p_3)}e_1^{(p_2)}e_{0}^{(p_1)}v_{-\ell\gL_0}
    \ \ \ \text{for $m \leq k \leq L$}.
 \]
 {\normalfont (2)} We have $v_k \in  \pm \bB(-\ell\gL_0)\cup\{0\}$ for $m\leq k \leq L$.\\
 {\normalfont (3)} If $\fg$ is not of type $E_6^{(2)}$, we have
  \begin{equation}\label{eq:vector_in_B} 
   E_{\bm{i}[k-1,m]}^{(p_4)}e_2^{(p_4)}v_k \in \pm \bB(-\ell\gL_0) \cup \{0\} \ \ \ \text{for $m \leq k \leq L$}.
  \end{equation}
  On the other hand if $ \fg$ is of type $E_6^{(2)}$, we have
  \[ E_{\bm{i}[k-1,m]}^{(p_4)}e_2^{(p_4)}v_k \in q^{\min(0,p_3-p_4-1)}L(-\ell\gL_0) \ \ \ \text{for $m \leq k \leq L$}.
  \]
\end{Lem}

\begin{proof}
 The assertion (1) is easily proved using~\eqref{eq:commutation4} and Lemma~\ref{Lem:1} (1).

 (2)  Set $v=e_{1}^{(p_2)}e_0^{(p_1)}v_{-\ell\gL_0}$ and $\gL = \wt_P(v)=-\ell\gL_0+p_1\ga_0+p_2\ga_1$,
 and fix $m\leq k \leq L$.
 By (the proof of) Lemma~\ref{Lem:belonging_to_B}, we have 
 \begin{equation}\label{eq:eikcp}
  e_{i_{k}}^{(c_\fg p_3-1)}E_{\bm{i}[k-1,0]}^{(p_3)}v \in \pm \bB(-\ell\gL_0) \cup \{0\}.
 \end{equation}
 For each $k < k' \leq L$, we have
 \begin{align}\label{eq:orbit2}
  s_{\bm{i}[k'-1,1]}^{-1}\wt_P(f_{i_{k'}}e_{\bm{i}[k'-1,k]}^{(c_\fg p_3-1)}E_{\bm{i}[k-1,0]}^{(p_3)}v) 
  &= \gL + p_3\ga_2 - s_{\bm{i}[k'-1,1]}^{-1}(\ga_{i_{k}}+\cdots+\ga_{i_{k'}}) \nonumber \\
  &= \gL +p_3\ga_2-s_{\bm{i}[L,1]}^{-1}(\ga_{i_{k}}+\cdots+\ga_{i_{k'-1}})
 \end{align}
 by the assumption on $\bm{i}$.
 We have $\langle h_{i_L},\theta_1\rangle >0$ by the condition~\eqref{equation:condition_for_i} on $\bm{i}$,
 and then it is easily checked that $\langle h_{i_r},\theta_1\rangle=0$ for $m\leq r\leq L-1$ (see Figure~\ref{figure} and Table~\ref{table}).
 Hence~\eqref{eq:orbit2} does not belong to $-\ell\gL_0+Q^+$ by Lemma~\ref{Lem:fundamental_properties} (5), 
 which implies $f_{i_{k'}}e_{\bm{i}[k'-1,k]}^{(c_\fg p_3-1)}E_{\bm{i}[k-1,0]}^{(p_3)}v=0$ for all $k'$.
 Now the assertion (2) follows from~\eqref{eq:eikcp} by applying Lemma~\ref{Lem:being_global_basis} (2) repeatedly.

 (3) First assume that $\fg$ is not of type $E_6^{(2)}$.
 We shall prove the assertion by the induction on $k$.
 In the case $k=m$, since $v_m \in \pm \bB(-\ell\gL_0) \cup \{0\}$ by (2) it suffices to show that $f_2v_m=0$, and as above, this is done by checking 
 $s_{\bm{i}[L,1]}^{-1}\wt_P(f_2v_m) \notin -\ell\gL_0+Q^+$. 
 Hence the induction begins.
 Assume that $k > m$.
 It follows from Lemma~\ref{Lem:1} (2) that
 \begin{align*}
  f_2v_k = f_{i_k}\cdots f_{i_L}f_2 E_{\bm{i}}^{(p_3)}v =0,
 \end{align*}
 and
 \begin{align*}
  f_{i_{k'}}E_{\bm{i}[k'-1,m]}^{(p_4)}e_2^{(p_4)}v_k =f_{i_{k}}\cdots f_{i_L}E_{\bm{i}[k'-1,m]}^{(p_4)}e_2^{(p_4)}f_{i_{k'}}E_{\bm{i}}^{(p_3)}v=0 \ \
  \text{for any $m \leq k' \leq k-2$}.
 \end{align*}
 Hence we have $E_{\bm{i}[k-2,m]}^{(p_4)}e_2^{(p_4)}v_k\in \pm \bB(-\ell\gL_0) \cup \{0\}$.
 Since
 \[ f_{i_{k-1}}^{(p)}E_{\bm{i}[k-2,m]}^{(p_4)}e_2^{(p_4)}v_k = \gd_{p1} E_{\bm{i}[k-2,m]}^{(p_4)}e_2^{(p_4)}v_{k-1}  \ \ \ \text{for $p \in \Z_{>0}$},
 \]
~\eqref{eq:vector_in_B} is now proved from the induction hypothesis and Lemma~\ref{Lem:being_global_basis}.

 Next assume that $\fg$ is of type $E_6^{(2)}$. In this case $\bm{i}=(432)$, $L=2$, $m=1$ and
 \[ v_k= e_4^{(p_3-1)}e_3^{(p_3-\gd_{k1})}e_2^{(p_3)}e_{1}^{(p_2)}e_0^{(p_1)}v_{-\ell\gL_0} \ \ \ (k=1,2).
 \] 
 We have
 \[ f_2^{(p)}v_1 = \begin{cases} [p_2-p_3+1] e_{432}^{(p_3-1)}e_{1}^{(p_2)}e_0^{(p_1)}v_{-\ell\gL_0} \in q^{-p_2+p_3}L(-\ell\gL_0) & (p=1), \\ 0 & (p \in \Z_{>1})
    \end{cases}
 \]
 (note that $v_1=0$ if $p_3>p_2$),
 and hence it follows from Lemma~\ref{Lem:belonging_to_L} (2) that $e_2^{(p_4)}v_1 \in q^{\min(0,p_3-p_4-1)}L(-\ell\gL_0)$.
 On the other hand, since $f_2v_2 =0$ we have $e_2^{(p_4)}v_2 \in \pm \bB(-\ell\gL_0) \cup \{0\}$,
 and then $e_3^{(p_4)}e_2^{(p_4)}v_2 \in q^{\min(0,p_3-p_4-1)}L(-\ell\gL_0)$ also follows since
 $f_3^{(p)}e_2^{(p_4)}v_2 = \gd_{p1}e_2^{(p_4)}v_1$ for $p \in \Z_{>0}$. 
 The proof is complete. 
\end{proof}

\begin{Lem}\label{Lem:belonging_to_L_of_e} 
 Let $\ell \in \Z_{>0}$ and $\bm{p}=(p_1,\ldots,p_5) \in \Z_{\geq 0}^5$.\\
 {\normalfont(1)} We have  
  \[ e_2E^{\bm{p}}v_{-\ell\gL_0} \in q^{\min(0,-p_4+p_5)}L(-\ell\gL_0).
  \]
 {\normalfont(2)} If $\fg$ is not of type $E_6^{(2)}$ and $i \in I_0\setminus \{1,2\}$, we have
  \[ e_iE^{\bm{p}}v_{-\ell\gL_0} \in q_i^{\min(0,-\langle h_i,\wt(\bm{p})\rangle)}L(-\ell\gL_0).
  \]
 {\normalfont(3)} If $\fg$ is of type $E_6^{(2)}$, we have
  \begin{align*}
   e_3E^{\bm{p}}v_{-\ell\gL_0} &\in q^{2\min(0,-p_3+p_4)-\gd_{p_3,p_4}}L(-\ell\gL_0), \ \ \text{and}\\
   e_4E^{\bm{p}}v_{-\ell\gL_0} &\in q^{2\min(0,-p_2+p_3)-\gd_{p_2,p_3}}L(-\ell\gL_0).
  \end{align*}
\end{Lem}

\begin{proof}
 (1) It suffices to show that $f_2E^{\bm{p}}v_{-\ell\gL_0} \in q^{-p_1+p_4}L(-\ell\gL_0)$
 by Lemmas~\ref{Lem:belonging_to_B} and \ref{Lem:belonging_to_L}.
 Since $f_2E_{\bm{j}}^{(p_3)}E_{\bm{i}}^{(p_2)}E_{10}^{(p_1)}v_{-\ell\gL_0}=0$ by Lemma~\ref{Lem:1} (4),
 it follows from the weight consideration that $E^{\bm{p}}v_{-\ell\gL_0}=0$ if $p_4>p_1$,
 and hence we may assume that $p_4\leq p_1$. 
 By a direct calculation, we have
 \[ f_2E^{\bm{p}}v_{-\ell\gL_0} = [p_1-p_4+1]E^{\bm{p}-\bm{\gee}_4}v_{-\ell\gL_0},
 \]
 which belongs to $q^{-p_1+p_4}L(-\ell\gL_0)$, as required. 

 (2) It suffices to show that $f_iE^{\bm{p}}v_{-\ell\gL_0} \in L(-\ell\gL_0)$.
  The proof is divided into three cases.
  First assume that $\langle h_i, \theta_1\rangle = \langle h_i,\theta_J\rangle =0$.
  In this case Lemma~\ref{Lem:1} implies $f_iE^{\bm{p}}v_{-\ell\gL_0} =0$, and hence the assertion holds.
  Next assume that $\langle h_i,\theta_J\rangle >0$. 
  By Lemma~\ref{Lem:fundamental_properties} (6), we may assume that the sequence $\bm{j}$ is chosen so that
  $j_{L'} = i$.
  For each $n \in \Z_{\geq 0}$, set
  \[ v_n=e_i^{(c_\fg p_3-1)}E_{\bm{j}[L'-1,0]}^{(p_3)}E_{\bm{i}}^{(p_2)}e_1^{(p_1-n)}e_0^{(p_1)}v_{-\ell\gL_0}.
  \] 
  We easily see using Lemma~\ref{Lem:1} (3) that
  \begin{equation}\label{eq:fi}
   f_iE^{\bm{p}}v_{-\ell\gL_0} = e_1^{(p_5)}e_2^{(p_4)}v_0, \ \ \ \text{and} \ \ \ f_1^{(n)}e_2^{(p_4)}v_0 = e_2^{(p_4)}v_n \ \ \text{for any} \ n \in \Z_{\geq 0}.
  \end{equation}
  Hence by Lemma~\ref{Lem:belonging_to_L0}, it suffices to show that $e_2^{(p_4)}v_n \in \pm \bB(-\ell\gL_0) \cup \{0\} \subseteq L(-\ell\gL_0)$ for any $n$.
  We have $v_n \in \pm \bB(-\ell\gL_0) \cup \{0\}$ by (the proof of) Lemma~\ref{Lem:belonging_to_B}.
  Since $\ga_2 + \ga_i$ is a positive root (see Table~\ref{table}), we have
  \begin{align*}
   s_{\bm{j}[L',1]}^{-1}\wt_P (f_2v_n)
   &=\wt_P(E_{\bm{i}}^{(p_2)}e_1^{(p_1-n)}e_0^{(p_1)}v_{-\ell\gL_0})+p_3\ga_2 - s_{\bm{j}[L',1]}^{-1}(\ga_2+\ga_i)\\
   &\notin \wt_P(E_{\bm{i}}^{(p_2)}e_1^{(p_1-n)}e_0^{(p_1)}v_{-\ell\gL_0}) +Q^+,
  \end{align*}
  and the same argument as in the proof of Lemma~\ref{Lem:1} (3) shows that this implies $f_2v_n =0$.
  Hence $e_2^{(p_4)}v_n\in \pm \bB(-\ell\gL_0) \cup \{0\}$ holds, as required.
  Finally assume that $\langle h_i, \theta_1\rangle >0$.
  We may assume that the sequence $\bm{i}$ is chosen so that $i_{L}=i$, and the assumption of Lemma~\ref{Lem:setminus_J} is satisfied.
  Let $m$ be as in the assumption. 
  Further, we may also assume that the sequence $\bm{j}$ is chosen so that $j_{k}=i_{m+k-1}$ for $1\leq k \leq L-m$.
  For each $n \in \Z_{\geq 0}$, set 
  \[ u_n = f_iE_{\bm{i}}^{(p_2)}e_1^{(p_1-n)}e_0^{(p_1)}v_{-\ell\gL_0}=e_i^{(c_\fg p_2-1)}E_{\bm{i}[L-1,0]}^{(p_2)}e_1^{(p_1-n)}e_0^{(p_1)}v_{-\ell\gL_0}.
  \]
  As above it is enough to show for any $n$ that
  \begin{equation}\label{f1ne2}
   e_2^{(p_4)}E_{\bm{j}}^{(p_3)}u_n\in \pm \bB(-\ell\gL_0) \cup \{0\}.
  \end{equation}
  It follows from Lemma~\ref{Lem:setminus_J} (3) that $E_{\bm{j}[L-m,0]}^{(p_3)}u_n \in \pm \bB(-\ell\gL_0) \cup \{0\}$.
  We easily see from Figure~\ref{figure} and Table~\ref{table} that
  \[ \{ j \in J \mid c_{ij} \neq 0\} = \{ i_{L-1}\}, \ \ \text{and} \ \ \#\{ 1\leq k \leq L' \mid j_k = i_{L-1}\} = 1.
  \]
  Then, since $j_{L-m}=i_{L-1}$, 
  we have $c_{ij_k} = 0$ for $L-m < k \leq L'$, and hence we have
  \[ f_{j_k}E_{\bm{j}[k-1,0]}^{(p_3)}u_n = f_if_{j_k}E_{\bm{j}[k-1,0]}^{(p_3)}E_{\bm{i}}^{(p_2)}e_1^{(p_1-n)}e_0^{(p_1)}v_{-\ell\gL_0}=0 \ \ \ \text{for all $L-m < k \leq L'$}
  \]
  by Lemma~\ref{Lem:1}.
  Similarly, $f_2E_{\bm{j}}^{(p_3)}u_n = 0$ is proved.
  Now~\eqref{f1ne2} is shown using Lemma~\ref{Lem:being_global_basis}, and the proof of (2) is complete.

 (3) We shall prove 
  \begin{equation*}
   f_3E^{\bm{p}}v_{-\ell\gL_0} = e_1^{(p_5)}e_2^{(p_4)}e_3^{(p_3-1)}e_2^{(p_3)}e_{432}^{(p_2)}e_{10}^{(p_1)}v_{-\ell\gL_0} \in q^{\min(0,p_3-p_4-1)}L(-\ell\gL_0),
  \end{equation*}
  which implies the former assertion by Lemma~\ref{Lem:belonging_to_L},
  and for this it is enough to show for any $n \in \Z_{\geq 0}$ that 
  \begin{align}\label{eq:E62}
   f_1^{(n)}e_2^{(p_4)}e_3^{(p_3-1)}e_2^{(p_3)}e_{432}^{(p_2)}e_{10}^{(p_1)}v_{-\ell\gL_0} &
   =e_2^{(p_4)}e_3^{(p_3-1)}e_2^{(p_3)}e_{432}^{(p_2)}e_1^{(p_1-n)}e_0^{(p_1)}v_{-\ell\gL_0} \nonumber\\
   &\in q^{\min(0,p_3-p_4-1)}L(-\ell\gL_0)
  \end{align}
  by Lemma~\ref{Lem:belonging_to_L0}.
  We have 
  \[ e_3^{(p_3-1)}e_2^{(p_3)}e_{432}^{(p_2)}e_1^{(p_1-n)}e_0^{(p_1)}v_{-\ell\gL_0} \in \pm \bB(-\ell\gL_0) \cup \{0\} \subseteq L(-\ell\gL_0),
  \]
  and since 
  \begin{align*}
   f_2^{(p)}&e_3^{(p_3-1)}e_2^{(p_3)}e_{432}^{(p_2)}e_1^{(p_1-n)}e_0^{(p_1)}v_{-\ell\gL_0} \\
     &=\begin{cases} [p_1-n-p_3+1]e_{32}^{(p_3-1)}e_{432}^{(p_2)}e_1^{(p_1-n)}e_0^{(p_1)}v_{-\ell\gL_0}\in q^{-p_1+p_3+n}L(-\ell\gL_0) & (p=1),\\
                   0 & (p \in \Z_{>1})\end{cases}
  \end{align*}
  (note that the left-hand side is $0$ if $p_3>p_1-n$),
 ~\eqref{eq:E62} follows from Lemma~\ref{Lem:belonging_to_L}, as required.
  The latter assertion is proved in a similar manner using Lemma~\ref{Lem:setminus_J}.
\end{proof}

Now we obtain the following estimates for $\normsq{e_iE^{\bm{p}}w_1^{\otimes \ell}}$.

\begin{Prop}\label{Prop:estimate}
 Let $\ell \in \Z_{>0}$ and $\bm{p}=(p_1,\ldots,p_5) \in \Z_{\geq 0}^5$.\\
 {\normalfont(1)} We have
  \[ \normsq{e_2E^{\bm{p}}w_1^{\otimes \ell}} \in q^{2\min(0,-p_4+p_5)}A.
  \]
 {\normalfont(2)} If $i \in I_{0}\setminus\{1,2\}$, we have
  \[ \normsq{e_iE^{\bm{p}}w_1^{\otimes \ell}} \in q_i^{2\min(0,-\langle h_i,\wt(\bm{p})\rangle)-1}A.
  \]
\end{Prop}

\begin{proof}
 By~\eqref{eq:gl.base_of_hwlw}, Lemma~\ref{Lem:hw-lw-lev0}, \cite[Theorem~1 (2)]{MR2074599} and the definition of $L(W^1)$, we have 
 \[ p^{\otimes \ell} \circ \Phi \circ \Psi \big(v_{\ell\gL_2} \otimes L(-3\ell\gL_0)\big) \subseteq L(W^{1})^{\otimes \ell},
 \]  
 where $\Psi\colon V(\ell\gL_2) \otimes V(-3\ell\gL_0) \to V(\ell\varpi_2)$ is the homomorphism given in the lemma,
 $\Phi\colon V(\ell\varpi_2) \hookrightarrow V(\varpi_2)^{\otimes \ell}$ is the one satisfying $\Phi(v_{\ell\varpi_2}) = v_{\varpi_2}^{\otimes \ell}$, 
 and $p\colon V(\varpi_2) \twoheadrightarrow W^{1}$ is the canonical projection.
 The assertions follow from this and Lemma~\ref{Lem:belonging_to_L_of_e}.
\end{proof}



%
%
%

Let $M_1,\ldots,M_n$ and $u_k\in (M_k)_{\gl_k}$ ($1\leq k\leq n$) be as in Lemma~\ref{Lem:coproduct_rule}.
We see that the vector $e_iE^{\bm{p}}(u_1\otimes \cdots\otimes u_n)$ for $i \in I$ and $\bm{p} \in \Z_{\geq 0}^5$
can be written in the form
\begin{equation}\label{eq:coproduct_rule2}
\begin{split}
 e_i&E^{\bm{p}}(u_1\otimes \cdots \otimes u_p) \\
 &=\sum_{k=1}^n\sum_{\begin{smallmatrix}\bm{p}_1,\ldots,\bm{p}_n \in \Z_{\geq 0}^5; \\ \bm{p}_1+\cdots+\bm{p}_n=\bm{p}\end{smallmatrix}}
  q^{m(\bm{p}_1,\ldots,\bm{p}_n:\bm{\gl},i,k)}
  E^{\bm{p}_1}u_1\otimes \cdots \otimes e_iE^{\bm{p}_k}u_k\otimes \cdots \otimes E^{\bm{p}_n}u_n
\end{split}
\end{equation}
with some $m(\bm{p}_1,\ldots,\bm{p}_n:\bm{\gl},i,k) \in D^{-1}\Z$.
Now the following lemma, together with Proposition~\ref{Prop:estimate} (2), completes the proof of (C3) for $i \in I_0\setminus \{1,2\}$.

\begin{Lem}\label{Lem:not12}
 Let $i \in I_0\setminus\{1,2\}$. \\
 {\normalfont(1)} If $\bm{p},\bm{p}' \in \Z_{\geq 0}^5$ satisfy $(e_iE^{\bm{p}}w_1,E^{\bm{p}'}w_1) \neq 0$, then we have $\wt_P(e_iE^{\bm{p}})=\wt_P(E^{\bm{p}'})$.\\
 {\normalfont(2)} For any $\bm{p} \in \Z_{\geq 0}^5$ and $\ell \in \Z_{>0}$, we have
 \[
 \normsq{e_iE^{\bm{p}}w_1^{\otimes \ell}}_{(W^1)^{\otimes \ell}} = \normsq{e_iE^{\bm{p}}w_\ell}_{W^\ell}.
 \]
\end{Lem}

\begin{proof}
 Since $(e_iE^{\bm{p}}w_1,E^{\bm{p}'}w_1) \neq 0$ implies $\wt_P(e_iE^{\bm{p}}) \in \wt_P(E^{\bm{p}'}) + \Z\gd$, in order to prove (1) it is enough to show that 
 $(e_iE^{\bm{p}}w_1,E^{\bm{p}'}w_1) = 0$ if $p_5\neq p_5$.
 Since $e_j,f_j$ ($j=0,1$) commute with $e_i$, this follows from the same argument as in the proof of Lemma~\ref{Lem:orthogonality}.
 Then we see from Lemma~\ref{Lem:orthogonality} and~\eqref{eq:coproduct_rule2} that 
 $X=Y=e_iE^{\bm{p}}$ satisfy the assumptions of Lemma~\ref{Lem:equality_of_prepolarizations_on_W},
 and hence the assertion (2) is proved.
\end{proof}

It remains to prove (C3) for $i = 2$ and $\bm{p} \in \ol{S}_\ell$, which is more involved. 
We will prove the following stronger statement,
and the proof will occupy the rest of this paper.

\begin{Prop}\label{Prop:e2epw}
 Let $\ell \in \Z_{>0}$.
 For any $\bm{p}=(p_1,p_2,p_3,p_4,p_5)\in \Z_{\geq 0}^5$, we have 
 \begin{equation}\label{eq:e2epw}
  \normsq{e_2E^{\bm{p}}w_\ell} \in q^{2\min(0,-p_4+p_5,p_1-p_4-\ell)-1}A.
 \end{equation}
\end{Prop}


\begin{Lem}\label{Lem:properties_of_W^1}
 Let $\ell \in \Z_{>0}$.
 If $\bm{p} \in \Z_{\geq 0}^5$ satisfies $(W^1)^{\otimes \ell} \ni E^{\bm{p}}w_1^{\otimes \ell} \neq 0$, 
 then $p_1 \leq 3\ell$ and $p_j \leq \min(2\ell,p_1)$ for $j \in \{2,3,4\}$. 
\end{Lem}

\begin{proof}
 By Lemma~\ref{Lem:coproduct_rule}, it is enough to show the assertion for $\ell=1$.
 In this case, since $\langle h_0,\varpi_2\rangle =-3$ and $f_0w_1 =0$, $p_1 \leq 3$ follows.
 Moreover, since 
 \[ \langle h_2, \wt(e_{10}^{(p_1)}w_1)\rangle = -p_1+1 \ \  \text{and} \ \ f_2^{(2)}e_{10}^{(p_1)}w_1=0,
 \]
 we have $e_2^{(p_1+1)}e_{10}^{(p_1)}w_1=0$,
 which implies $p_2 \leq p_1$. 
 We easily see using  Lemma~\ref{Lem:1} (2) that
 \[ V(\ell\gL_2) \otimes V(-3\ell\gL_0) \ni f_2^{(2)}E_{\bm{i}}^{(p_2)}e_{10}^{(p_1)}(v_{\gL_2} \otimes v_{-3\gL_0}) = 0,
 \]
 and then the existence of the map $p\circ \Psi \colon V(\gL_2)\otimes V(-3\gL_0) \to W^1$ implies that
 $f_2^{(2)}E_{\bm{i}}^{(p_2)}e_{10}^{(p_1)}w_1=0$.
 Hence $p_3 \leq p_1$ is proved by the weight consideration. Similarly $p_4\leq p_1$ is proved from Lemma~\ref{Lem:1} (4).
 Finally we have to show that $p_j \leq 2$ for $j \in \{2,3,4\}$ even if $p_1=3$.
 Similarly as above, these are deduced from the fact that $f_2e_{10}^{(3)}w_1 = 0$, and this fact follows since $w_1$ is an extremal weight vector
 (see \cite[Theorem 5.17]{MR1890649}). The proof is complete.
\end{proof}

In the sequel, we use the symbol 
\[ \bm{a} = (1,1,1,0,1) \in \Z_{\geq 0}^5.
\]
The difficulty in the case $i=2$ is that the statements of Lemma~\ref{Lem:not12} for $i=2$ do not hold in general.
Instead, we have the following.

\begin{Lem}\label{Lem:weak_orthogonality}
 Let $\ell \in \Z_{>0}$, and assume that either $w=w_1^{\otimes \ell}\in (W^{1})^{\otimes \ell}$ or $w=w_\ell\in W^{\ell}$.
 For any $\bm{p},\bm{p}'\in\Z_{\geq 0}^5$, we have
 \[ (e_2E^{\bm{p}}w,E^{\bm{p}'}w) = 0 \ \ \ \text{unless } \ \bm{p}'=\bm{p}-\bm{a} \ \ \text{or} \ \ \bm{p}'=\bm{p}+\bm{\gee}_4.
 \]
\end{Lem}

\begin{proof}
 By the weight consideration, it is enough to show that $(e_2E^{\bm{p}}w,E^{\bm{p}'}w) = 0$ holds if $p_5 < p_5'$ or $p_5-1>p_5'$.
 If $p_5<p_5'$, the proof is similar to that of Lemma~\ref{Lem:orthogonality}.

 Assume that $p_5-1>p_5'$. It follows from~\eqref{eq:commutation2} that
 \begin{equation}\label{eq:review}
  (e_2E^{\bm{p}}w,E^{\bm{p}'}w)=\big((e_1^{(p_5-1)}e_2e_1-[p_5-1]e_1^{(p_5)}e_2)E^{\bm{p}-p_5\bm{\gee}_5}w,E^{\bm{p}'}w).
 \end{equation}
 As in Lemma~\ref{Lem:orthogonality}, it can be proved using $p_5-1>p_5'$ that 
 \[ (e_1^{(p_5-1)}e_2e_1E^{\bm{p}-p_5\bm{\gee}_5}w,E^{\bm{p}'}w) = 0 = (e_1^{(p_5)}e_2E^{\bm{p}-p_5\bm{\gee}_5}w,E^{\bm{p}'}w),
 \]
 and hence the right-hand side of~\eqref{eq:review} is zero. The proof is complete.
\end{proof}

We shall prove Proposition~\ref{Prop:e2epw} by the induction on $\ell$.
By Proposition~\ref{Prop:estimate} (1) with $\ell=1$, we have 
\[ \normsq{e_2E^{\bm{p}}w_1} \in q^{2\min(0,-p_4+p_5)}A \subseteq  q^{2\min(0,-p_4+p_5,p_1-p_4-1)-1}A
\]
for any $\bm{p} \in \Z_{\geq 0}^5$, and hence the induction begins.
Throughout the rest of this section, fix $\ell \in \Z_{>0}$ and assume that~\eqref{eq:e2epw} holds for this $\ell$.
Our goal is to prove~\eqref{eq:e2epw} with $\ell$ replaced by $\ell+1$, that is,
\begin{equation}\label{eq:case_ell+1}
 \normsq{e_2E^{\bm{p}}w_{\ell+1}} \in q^{2\min(0,-p_4+p_5,p_1-p_4-\ell-1)-1}A \ \ \ \text{for any } \bm{p} \in \Z_{\geq 0}^5.
\end{equation}
From now on, we write
\[ m(\bm{p}_1,\bm{p}_2) = m(\bm{p}_1,\bm{p}_2;\varpi_2,\ell\varpi_2) \ \ \ \text{for } \bm{p}_1,\bm{p}_2 \in \Z_{\geq 0}^5
\]
for short (the right-hand side is defined in Lemma~\ref{Lem:coproduct_rule}).
For any $\bm{p}\in \Z_{\geq 0}^5$ we have
\begin{equation}\label{eq:simple_cal}
 E^{\bm{p}}w_1^{\otimes (\ell+1)} = \sum_{\begin{smallmatrix}\bm{p}_1,\bm{p}_2 \in \Z_{\geq 0}^5;\\ \bm{p}_1+\bm{p}_2 = \bm{p}\end{smallmatrix}}
   q^{m(\bm{p}_1,\bm{p}_2)}E^{\bm{p}_1}w_1 \otimes E^{\bm{p}_2}w_1^{\otimes \ell}.
\end{equation}

\begin{Lem}\label{Lem:explixit_m}
 For $\bm{p}_1,\bm{p}_2 \in \Z_{\geq 0}^5$ with $\bm{p}_k=(p_{k1},\ldots,p_{k5})$, we have 
\begin{align*}
 m(\bm{p}_1,\bm{p}_2) =-\sum_{j=1}^5 p_{1j}p_{2j}&+(p_{12}+p_{13}+p_{14})p_{21}\\&+p_{15}(-p_{21}+p_{22}+p_{23}+p_{24})
 +\ell(3p_{11}-p_{12}-p_{13}-p_{14}).
\end{align*}
\end{Lem}

\begin{proof}
 Given weight vectors $u_1,u_2$ of some $U_q'(\fg)$-modules, 
 it follows for $i \in I$ and $p \in \Z_{\geq 0}$ that
 \[ e_i^{(p)}(u_1\otimes u_2)=\sum_{\begin{smallmatrix}p_1,p_2 \in \Z_{\geq 0};\\ p_1+p_2=p\end{smallmatrix}} 
    q_i^{-p_1(\langle h_i,\wt(u_2)\rangle+p_2)}e_i^{(p_1)}u_1 \otimes e_i^{(p_2)}u_2.
 \]
 In particular, if $e_i^{(p_1+1)}u_1=0$, $e_i^{(p_2+1)}u_2=0$ and $\langle h_i,\wt(u_2)\rangle = -p_2$, it follows that $e_i^{(p_1+p_2)}(u_1\otimes u_2)=e_i^{(p_1)}u_1\otimes e_i^{(p_2)}u_2$.
 Using these equalities, the assertion is obtained straightforwardly by calculating the coefficient of $E^{\bm{p}_1}w_1 \otimes E^{\bm{p}_2}w_1^{\otimes \ell}$
 in $E^{\bm{p}_1+\bm{p}_2}w_1^{\otimes (\ell+1)}$.
\end{proof}

\begin{Lem}\label{Lem:nonnegativity0}
 Let $\bm{p}_1,\bm{p}_2 \in \Z_{\geq 0}^5$, and assume that $E^{\bm{p}_1}w_1 \neq 0$
 and $E^{\bm{p}_2}w_1^{\otimes \ell}\neq 0$.
 Then $m(\bm{p}_1,\bm{p}_2) \geq 0$ holds.
\end{Lem}

\begin{proof}
 Let $\bm{p}=\bm{p}_1+\bm{p}_2$.
 By~\eqref{eq:simple_cal} and Lemma~\ref{Lem:orthogonality}, it follows that 
 \[ \normsq{E^{\bm{p}}w_1^{\otimes (\ell+1)}} = \sum_{\bm{p}_1'+\bm{p}_2'=\bm{p}} q^{2m(\bm{p}_1',\bm{p}_2')}\normsq{E^{\bm{p}_1'}w_1}\normsq{E^{\bm{p}_2'}w_1^{\otimes \ell}}.
 \]
 Then Proposition~\ref{Prop:equal_of_prepolarization} (2) implies that, if $E^{\bm{p}_1'}w_1$ and 
 $E^{\bm{p}_2'}w_1^{\otimes \ell}$ are both nonzero, then $m(\bm{p}_1',\bm{p}_2') \geq 0$. Hence the assertion is proved.
\end{proof}

For $\bm{p} \in \Z_{\geq 0}^5$, we have
\[ \normsq{e_2E^{\bm{p}}w_{\ell+1}}=\Big(e_2E^{\bm{p}}\big(\iota_{\ell}(w_1)\otimes \iota_{-1}(w_{\ell})\big),e_2E^{\bm{p}}\big(\iota_{-\ell}(w_1)\otimes 
   \iota_1(w_{\ell})\big)\Big)_1
\]
by Lemma~\ref{Lem:realization_of_KR}, and
\begin{align*}
 e_2E^{\bm{p}}(\iota_{\pm\ell}(w_1)\otimes &\iota_{\mp1}(w_{\ell}))\\
 =\sum_{\begin{smallmatrix} \bm{p}_1,\bm{p}_2 \in \Z_{\geq 0}^5;\\ \bm{p}_1+\bm{p}_2 = \bm{p}\end{smallmatrix}}
 &q^{m(\bm{p}_1,\bm{p}_2)\pm \ell p_{11}\mp p_{21}}\Big(\iota_{\pm\ell}(E^{\bm{p}_1}w_1)\otimes 
 \iota_{\mp 1}(e_2E^{\bm{p}_2}w_{\ell})\\
 &\quad \quad\quad\quad\quad\quad\quad\quad\quad+q^{-\langle h_2,\wt(E^{\bm{p}_2}w_{\ell})\rangle}
 \iota_{\pm\ell}(e_2E^{\bm{p}_1}w_1)\otimes \iota_{\mp 1}(E^{\bm{p}_2}w_{\ell}) \Big).
\end{align*}
Set
\[ x(\bm{p})=-\langle h_2, \wt(E^{\bm{p}}w_\ell)\rangle = p_1-2p_4+p_5-\ell \ \ \text{for} \ \bm{p} \in \Z_{\geq 0}^5.
\]
It follows from Lemma~\ref{Lem:weak_orthogonality} that
\begin{align*}
 \normsq{e_2E^{\bm{p}}w_{\ell+1}}=Z_1+Z_2+Z_3+Z_4,
\end{align*}
where
\begin{align*}
 Z_1&=\sum q^{2m(\bm{p}_1,\bm{p}_2)}\normsq{E^{\bm{p}_1}w_1}\cdot\normsq{e_2E^{\bm{p}_2}w_{\ell}},\\
 Z_2&=[2]_{q^{\ell+1}}\sum q^{m(\bm{p}_1,\bm{p}_2)+m(\bm{p}_1-\bm{a},\bm{p}_2+\bm{a})+x(\bm{p}_2)}
  (e_2E^{\bm{p}_1}w_1,E^{\bm{p}_1-\bm{a}}w_1)(E^{\bm{p}_2}w_{\ell},e_2E^{\bm{p}_2+\bm{a}}w_{\ell}),\\
 Z_3&=2\sum q^{m(\bm{p}_1,\bm{p}_2)+m(\bm{p}_1+\bm{\gee}_4,\bm{p}_2-\bm{\gee}_4)+x(\bm{p}_2)}
  (e_2E^{\bm{p}_1}w_1,E^{\bm{p}_1+\bm{\gee}_4}w_1)(E^{\bm{p}_2}w_{\ell},e_2E^{\bm{p}_2-\bm{\gee}_4}w_{\ell}),\\
 Z_4&=\sum q^{2m(\bm{p}_1,\bm{p}_2)+2x(\bm{p}_2)}\normsq{e_2E^{\bm{p}_1}w_1}\cdot\normsq{E^{\bm{p}_2}w_{\ell}}.
\end{align*}
Here all the sums are over the set $\{\bm{p}_1,\bm{p}_2 \in \Z_{\geq 0}^5\mid \bm{p}_1+\bm{p}_2 = \bm{p}\}$.
Now it suffices to show that $Z_1+Z_2+Z_3+Z_4$ belongs to the subset of $\Q(q_s)$ in~\eqref{eq:case_ell+1}.

First we shall show that $Z_2$ does.
For $k \in \Z$, write
\[ [k]_+ = \begin{cases} [k] & (k > 0) \\ 0 &(k \leq 0) \end{cases}.
\]  

\begin{Lem}\label{Lem:technical0}
 Let $\bm{p} \in \Z_{\geq 0}^5$, and set $k = p_1-p_4-\ell+1$.\\
 {\normalfont(1)} The vector $(f_2E^{\bm{p}}-[k]_+E^{\bm{p}-\bm{\gee}_4})v_{\ell\varpi_2} \in V(\ell\varpi_2)$ 
  belongs to $\pm\bB(\ell\varpi_2) \cup \{0\}$.\\
 {\normalfont(2)} We have $(f_2E^{\bm{p}} - [k]_+ E^{\bm{p}-\bm{\gee}_4})w_1^{\otimes \ell} \in L(W^{1})^{\otimes \ell}$.
\end{Lem}

\begin{proof}
 (1) By Lemma~\ref{Lem:hw-lw-lev0}, it is enough to show that $(f_2E^{\bm{p}}-[k]_+E^{\bm{p}-\bm{\gee}_4})(v_{\ell\gL_2}\otimes v_{-3\ell\gL_0})$
  belongs to $\pm\bB(\ell\gL_2,-3\ell\gL_0) \cup \{0\}$.
 The bar-invariance is obvious, and it is easily checked that
 \begin{align*}
  (f_2E^{\bm{p}}-[k]_+&E^{\bm{p}-\bm{\gee}_4})(v_{\ell\gL_2} \otimes v_{-3\ell\gL_0})\\
  &=f_2v_{\ell\gL_2}\otimes E^{\bm{p}}v_{-3\ell\gL_0}+(q^{\ell}[p_1-p_4+1]-[k]_+)v_{\ell\gL_2}\otimes E^{\bm{p}-\bm{\gee}_4}v_{-3\ell\gL_0}.
 \end{align*}
 We have $f_2v_{\ell\gL_2} \in \bB(\ell\gL_2)$, $E^{\bm{p}}v_{-3\ell\gL_0} \in \pm\bB(-3\ell\gL_0)\cup\{0\}$ by Lemma~\ref{Lem:belonging_to_B}, and 
 \[  (q^\ell[p_1-p_4+1]-[k]_+)v_{\ell\gL_2}\otimes E^{\bm{p}-\bm{\gee}_4}v_{-3\ell\gL_0} \in qL(\ell\gL_2)
  \otimes L(-3\ell\gL_0)
 \]
 since $p_1-p_4+1 < 0$ implies $E^{\bm{p}-\bm{\gee}_4}v_{-3\ell\gL_0} = 0$ (see the proof of Lemma~\ref{Lem:belonging_to_L_of_e} (1)).
 Hence we have $(f_2E^{\bm{p}}-[k]_+E^{\bm{p}-\bm{\gee}_4})(v_{\ell\gL_2}\otimes v_{-3\ell\gL_0})
 \in\pm\bB(\ell\gL_2,-3\ell\gL_0) \cup \{0\}$, as required.
 The assertion (2) follows from (1) since the map $p^{\otimes \ell}\circ \Phi\colon V(\ell\varpi_2) \to (W^{1})^{\otimes \ell}$
 sends $L(\ell\varpi_2)$ to $L(W^{1})^{\otimes \ell}$.
\end{proof}

We need the following relation in $W^{\ell}$: there exists a certain element $c_\ell \in \pm 1+q_sA$ such that
\begin{equation}\label{eq:the_relation}
 e_2E^{\bm{p}}w_\ell = c_\ell E^{\bm{p}-\bm{a}+\bm{\gee}_4}f_2w_\ell + [p_4-p_5+1]E^{\bm{p}+\bm{\gee}_4}w_\ell
\end{equation}
for $\bm{p}\in \Z_{\geq 0}^5$.
It is a rather straightforward computation, but we will give a proof in Appendix~\ref{Appendix} (Proposition~\ref{Prop:A}) as it is somewhat lengthy and technical.

\begin{Lem}\label{Lem:min_max}
  Let $\bm{p} \in \Z_{\geq 0}^5$.\\[2pt]
 {\normalfont(1)} We have
  \[ (e_2E^{\bm{p}}w_\ell,E^{\bm{p}-\bm{a}}w_\ell) \in q^{\min(0,p_1-p_4-\ell)}A.
  \]
 {\normalfont(2)} When $\ell =1$, the following stronger statement holds:
  \[ (e_2E^{\bm{p}}w_1,E^{\bm{p}-\bm{a}}w_1) \in q^{\max(0,p_1-p_4-1)}A.
  \]
\end{Lem}

\begin{proof}
 (1) By~\eqref{eq:the_relation} and Lemma~\ref{Lem:orthogonality}, we have 
 \begin{align}\label{eq:e2epwlepawl}
   (e_2E^{\bm{p}}w_\ell,E^{\bm{p}-\bm{a}}w_\ell) 
   =c_\ell (E^{\bm{p}-\bm{a}+\bm{\gee}_4}f_2w_\ell, E^{\bm{p}-\bm{a}}w_\ell).
 \end{align}
 It is easily checked that $X=E^{\bm{p}-\bm{a}+\bm{\gee}_4}f_2$ and $Y=E^{\bm{p}-\bm{a}}$ satisfy the assumptions of Lemma~\ref{Lem:equality_of_prepolarizations_on_W},
 and hence we have
 \begin{equation}\label{eq:textref}
  (\text{\ref{eq:e2epwlepawl}})=c_\ell(E^{\bm{p}-\bm{a}+\bm{\gee}_4}f_2w_1^{\otimes \ell},E^{\bm{p}-\bm{a}}w_1^{\otimes \ell})_{(W^1)^{\otimes \ell}}.
 \end{equation}
 A calculation using Lemma~\ref{Lem:1} shows that
 \begin{align*}
  E^{\bm{p}-\bm{a}+\bm{\gee}_4}f_2(v_{\ell\gL_2}\otimes v_{-3\ell\gL_0})= (f_2E^{\bm{p}-\bm{a}+\bm{\gee}_4} + [-p_1+p_4+\ell+1]E^{\bm{p}-\bm{a}})
  (v_{\ell\gL_2}\otimes v_{-3\ell\gL_0}),
 \end{align*}
 and then the existence of the map $V(\ell\gL_2)\otimes V(-3\ell\gL_0) \to (W^1)^{\otimes \ell}$ implies that 
 \begin{equation}\label{eq;textrefeq}
  (\text{\ref{eq:textref}})= c_\ell(f_2E^{\bm{p}-\bm{a}+\bm{\gee}_4}w_1^{\otimes \ell}+[-p_1+p_4+\ell+1]E^{\bm{p}-\bm{a}}w_1^{\otimes \ell},
    E^{\bm{p}-\bm{a}}w_1^{\otimes \ell})_{(W^1)^{\otimes \ell}}.
 \end{equation}
 By Lemma~\ref{Lem:technical0} (2), we have
 \begin{align*}
  f_2E^{\bm{p}-\bm{a}+\bm{\gee}_4}w_1^{\otimes \ell}+[-p_1+p_4+\ell+1]&E^{\bm{p}-\bm{a}}w_1^{\otimes \ell}\\
  &\equiv [-p_1+p_4+\ell+1]_+E^{\bm{p}-\bm{a}}w_1^{\otimes \ell} \ \ \ \text{mod } L(W^{1})^{\otimes \ell},
 \end{align*}
 and hence it follows from Proposition~\ref{Prop:equal_of_prepolarization} and~\eqref{eq:LWr1} that
 \[ (f_2E^{\bm{p}-\bm{a}+\bm{\gee}_4}w_1^{\otimes \ell}+[-p_1+p_4+\ell+1]E^{\bm{p}-\bm{a}}w_1^{\otimes \ell},
    E^{\bm{p}-\bm{a}}w_1^{\otimes \ell}) \in q^{\min(0,p_1-p_4-\ell)}A.
 \]
 Now the assertion (1) is proved since $c_\ell \in \pm 1+q_sA$.

 (2) We may assume that $E^{\bm{p}-\bm{a}}w_1 \neq 0$, and hence that $p_4 \leq p_1-1$ by Lemma~\ref{Lem:properties_of_W^1}.
  Then by (1), it is enough to consider the case $p_1-p_4 \geq 2$.
  First assume that $p_1-p_4=2$. 
  By~\eqref{eq:e2epwlepawl} and \eqref{eq;textrefeq}, it suffices to show that
  \begin{equation}\label{eq:belonging_to_qA}
   (f_2E^{\bm{p}-\bm{a}+\bm{\gee}_4}w_1,E^{\bm{p}-\bm{a}}w_1) \in qA,
  \end{equation} 
  and we may assume that the two vectors are both nonzero.
  Since the two vectors $f_2E^{\bm{p}-\bm{a}+\bm{\gee}_4}(v_{\gL_2}\otimes v_{-3\gL_0})$ and $v_{\gL_2}\otimes E^{\bm{p}-\bm{a}}v_{-3\gL_0}$
  both belong to 
  $\pm\bB(\gL_2,-3\gL_0)$ and are obviously linearly independent,
  we see from Lemma~\ref{Lem:hw-lw-lev0} that
  $f_2E^{\bm{p}-\bm{a}+\bm{\gee}_4} v_{\varpi_2}$ and $E^{\bm{p}-\bm{a}}v_{\varpi_2} $ both belong to $\pm\bB(\varpi_2)$ and are linearly independent.
  Moreover since their $P$-weights are the same, 
  $(f_2E^{\bm{p}-\bm{a}+\bm{\gee}_4} v_{\varpi_2}, z_2^k E^{\bm{p}-\bm{a}}v_{\varpi_2}) = 0$ if $k \neq 0$.
  Hence~\eqref{eq:belonging_to_qA} follows from Proposition~\ref{Prop:Nakajima_prepolarization} (3) and~\eqref{eq:invariance_of_zk}.
   
  It remains to show the assertion in the case $p_1-p_4=3$, that is, $p_1=3$ and $p_4=0$.
  By the admissibility, we have 
  \begin{align*}
   (e_2E^{\bm{p}}w_1,E^{\bm{p}-\bm{a}}w_1) &= q^{p_5+1}(E^{\bm{p}}w_1,f_2E^{\bm{p}-\bm{a}}w_1).
  \end{align*}
  Since $E^{\bm{p}}w_1$ and $f_2E^{\bm{p}-\bm{a}}w_1$ both 
  belong to $L(W^{1})$ and $E^{\bm{p}-\bm{a}}w_1 \neq 0$ implies $p_5 \geq 1$, this belongs to
  $q^2A$. The proof is complete.
\end{proof}

Now we show the following proposition, which assures that $Z_2$ belongs to the set in~\eqref{eq:case_ell+1}.

\begin{Prop}
 Let $\bm{p}_1,\bm{p}_2 \in \Z_{\geq 0}^5$, and set $\bm{p} = \bm{p}_1+\bm{p}_2$.
 Then we have 
 \begin{align*}
  q^{m(\bm{p}_1,\bm{p}_2)+m(\bm{p}_1-\bm{a},\bm{p}_2+\bm{a})+x(\bm{p}_2)}
  (e_2E^{\bm{p}_1}w_1,E^{\bm{p}_1-\bm{a}}w_1)(E^{\bm{p}_2}w_{\ell},&e_2E^{\bm{p}_2+\bm{a}}w_{\ell})\\
   &\in q^{\min(0,p_1-p_4-\ell)+p_1-p_4-1}A,
 \end{align*}
 where $\bm{p}=(p_1,\ldots,p_5)$ and $x(\bm{p}_2) = -\langle h_2,\wt(E^{\bm{p}_2}w_\ell)\rangle$.
\end{Prop}

\begin{proof}
 Set $\bm{p}_i=(p_{i1},\ldots,p_{i5})$ ($i =1,2$).
 It is directly checked from Lemma~\ref{Lem:explixit_m} that 
 \begin{equation}\label{eq:mysterious}
  m(\bm{p}_1,\bm{p}_2) + x(\bm{p}_2) = m(\bm{p}_1-\bm{a},\bm{p}_2+\bm{a})+ p_1-p_4-1.
 \end{equation}
 We may assume that $E^{\bm{p}_1-\bm{a}}w_1\neq 0$ and $E^{\bm{p}_2+\bm{a}}w_\ell \neq 0$.
 By the induction hypothesis, it follows from Proposition~\ref{Prop:criterion} that the prepolarization $(\ ,\ )_{W^\ell}$ is positive definite,
 and hence $E^{\bm{p}_2+\bm{a}}w_\ell \neq 0$ implies 
 $E^{\bm{p}_2+\bm{a}}w_1^{\otimes \ell} \neq 0$ by Proposition~\ref{Prop:equality}.
 Then it follows from Lemmas~\ref{Lem:nonnegativity0} and~\ref{Lem:min_max} that
 \begin{align*}
  &q^{m(\bm{p}_1,\bm{p}_2)+m(\bm{p}_1-\bm{a},\bm{p}_2+\bm{a})+x(\bm{p}_2)}
  (e_2E^{\bm{p}_1}w_1,E^{\bm{p}_1-\bm{a}}w_1)(E^{\bm{p}_2}w_{\ell},e_2E^{\bm{p}_2+\bm{a}}w_{\ell})\\
    &\phantom{aaaa}\in q^{2m(\bm{p}_1-\bm{a},\bm{p}_2+\bm{a})+p_1-p_4-1}\cdot q^{\max(0,p_{11}-p_{14}-1)}\cdot q^{\min(0,p_{21}-p_{24}-\ell+1)}A\\ 
    &\phantom{aaaa}\subseteq q^{\min(0,p_1-p_4-\ell)+p_1-p_4-1}A.
 \end{align*}
 The assertion is proved.
\end{proof}

Next we shall show that $Z_1$ belongs to the set in~\eqref{eq:case_ell+1}.

\begin{Lem}\label{Lem:technical}
 Assume that $\bm{p} \in\Z_{\geq 0}^5$ satisfies $E^{\bm{p}}w_1^{\otimes \ell}\neq 0$.\\
 {\normalfont(1)} If $p_1>p_4+\ell$, then $E^{\bm{p}+\bm{\gee}_4}w_1^{\otimes \ell} \neq 0$.\\
 {\normalfont(2)} If $p_4>p_5$,
 then either $E^{\bm{p}-\bm{\gee}_4}w_1^{\otimes \ell} \neq 0$ or $E^{\bm{p}+\bm{\gee}_5}w_1^{\otimes \ell}\neq 0$ holds.
\end{Lem}

\begin{proof}
 (1) First consider the case $\ell=1$.
  By (the proof of) Lemma~\ref{Lem:technical0} (1) and Lemma~\ref{Lem:hw-lw-lev0}, the vector $(f_2E^{\bm{p}+\bm{\gee}_4} - [p_1-p_4-1]E^{\bm{p}})w_1$ is either $0$,
  or not proportional to $E^{\bm{p}}w_1$.
  In both cases we have $f_2E^{\bm{p}+\bm{\gee}_4}w_1 \neq 0$, and hence the assertion (1) is proved for $\ell =1$.

 Assume that $\ell>1$.
 Obviously, $E^{\bm{p}}w_1^{\otimes \ell} \neq 0$ implies $E^{\bm{p}_1}w_1 \otimes \cdots \otimes 
 E^{\bm{p}_\ell}w_1 \neq 0$ for some $\bm{p}_1,\ldots,\bm{p}_\ell \in \Z_{\geq 0}^5$
 such that $\bm{p} = \bm{p}_1+\cdots+\bm{p}_\ell$.
 The assumption implies that there exists some $k$ such that $p_{k1}-p_{k4}>1$, and then $E^{\bm{p}_k+\bm{\gee}_4}w_1 \neq 0$ holds by the argument for
 $\ell=1$.
 Since the nonzero vectors of the form $E^{\bm{p}'_1}w_1 \otimes \cdots \otimes E^{\bm{p}'_\ell}w_1$ are linearly independent by Lemma~\ref{Lem:orthogonality},
 this implies that $E^{\bm{p}+\bm{\gee}_4}w_1^{\otimes \ell}$ is nonzero.
 The assertion is proved.

 (2) First assume that $\ell=1$.
 If $p_5 =0$, $E^{\bm{p}-\bm{\gee}_4}w_1 \neq 0$ obviously holds, and hence we may assume $p_5\geq 1$.  
 That $E^{\bm{p}}w_1 \neq 0$ implies $p_4\leq 2$ by Lemma~\ref{Lem:properties_of_W^1},
 which forces $p_4=2$ and $p_5=1$.
 If 
 \[
 \langle h_1,\wt(E^{\bm{p}-\bm{\gee}_5}w_1)\rangle =p_1-p_2-p_3-2 \leq -2,
 \]
 then $E^{\bm{p}+\bm{\gee}_5}w_1 \neq 0$ follows, and hence we may assume that $p_1 > p_2+p_3$.
 If $p_3=0$, since~\eqref{eq:commutation2} implies $e_1E_{\bm{i}}^{(p_2)}e_{10}^{(p_1)}w_1 = 0$,
 we have  
 \[
 e_2E^{\bm{p}-\bm{\gee}_4}w_1 = E^{\bm{p}}w_1+e_2^{(2)}E^{\bm{p}-2\bm{\gee}_4}w_1=E^{\bm{p}}w_1 \neq 0,
 \]
 which implies $E^{\bm{p}-\bm{\gee}_4}w_1\neq 0$.
 It is also checked similarly that $E^{\bm{p}-\bm{\gee}_4}w_1\neq 0$ holds if $p_2=0$.
 The remaining case is $\bm{p}=(3,1,1,2,1)$ only, and in this case $E^{(3,1,1,1,1)}w_1 \neq 0$ is proved from~\eqref{eq:elementary_fact} and
 \[
 f_1E^{(3,1,1,1,0)}w_1= E^{(0,1,1,1,0)}e_1^{(2)}e_0^{(3)}w_1 \neq 0.
 \]
 The proof for $\ell=1$ is complete.
 Then the same argument used in the proof of (1) also works here, and (2) for general $\ell$ is proved.
\end{proof}

\begin{Lem}\label{Lem:nonnegativity}
 Let $\bm{p}_1,\bm{p}_2 \in \Z_{\geq 0}^5$ be such that $E^{\bm{p}_1}w_1 \neq 0$ and $E^{\bm{p}_2}w_1^{\otimes \ell}\neq 0$.\\
 {\normalfont(1)} If $p_{11} >p_{14}+1$, then $m(\bm{p}_1,\bm{p}_2) \geq -p_{21}+p_{24}+\ell$.\\
 {\normalfont(2)} If $p_{24} >p_{25}$, then we have $m(\bm{p}_1,\bm{p}_2) \geq -p_{14}+p_{15}$.
\end{Lem}

\begin{proof}
 (1) By Lemma~\ref{Lem:technical} (1), we have $E^{\bm{p}_1+\bm{\gee}_4}w_1 \neq 0$, and hence $m(\bm{p}_1+\bm{\gee}_4,\bm{p}_2) \geq 0$ follows
  from Lemma~\ref{Lem:nonnegativity0}.
  Since we have 
 \begin{align*}
  m(\bm{p}_1,\bm{p}_2) = m(\bm{p}_1+\bm{\gee}_4,\bm{p}_2) - p_{21}+p_{24}+\ell
 \end{align*}
 by Lemma~\ref{Lem:explixit_m}, the assertion (1) follows.

 (2) By Lemma~\ref{Lem:technical} (2), we have either $E^{\bm{p}_2-\bm{\gee}_4}w_1^{\otimes \ell} \neq 0$ or $E^{\bm{p}_2+\bm{\gee}_5}w_1^{\otimes \ell} 
 \neq 0$, and hence either $m(\bm{p}_1,\bm{p}_2-\bm{\gee}_4) \geq 0$ or $m(\bm{p}_1,\bm{p}_2+\bm{\gee}_5) \geq 0$ holds.
 Since we have
 \begin{align*}
 m(\bm{p}_1,\bm{p}_2) = m(\bm{p}_1,\bm{p}_2-\bm{\gee}_4) - p_{14}+p_{15}\ \ \ \text{and} \ \ \ 
 m(\bm{p}_1,\bm{p}_2) = m(\bm{p}_1,\bm{p}_2+\bm{\gee}_5) + p_{15},
 \end{align*}
 in both cases $m(\bm{p}_1,\bm{p}_2) \geq -p_{14}+p_{15}$ holds, and the proof is complete.
\end{proof}

Now the following proposition implies that $Z_1$ belongs to the set in~\eqref{eq:case_ell+1}.

\begin{Prop}
 Assume that $\bm{p}_1,\bm{p}_2 \in \Z_{\geq 0}^5$ satisfy $E^{\bm{p}_1}w_1 \neq 0$ and $E^{\bm{p}_2}w_{\ell} \neq 0$.
 Setting $\bm{p}=\bm{p}_1+\bm{p}_2$, we have 
 \begin{equation}\label{eq:belonging1}
  q^{2m(\bm{p}_1,\bm{p}_2)}\normsq{E^{\bm{p}_1}w_1}\cdot\normsq{e_2E^{\bm{p}_2}w_{\ell}} \in q^{2\min(0,-p_4+p_5,p_1-p_4-\ell-1)-1}A.
 \end{equation}
\end{Prop}

\begin{proof}
 Set 
 \[ N = \min(0,-p_4+p_5,p_1-p_4-\ell-1) \ \ \text{and} \ \ N_2 = \min(0,-p_{24}+p_{25},p_{21}-p_{24}-\ell).
 \]
 Since $\normsq{E^{\bm{p}_1}w_1} \in 1+q_sA$ by Proposition~\ref{Prop:equal_of_prepolarization} and $\normsq{e_2E^{\bm{p}_2}w_{\ell}}
 \in q^{2N_2-1}A$ by~\eqref{eq:e2epw} with $\bm{p}$ replaced by $\bm{p}_2$ (which we are assuming to hold), it suffices to show that
 \begin{equation}\label{eq:mM2}
  m(\bm{p}_1,\bm{p}_2) + N_2 \geq N.
 \end{equation}
 If $N_2 = 0$, this follows from Lemma~\ref{Lem:nonnegativity0}.
 Moreover if $N_2 = -p_{24}+p_{25}<0$, this holds since
 \[ m(\bm{p}_1,\bm{p}_2) + (-p_{24}+p_{25})\geq (-p_{14}+p_{15})+(-p_{24}+p_{25}) = -p_4+p_5
 \] 
 by Lemma~\ref{Lem:nonnegativity} (2).
 Finally assume that $N_2 =p_{21}-p_{24}-\ell$. 
 If $p_{11} \leq p_{14} + 1$, then~\eqref{eq:mM2} holds since
 \[ N_2 \geq N_2 + (p_{11}-p_{14}-1) = p_1-p_4-\ell-1.
 \]
 On the other hand if $p_{11}>p_{14} + 1$,~\eqref{eq:mM2} follows from Lemma~\ref{Lem:nonnegativity} (1).
 The proof is complete.  
\end{proof}

Finally, we shall show that $Z_3+Z_4$ belongs to the set in~\eqref{eq:case_ell+1}, 
which completes the proof of Proposition~\ref{Prop:e2epw}.
By a similar calculation that we did for $\normsq{e_2E^{\bm{p}}w_{\ell+1}}$, we have
\begin{equation*}
 \normsq{e_2E^{\bm{p}}w_{1}^{\otimes (\ell+1)}} = W_1+W_2+W_3+W_4,
\end{equation*}
where 
\begin{align*}
 W_1&=\sum q^{2m(\bm{p}_1,\bm{p}_2)}\normsq{E^{\bm{p}_1}w_1}\cdot\normsq{e_2E^{\bm{p}_2}w_1^{\otimes \ell}},\\
 W_2&=2\sum q^{m(\bm{p}_1,\bm{p}_2)+m(\bm{p}_1-\bm{a},\bm{p}_2+\bm{a})+x(\bm{p}_2)}
  (e_2E^{\bm{p}_1}w_1,E^{\bm{p}_1-\bm{a}}w_1)(E^{\bm{p}_2}w_1^{\otimes \ell},e_2E^{\bm{p}_2+\bm{a}}w_1^{\otimes \ell}),\\
 W_3&=2\sum q^{m(\bm{p}_1,\bm{p}_2)+m(\bm{p}_1+\bm{\gee}_4,\bm{p}_2-\bm{\gee}_4)+x(\bm{p}_2)}
  (e_2E^{\bm{p}_1}w_1,E^{\bm{p}_1+\bm{\gee}_4}w_1)(E^{\bm{p}_2}w_1^{\otimes \ell},e_2E^{\bm{p}_2-\bm{\gee}_4}w_1^{\otimes \ell}),\\
 W_4&=\sum q^{2m(\bm{p}_1,\bm{p}_2)+2x(\bm{p}_2)}\normsq{e_2E^{\bm{p}_1}w_1}\cdot\normsq{E^{\bm{p}_2}w_1^{\otimes \ell}}.
\end{align*}
We have $W_4 = Z_4$ by Proposition~\ref{Prop:equality}.
Moreover, the equality
\[ (E^{\bm{p}_2}w_1^{\otimes \ell},e_2E^{\bm{p}_2-\bm{\gee}_4}w_1^{\otimes \ell})_{(W^1)^{\otimes \ell}}=(E^{\bm{p}_2}w_\ell,e_2E^{\bm{p}_2-\bm{\gee}_4}w_\ell)_{W^\ell}
\]
is proved for any $\bm{p}_2$ by checking $X=E^{\bm{p}_2}$ and $Y=e_2E^{\bm{p}_2-\bm{\gee}_4}$ satisfy the assumptions of Lemma 
\ref{Lem:equality_of_prepolarizations_on_W}, and hence $W_3=Z_3$ follows.
On the other hand, the left-hand side $\normsq{e_2E^{\bm{p}}w_{1}^{\otimes (\ell+1)}}$ belongs to $q^{2\min(0,-p_4+p_5)}$ by Proposition~\ref{Prop:estimate} (1).
Hence in order to show that $Z_3+Z_4(=W_3+W_4)$ belongs to the set in~\eqref{eq:case_ell+1}, 
it is enough to prove that both $W_1$ and $W_2$ do.
The assertion for $W_1$ is deduced from the following lemma.

\begin{Lem}
 For any $\bm{p}_1,\bm{p}_2 \in \Z_{\geq 0}$, we have
 \[ q^{2m(\bm{p}_1,\bm{p}_2)}\normsq{E^{\bm{p}_1}w_1} \cdot \normsq{ e_2E^{\bm{p}_2}w_1^{\otimes \ell}} \in q^{2\min(0,-p_4+p_5)}A,
 \]
 where we set $\bm{p} = \bm{p}_1+\bm{p}_2$.
\end{Lem} 

\begin{proof}
 We may assume that $E^{\bm{p}_1}w_1 \neq 0$ and $E^{\bm{p}_2}w_1^{\otimes \ell} \neq 0$.
 We have $\normsq{E^{\bm{p}_1}w_1} \in 1+q_sA$ by Proposition~\ref{Prop:equal_of_prepolarization}, 
 and $\normsq{e_2E^{\bm{p}_2}w_1^{\otimes \ell}} \in q^{2\min(0,-p_{24}+p_{25})}A$ by Proposition~\ref{Prop:estimate} (1).
 If $-p_{24}+p_{25} \geq 0$, the assertion follows from Lemma~\ref{Lem:nonnegativity0}.
 Otherwise we have $m(\bm{p}_1,\bm{p}_2)\geq -p_{14}+p_{15}$ by Lemma~\ref{Lem:nonnegativity} (2), and hence
 the assertion is proved.  
\end{proof}

The assertion for $W_2$ is easily proved from the following lemma and~\eqref{eq:mysterious}.

\begin{Lem}
 For any $\bm{p} \in \Z_{\geq 0}^5$, we have 
 \begin{equation}\label{eq:final}
  (e_2E^{\bm{p}}w_1^{\otimes \ell},E^{\bm{p}-\bm{a}}w_1^{\otimes \ell}) \in A.
 \end{equation}
\end{Lem}

\begin{proof}
 We proceed by the induction on $\ell$.
 The assertion for the base case of $\ell = 1$ follows from Lemma~\ref{Lem:min_max} (2).

 Assume~\eqref{eq:final} for a fixed $\ell$ and any $\bm{p}$.
 Our task is to prove this with $\ell$ replaced by $\ell+1$. 
 We have
 \begin{align*}
  (e_2E^{\bm{p}}&w_1^{\otimes(\ell+1)},E^{\bm{p}-\bm{a}}w_1^{\otimes (\ell+1)})\\
   &=\sum_{\bm{p}_1+\bm{p}_2 = \bm{p}}
   q^{m(\bm{p}_1,\bm{p}_2)}\Big(q^{m(\bm{p}_1-\bm{a},\bm{p}_2)+x(\bm{p}_2)}(e_2E^{\bm{p}_1}w_1,E^{\bm{p}_1-\bm{a}}w_1)
  \normsq{E^{\bm{p}_2}w_1^{\otimes \ell}}\\
   & \hspace{150pt}+q^{m(\bm{p}_1,\bm{p}_2-\bm{a})}\normsq{E^{\bm{p}_1}w_1}(e_2E^{\bm{p}_2}w_1^{\otimes \ell}, E^{\bm{p}_2-\bm{a}}w_1^{\otimes \ell})\Big).
 \end{align*}
 By the induction hypothesis and Lemma~\ref{Lem:nonnegativity0},
 \[ q^{m(\bm{p}_1,\bm{p}_2)+m(\bm{p}_1,\bm{p}_2-\bm{a})}\normsq{E^{\bm{p}_1}w_1}(e_2E^{\bm{p}_2}w_1^{\otimes \ell}, E^{\bm{p}_2-\bm{a}}w_1^{\otimes \ell}) \in A
 \]
 holds.
 On the other hand, $E^{\bm{p}_2}w_1^{\otimes \ell}\neq 0$ implies $p_{21}\geq p_{24}$ by Lemma~\ref{Lem:properties_of_W^1}.
 Since 
 \[ m(\bm{p}_1,\bm{p}_2)+x(\bm{p}_2)= m(\bm{p}_1-\bm{a},\bm{p}_2)+p_{21}-p_{24}
 \]
 by Lemma~\ref{Lem:explixit_m}, it also follows from the induction hypothesis that
 \[ q^{m(\bm{p}_1,\bm{p}_2)+m(\bm{p}_1-\bm{a},\bm{p}_2)+x(\bm{p}_2)}
    (e_2E^{\bm{p}_1}w_1,E^{\bm{p}_1-\bm{a}}w_1)
  \normsq{E^{\bm{p}_2}w_1^{\otimes \ell}} \in A.
 \]
 The proof is complete.
\end{proof}

\appendix
\section{}\label{Appendix}

The goal of this appendix is to show the following.

\begin{PropA}\label{Prop:A}
 Let $\ell \in \Z_{>0}$. 
 There exists an element $c_\ell \in \pm 1+q_sA$ such that
 \[ e_2E^{\bm{p}}w_\ell = c_\ell E^{\bm{p}-\bm{a}+\bm{\gee}_4}f_2w_\ell + [p_4-p_5+1]E^{\bm{p}+\bm{\gee}_4}w_\ell
 \]
 for any $\bm{p} \in \Z_{\geq 0}^5$.
\end{PropA}

A fundamental tool for the proof is the braid group action on $U_q(\fg)$ introduced by Lusztig.
For $i \in I$, let $T_i = T_{i,1}''$ be the algebra automorphism of $U_q(\fg)$ in \cite[Chapter~37]{MR1227098}.
For a sequence $i_p\cdots i_1$ of elements of $I$, write $T_{i_p\cdots i_1} = T_{i_p}\cdots T_{i_1}$.
Here we collect the properties of $T_i$; for the proofs, see~\cite{lusztig1990finite,MR1227098}.

\begin{LemA}\label{LemA:2}\
 \begin{enumerate}
  \renewcommand{\labelenumi}{(\alph{enumi})}
  \renewcommand{\theenumi}{(\alph{enumi})}
  \setlength{\parskip}{1pt} 
  \setlength{\itemsep}{1pt} 
  \setlength{\leftmargin}{2pt}
 \item For $i \in I$ and $\ga \in Q$, we have $T_iU_q(\fg)_\ga = U_q(\fg)_{s_i(\ga)}$.\label{en0-1}
 \item\label{en0-2} For $i,j \in I$ and $p \in \Z_{> 0}$, we have
  \[ T_i(e_j^{(p)}) = \sum_{k=0}^{-c_{ij}p} (-q_i)^{-k}e_i^{(-c_{ij}p-k)}e_j^{(p)}e_i^{(k)}.
  \]
 \item\label{en0-3} For $i, j\in I$, we have $\underbrace{T_iT_j\cdots}_{c_{ij}c_{ji}+2} = \underbrace{T_jT_i\cdots}_{c_{ij}c_{ji}+2}$.
 \item\label{en0-4} If $i_p\cdots i_1$ is a reduced word, then $T_{i_p\cdots i_{2}}(e_{i_1}) \in U_q(\fn_+)$.
  Moreover, if we further assume that $s_{i_p}\cdots s_{i_2}(\ga_{i_1}) = \ga_j$ for some $j \in I$,
  then we have $T_{i_p\cdots i_{2}}(e_{i_1})=e_j$.
 \item\label{en0-5} Let $i,j\in I$ be such that $c_{ij}=c_{ji}=-1$ and $p \in \Z_{>0}$. Then we have
  \[ e_ie_j^{(p)} = e_j^{(p-1)}T_i(e_j)+q^{-p}e_j^{(p)}e_i\ \ \ \text{and} \ \ \ T_i(e_j)e_j = qe_jT_i(e_j).
  \]
 \item Let $M$ be an integrable $U_q(\fg)$-module, and $i\in I$. There is a $\Q(q_s)$-linear automorphism $\wti{T}_i$
  (denoted by $T_{i,1}''$ in~\cite{MR1227098}) satisfying $\wti{T}_i(Xm) = T_i(X)\wti{T}_i(m)$ for $X \in U_q(\fg)$ and $m \in M$.
  Moreover if $m \in M_\gl$ for $\gl \in D^{-1}P$ and $f_im=0$, we have 
  \[ \wti{T}_i(e_i^{(p)}m)=(-1)^pq_i^{p(-\gl_i -p+1)}e_i^{(-\gl_i -p)}m
  \]
  for $p \in \Z_{\geq 0}$, where we set $\gl_i = \langle h_i,\gl\rangle$.\label{en0-6}
 \end{enumerate}
\end{LemA}

\begin{LemA}\label{LemA:concatenation}
 The word $\bm{j}\bm{i} =(j_{L'}\cdots j_0i_{L}\cdots i_0)$ is reduced.
\end{LemA}

\begin{proof}
 For any $0\leq k \leq L$, we have
 \[ \langle s_{\bm{j}}s_{\bm{i}[L,k+1]}(h_{i_k}), \theta_1\rangle =\langle h_{i_k}, s_{\bm{i}[k,1]}(\ga_2)\rangle >0,
 \]
 which implies $s_{\bm{j}}s_{\bm{i}[L,k+1]}(\ga_{i_k}) \in R_1^+$.
 This, together with Lemma~\ref{Lem:fundamental_properties} (4), implies the assertion.
\end{proof}

In the sequel, we write $\bm{i}_0=\bm{i}[L,1]$ and $\bm{j}_0 =\bm{j}[L',1]$ for short.

\begin{LemA}\label{LemA:many}
 Let $M$ be an integrable $U_q(\fg)$-module, $v \in M\setminus \{0\}$, and $p \in \Z_{>0}$.
 \begin{enumerate}
  \renewcommand{\labelenumi}{(\arabic{enumi})}
  \renewcommand{\theenumi}{(\arabic{enumi})}
  \setlength{\parskip}{1pt} 
  \setlength{\itemsep}{1pt} 
  \setlength{\leftmargin}{0pt}
  \item\label{en0} If $e_iv=0$ $(i\in \{1\}\sqcup J \setminus \{2\})$ and $e_{1}e_{2}v=0$, then $T_{\bm{j}}(e_1)v=0$.
  \item\label{en1-2} If $e_iv=0$ $(i\in \{1\}\sqcup J \setminus \{2\})$, then $T_{\bm{j}_01}(e_2^{(p)})v=E_{1\bm{j}}^{(p)}v=(-q)^pT_{\bm{j}}(e_1^{(p)})v$.
  \item\label{en1-1} If $e_iv = 0$ $(i \in J \setminus\{2\})$, then $T_{\bm{j}_0}(e_2^{(p)})v = E_{\bm{j}}^{(p)}v$.
  \item\label{en1-4} We have $T_{\bm{i}_01}(e_2)=e_1T_{\bm{i}_0}(e_2)-q^{-1}T_{\bm{i}_0}(e_2)e_1$.
  \item\label{en1-3} If $e_iv=0$ $(i\in I_0)$, then $T_{\bm{i}1}(e_0)v=E_{\bm{i}10}v$.
  \item\label{en2-5} We have $e_1T_{\bm{i}_0}(e_2^{(p)})=T_{\bm{i}_0}(e_2^{(p-1)})T_{\bm{i}_01}(e_2)+
   q^{-p}T_{\bm{i}_0}(e_2^{(p)})e_1$.
  \item\label{en2-1} If $e_1v = 0$, then $e_1T_{\bm{j}_0}(e_2^{(p)})v= T_{\bm{j}_0}(e_2^{(p-1)})T_{\bm{j}_01}(e_2)v$.
  \item\label{en2-2} If $e_iv=0$ $(i \in I_{01})$, then $T_{\bm{i}_0}(e_2)e_1^{(p)}v=e_1^{(p-1)}T_{\bm{i}}(e_1)v$.
  \item\label{en2-3} If $e_iv=0$ $(i \in I_0)$, then $T_{\bm{i}}(e_1)e_0^{(p)}v=e_0^{(p-1)}T_{\bm{i}1}(e_0)v$.
  \item\label{en2-4} We have $T_{\bm{j}}(e_1)e_0^{(p)} = e_0^{(p-1)}T_{\bm{j}1}(e_0)+q^{-p}e_0^{(p)}T_{\bm{j}}(e_1)$.
  \item\label{en3-1} We have $T_{\bm{j}}(e_1)e_1^{(p)}=q^{p}e_1^{(p)}T_{\bm{j}}(e_1).$
  \item\label{en4-1} We have $e_1T_{\bm{j}1}(e_0)T_{\bm{i}1}(e_0)=T_{\bm{j}1}(e_0)T_{\bm{i}1}(e_0)e_1$.
  \item\label{en1-5} If $e_iv=0$ $(i \in I_0 \setminus\{1,2\})$, then $T_{\bm{i}_0}
   (e_2^{(p)})v=E_{\bm{i}}^{(p)}v=a^{p}T_{\bm{j}\bm{i}_0}(e_2^{(p)})v$ with some nonzero $a \in \Q(q_s)$.
 \end{enumerate}
\end{LemA}

\begin{proof}
 Let us prepare some notation.
 For a subset $L \subseteq I$ and $\gL \in -P^+$, denote by $V_L(\gL)$ the $U_q(\fg_L)$-submodule of $V(\gL)$ generated by $v_{\gL}$, which is isomorphic to the simple 
 lowest weight $U_q(\fg_L)$-module whose lowest weight is the restriction of $\gL$ on $\sum_{i \in L} D^{-1}h_i$.

 Let us prove the assertion~\ref{en0}. 
 Set $J'=\{1\}\sqcup J$, and $\ell= \max\{m \in \Z_{\geq 0}\mid e_2^{(m)}v\neq 0\}$.
 By the well-known fact for the defining relations (see the proof of Lemma~\ref{Lem:coproduct_rule}),
 there is a $U_q(\fn_{+,J'})$-module homomorphism from $V_{J'}(-\ell\gL_2)$ to $M$ mapping $v_{-\ell\gL_2}$ to $v$. 
 Hence we may assume that $v=v_{-\ell\gL_2}$,
 and then the assertion~\ref{en0} is proved as follows: By Lemma~\ref{LemA:2}~\ref{en0-2} and~\ref{en0-6},
 \[ T_{\bm{j}}(e_1)v = T_{\bm{j}_0}(e_2e_1-q^{-1}e_1e_2)v = \widetilde{T}_{\bm{j}_0}\big((e_2e_1-q^{-1}e_1e_2)v\big)=0.
 \]

 Next we shall prove the assertion~\ref{en1-2}.
 As above, we may assume that $v=v_{-\ell\gL_2}$ for some $\ell \in \Z_{>0}$.
 The first equality is proved using Lemma~\ref{LemA:2}~\ref{en0-6} as follows:
 \[ T_{\bm{j}_01}(e_2^{(p)})v=\wti{T}_{\bm{j}_01}(e_2^{(p)}v)=E_{1\bm{j}}^{(p)}v.
 \]
 By Lemma~\ref{LemA:2}~\ref{en0-2}, we have
 \[ T_{\bm{j}}(e_1^{(p)})v=\sum_k(-q)^{-k} T_{\bm{j}_0}(e_2^{(p-k)}e_1^{(p)}e_2^{(k)})v= \sum_k (-q)^{-k}T_{\bm{j}_0}(e_2^{(p-k)})e_1^{(p)}T_{\bm{j}_0}(e_2^{(k)})v,
 \]
 and since $f_1T_{\bm{j}_0}(e_2^{(k)})v=0$, $e_1^{(p)}T_{\bm{j}_0}(e_2^{(k)})v=0$ holds unless $k=p$.
 Now the second equality is proved similarly as above.
 The proofs of the assertions~\ref{en1-1}--\ref{en1-3} are similar.

 The assertion~\ref{en2-5} is proved as follows: By Lemma~\ref{LemA:2}~\ref{en0-2} and~\ref{en0-5}, we have
 \begin{align*}
  e_1T_{\bm{i}_0}(e_2^{(p)}) &= T_{\bm{i}_0}(e_1e_2^{(p)})=T_{\bm{i}_0}(e_2^{(p-1)}T_1(e_2)+q^{-p}e_2^{(p)}e_1)\\
  &=T_{\bm{i}_0}(e_2^{(p-1)})T_{\bm{i}_01}(e_2)+q^{-p}T_{\bm{i}_0}(e_2^{(p)})e_1.
 \end{align*}
 The assertions~\ref{en2-1}--\ref{en2-4} are proved similarly.

 The assertion~\ref{en3-1} is proved as follows: By Lemma~\ref{LemA:2}~\ref{en0-5}, we have
 \[
 T_{\bm{j}}(e_1)e_1^{(p)}=T_{\bm{j}_0}\left(T_2(e_1)e_1^{(p)}\right)=q^{p}T_{\bm{j}_0}\left(e_1^{(p)}T_2(e_1)\right) = q^{p}e_1^{(p)}T_{\bm{j}}(e_1).
 \]

 The assertion~\ref{en4-1} is proved as follows: Since $s_2s_1(\ga_2)=\ga_1$, from  Lemma~\ref{LemA:2}~\ref{en0-4}, it follows that
 \begin{align*} e_1T_{\bm{j}1}(e_0)T_{\bm{i}1}(e_0)&=T_{\bm{j}1}\left(e_2e_0\right)T_{\bm{i}1}(e_0)=T_{\bm{j}1}(e_0)e_1T_{\bm{i}1}(e_0)\\
  &=T_{\bm{j}1}(e_0)T_{\bm{i}1}(e_2e_0)=T_{\bm{j}1}(e_0)T_{\bm{i}1}(e_0)e_1.
 \end{align*}

 Finally let us show the assertion~\ref{en1-5}. 
 As above, setting $\ell = \max\{m \in \Z_{\geq 0}\mid e_2^{(m)}v \neq 0\}$, we may assume that $v = v_{-\ell\gL_2}$,
 and the first equality is proved similarly.
 To prove the other one, note first that $\wt_P\big(T_{\bm{j}\bm{i}_0}(e_2^{(p)})\big) =p\theta_1$, and
 \begin{equation}\label{equation:dimension}
  \dim V_{I_{01}}(-\ell\gL_2)_{-\ell\gL_2+p\theta_1} = \begin{cases} 1 & (0 \leq p \leq \ell), \\ 0 & (p>\ell),\end{cases} \tag{A.1}
 \end{equation}
 which is proved by taking the classical limit and applying the Poincar\'{e}--Birkhoff--Witt theorem.
 Moreover, since
 \[ T_{\bm{j}\bm{i}_0}(e_2^{(p)})v=\wti{T}_{\bm{j}\bm{i}_0}(e_2^{(p)}\wti{T}_{\bm{j}\bm{i}_0}^{-1}(v)) \ \ \text{and} \ \ 
    \langle h_2, \wt_P\wti{T}_{\bm{j}\bm{i}_0}^{-1}(v)\rangle = \langle h_{\theta_1}, -\ell\gL_2\rangle = -\ell,
 \]
 we have $T_{\bm{j}\bm{i}_0}(e_2^{(p)})v \neq 0$ if and only if $0\leq p \leq \ell$.
 Hence for each $1\leq p \leq \ell$ there is some nonzero $a_p \in \Q(q_s)$ such that $a_pT_{\bm{j}\bm{i}_0}(e_2^{(p)})v = E_{\bm{i}}^{(p)}v$,
 and $T_{\bm{j}\bm{i}_0}(e_2^{(p)})v = E_{\bm{i}}^{(p)}v=0$ if $p>\ell$.
 It remains to prove that $a_p = a_1^p$, which we show by the induction on $p$. The case $p=1$ is trivial.
 Assume that $p>1$. 
 By Lemma~\ref{Lem:1} and weight considerations, we see that $e_iE_{\bm{i}}^{(p-1)}v=0$ for $i \in I_{0}\setminus\{1,2\}$,
 and hence it follows from the induction hypothesis that
 \[ T_{\bm{j}\bm{i}_0}(e_2^{(p)})v = a_1^{-p+1}[p]^{-1}T_{\bm{j}\bm{i}_0}(e_2)E_{\bm{i}}^{(p-1)}v=a_1^{-p}[p]^{-1}E_{\bm{i}}^{(1)}E_{\bm{i}}^{(p-1)}v
 \]
 (note that $E_{\bm{i}}^{(p)}\neq \Big(E_{\bm{i}}^{(1)}\Big)^{(p)}$ by our convention~\eqref{eq:epEp}).
 Hence it suffices to show that $E_{\bm{i}}^{(1)}E_{\bm{i}}^{(p-1)}v=[p]E_{\bm{i}}^{(p)}v$.
 It is proved by a direct calculation that
 \[ f_{i_1}^{(c_\fg p)}\cdots f_{i_L}^{(c_\fg p)} E_{\bm{i}}^{(1)}E_{\bm{i}}^{(p-1)}v = e_2e_2^{(p-1)}v=[p] e_2^{(p)}v
    =[p]f_{i_1}^{(c_\fg p)}\cdots f_{i_L}^{(c_\fg p)}E_{\bm{i}}^{(p)}v,
 \]
 which implies $E_{\bm{i}}^{(1)}E_{\bm{i}}^{(p-1)}v=[p]E_{\bm{i}}^{(p)}v$ by~\eqref{equation:dimension}.
 The proof of~\ref{en1-5} is complete.
\end{proof}

\begin{LemA}\label{LemA:hard_formula}
 For any $\ell \in \Z_{>0}$ and $(p_1,p_2,p_3) \in \Z_{\geq 0}^3$, we have
 \[ e_1E_{\bm{j}}^{(p_3)}E_{\bm{i}}^{(p_2)}e_{10}^{(p_1)}w_\ell = E_{\bm{j}}^{(p_3-1)}E_{\bm{i}}^{(p_2-1)}e_{10}^{(p_1-1)}E^{\bm{a}}w_\ell.
 \]
\end{LemA}

\begin{proof}
 If $p_1<p_2$, the left-hand side is $0$ by~\eqref{eq:commutation3}, and so is the right-hand side since
 \[ e_{i_1}e_2E^{\bm{a}}w_\ell \in W^\ell_{\ell\varpi_2+ \ga_{i_1}}=0.
 \]
 Hence we may assume that $p_1\geq p_2$.
 Set
 \[ w=e_{10}^{(p_1)}w_\ell, \ \ \text{and} \ \ w'=E_{\bm{i}}^{(p_2)}e_{10}^{(p_1)}w_\ell.
 \]
 We have 
 \[ e_iw'=0 \ \ \text{for} \ i\in \{1\}\sqcup J \setminus \{2\} \ \ \ \text{and} \ \ \ e_iE_{\bm{j}}^{(1)}w'=0 \ \ \text{for} \ i \in J 
 \]
 by Lemma~\ref{Lem:1} and~\eqref{eq:commutation2}, and therefore we have the following;
 \begin{align*}
  e_1E_{\bm{j}}^{(p_3)}E_{\bm{i}}^{(p_2)}e_{10}^{(p_1)}w_\ell&=e_1E_{\bm{j}}^{(p_3)}w'\stackrel{\text{\ref{en1-1}}}{=} e_1T_{\bm{j}_0}(e_2^{(p_3)})w'\\
   &\stackrel{\text{\ref{en2-1}}}{=} 
     T_{\bm{j}_0}(e_2^{(p_3-1)})T_{\bm{j}_01}(e_2)w'\stackrel{\text{\ref{en1-2} \ref{en1-1}}}{=}E_{\bm{j}}^{(p_3-1)}e_1E_{\bm{j}}w',
 \end{align*}
 where a number over an equality indicates which assertion of Lemma~\ref{LemA:many} is used there.
 Since $e_iw = 0$ for $i \in I_0\setminus \{2\}$ and $e_1e_2w=0$, we have the following;
 \begin{align*}
  e_1E_{\bm{j}}w'&\stackrel{\text{\ref{en1-2}}}{=}
  -qT_{\bm{j}}(e_1)E_{\bm{i}}^{(p_2)}w\stackrel{\text{\ref{en1-5}}}{=}-qa^{p_2}T_{\bm{j}}(e_1)T_{\bm{j}\bm{i}_0}(e_2^{(p_2)})w\\
  &\stackrel{\text{\ref{en2-5}}}{=} -qa^{p_2}T_{\bm{j}}\left(T_{\bm{i}_0}(e_2^{(p_2-1)})T_{\bm{i}_01}(e_2)+ q^{-p_2}T_{\bm{i}_0}(e_2^{(p)})e_1\right)w\\
  &\stackrel{\text{\ref{en0} \ref{en1-4}}}{=}-qa^{p_2}T_{\bm{j}\bm{i}_0}(e_2^{(p_2-1)})T_{\bm{j}}(e_1)T_{\bm{j}\bm{i}_0}(e_2)w\\
  &\stackrel{\text{\ref{en1-2} \ref{en1-5}}}{=} E_{\bm{i}}^{(p_2-1)} e_1E_{\bm{j}}^{(1)}E_{\bm{i}}^{(1)}w.
 \end{align*}
 Finally, we have
 \begin{align*}
  e_1E_{\bm{j}}^{(1)}E_{\bm{i}}^{(1)}w&=e_1E_{\bm{j}}^{(1)}E_{\bm{i}}^{(1)}e_{10}^{(p_1)}w_\ell\stackrel{\text{\ref{en1-2} \ref{en1-5}}}{=} -qT_{\bm{j}}(e_1)T_{\bm{i}_0}(e_2)e_{10}^{(p_1)}
   w_\ell\\
  &\stackrel{\text{\ref{en2-2} \ref{en2-3}}}{=} -qT_{\bm{j}}(e_1)e_{10}^{(p_1-1)}T_{\bm{i}1}(e_0)w_\ell\\
  &\stackrel{\text{\ref{en3-1}}}{=} 
  -q^{p_1}e_1^{(p_1-1)}T_{\bm{j}}(e_1)e_0^{(p_1-1)}T_{\bm{i}1}(e_0)w_\ell\\
  &\stackrel{\text{\ref{en2-4}}}{=}  -q^{p_1}e_1^{(p_1-1)}\left(e_0^{(p_1-2)}T_{\bm{j}1}(e_0)+q^{-p_1+1}e_0^{(p_1-1)}T_{\bm{j}}(e_1)\right)T_{\bm{i}1}(e_0)w_\ell\\
  &\stackrel{\text{\ref{en4-1}}}{=} -qe_{10}^{(p_1-1)}T_{\bm{j}}(e_1)T_{\bm{i}1}(e_0)w_\ell\\
  &\stackrel{\text{\ref{en1-2} \ref{en1-3}}}{=} e_{10}^{(p_1-1)}E^{\bm{a}}w_\ell.
 \end{align*}
 The assertion is proved.
\end{proof}

\noindent\textit{Proof of Proposition~\ref{Prop:A}.}\
By~\eqref{eq:commutation2} and Lemma~\ref{LemA:hard_formula}, we have
\begin{align*}
 e_2E^{\bm{p}}w_\ell &= \left(e_1^{(p_5-1)}e_2^{(p_4+1)}e_1E_{\bm{j}}^{(p_3)}E_{\bm{i}}^{(p_2)}e_{10}^{(p_1)}+[p_4-p_5+1]E^{\bm{p}+\bm{\gee}_4}\right)w_\ell\\
 &= \left(E^{\bm{p}-\bm{a}+\bm{\gee}_4}E^{\bm{a}}+[p_4-p_5+1]E^{\bm{p}+\bm{\gee}_4}\right)w_\ell.
\end{align*}
Hence it suffices to show that $E^{\bm{a}}w_\ell = c_\ell f_2w_\ell$ holds for some $c_\ell \in \pm 1 + q_sA$.
We see from Proposition~\ref{Prop:critical_3} (C1) that $\dim W^{\ell}_{\ell\varpi_2-\ga_2} = 1$,
and hence we have $E^{\bm{a}}w_\ell = c_\ell f_2w_\ell$ for some $c_\ell \in \Q(q_s)$.
Now $c_\ell \in \pm 1 +q_sA$ follows since both $\normsq{E^{\bm{a}}w_\ell}$ and $\normsq{f_2w_\ell}$ belong to $1+q_sA$.
The proof is complete. \qed

\newcommand{\etalchar}[1]{$^{#1}$}
\def\cprime{$'$} \def\cprime{$'$} \def\cprime{$'$} \def\cprime{$'$}


\end{document}